%% file: sheafTrop.tex
\documentclass[12pt]{amsart}
\pdfoutput=1 
\usepackage{xcolor}
\definecolor{cadmiumgreen}{rgb}{0.0, 0.42, 0.24}
\usepackage[
colorlinks, citecolor=cadmiumgreen,
pagebackref,
pdfauthor={Andreas Gross, Farbod Shokrieh}, 
pdftitle={A sheaf-theoretic approach to tropical homology},
pdfstartview ={FitV},
]{hyperref}

\usepackage[
alphabetic,
backrefs,
msc-links,
nobysame,
lite,
]{amsrefs}

\usepackage[left=3.5cm,top=3.5cm,right=3.5cm]{geometry}

\usepackage[T1]{fontenc}
\usepackage[utf8]{inputenc}

\usepackage{mathptmx} 
\usepackage[scaled]{helvet}
\usepackage{courier}
\usepackage{microtype}
\usepackage{amssymb,amsthm, amsmath,mathtools,calligra,mathrsfs}	
\usepackage{dsfont}						
\usepackage{tikz}							
\usetikzlibrary[matrix,arrows,calc]
\usepackage{enumitem}

 \theoremstyle{plain}
 \newtheorem{thm}{Theorem}[section]
  \theoremstyle{definition}
  \newtheorem{defn}[thm]{Definition}
  \theoremstyle{definition}
  \newtheorem{example}[thm]{Example}
  \theoremstyle{plain}
  \newtheorem{lemma}[thm]{Lemma}
  \theoremstyle{plain}
  \newtheorem{cor}[thm]{Corollary}
  \theoremstyle{remark}
  \newtheorem{rem}[thm]{Remark}
  \theoremstyle{plain}
  \newtheorem{prop}[thm]{Proposition}
  \theoremstyle{plain}
  
  \theoremstyle{remark}
  
  \theoremstyle{remark}
  
  \theoremstyle{plain}
  
  \theoremstyle{definition}
  
  \numberwithin{equation}{section}

\DeclareMathOperator{\Hom}{Hom}
\DeclareMathOperator{\Tor}{Tor}
\DeclareMathOperator{\Ext}{Ext}
\DeclareMathOperator{\Homs}{\mathscr{H}\text{\kern -3pt {\calligra\large om}}\,}
\DeclareMathOperator{\Divs}{\mathscr{D}\text{\kern -3pt {\calligra\large iv}}\,}
\DeclareMathOperator{\Rat}{\mathscr{M}}

\DeclareMathOperator{\relint}{relint}

\DeclareMathOperator{\supp}{supp}

\DeclareMathOperator{\divv}{div}
\DeclareMathOperator{\Div}{Div}

\DeclareMathOperator{\conv}{conv}

\DeclareMathOperator{\id}{id}

\DeclareMathOperator{\cyc}{cyc}
\newcommand\tcyc{\widetilde{\mathrm{cyc}}}

\newcommand{\Aff}{\operatorname{Aff}}

\newcommand{\LC}{\operatorname{LC}}
\newcommand{\dual}{{\mathds D}}
\newcommand{\vdual}{{\mathscr D}}
\newcommand{\SoftDual}{{\mathscr S}}
\newcommand{\berg}[1]{L_{#1}}
\newcommand{\sing}{{\mathrm{sing}}}

\newcommand{\R}{{\mathds R}}
\newcommand{\Rbar}{{\overline\R}}

\newcommand{\Hyper}{{\mathds H}}

\newcommand{\Z}{{\mathds Z}}
\newcommand{\N}{{\mathds N}}

\newcommand{\mF}{{\mathcal F}}
\newcommand{\mG}{{\mathcal G}}
\newcommand{\mL}{{\mathcal L}}

\newcommand{\mS}{{\mathcal S}}

\newcommand{\scrH}{{\mathscr H}}

\renewcommand{\injlim}{\varinjlim}
\renewcommand\phi\varphi

\newcommand{\der}[1]{\ensuremath{D(\Z_{#1})}}

\begin{document}

\title{A sheaf-theoretic approach to tropical homology}

\author{Andreas Gross}
\email{\href{mailto:andreas.gross@colostate.edu}{andreas.gross@colostate.edu}}

\author{Farbod Shokrieh}
\email{\href{mailto:farbod@math.ku.dk}{farbod@math.ku.dk}, \href{mailto:farbod@uw.edu}{farbod@uw.edu}}

\date{\today}

\subjclass[2010]{
\href{https://mathscinet.ams.org/msc/msc2010.html?t=14T05}{14T05},
\href{https://mathscinet.ams.org/msc/msc2010.html?t=14F05}{14F05},
\href{https://mathscinet.ams.org/msc/msc2010.html?t=14C17}{14C17},
\href{https://mathscinet.ams.org/msc/msc2010.html?t=52B40}{52B40},
\href{https://mathscinet.ams.org/msc/msc2010.html?t=14C25}{14C25}}

\begin{abstract} 
We introduce a sheaf-theoretic approach to tropical homology, especially for tropical homology with potentially non-compact supports. Our setup is suited to study the functorial properties of tropical homology, and we show that it behaves analogously to classical Borel-Moore homology in the sense that there are proper push-forwards, cross products, and cup products with tropical cohomology classes, and that it satisfies identities like the projection formula and the K\"unneth theorem. Our framework allows for a natural definition of the tropical cycle class map, which we show to be a natural transformation. Finally, we prove Poincar\'e-Verdier duality over the integers on tropical manifolds.
\end{abstract}

\maketitle

\setcounter{tocdepth}{1}
\tableofcontents

\section{Introduction}
\renewcommand*{\thethm}{\Alph{thm}}

\subsection{Background}
Tropical (co)homology theory is a new tool to associate algebraic invariants to the spaces appearing in tropical geometry. They were introduced in \cite{TropHomology} where it was shown that the tropical cohomology groups of tropical manifolds have a Hodge-theoretic interpretation in algebraic geometry in case the tropical manifold arises a the tropicalization of a smooth projective variety. As one would expect by analogy to the algebro-geometric picture, tropical homology is also closely related to tropical intersection theory. In \cite{MZeigenwave}, Mikhalkin and Zharkov introduced the tropical cycle class map on rational polyhedral spaces equipped with that assigns a class in tropical homology to every tropical cycle. This map has been further studied in \cite{TropicalSurfaces} in the case of tropical surfaces, and in \cite{Lefschetz} with a special emphasis on locally finite tropical cycles with integer coefficients. An excellent introduction to the subject can be found in \cite{TropIntro}.

\subsection{Our contributions}

We introduce a  sheaf-theoretic viewpoint on tropical homology which reduces constructions in tropical homology, which does not have a direct counterpart in algebraic geometry, to well-known constructions in sheaf theory. The basis for this is the following theorem that we prove in \S\ref{subsec:Comparison of homologies}:

\begin{thm}[= Theorem \ref{thm:compatibility of homology theories}]
\label{thm:introcomparison}
Let $X$ be a rational polyhedral space. Then there exists a natural isomorphism
\[
H^{lf}_{p,q}(X) \cong H^{-q}R\Hom^\bullet(\Omega_X^p,\dual_X) \ ,
\]
where $H^{lf}_{p,q}(X)$ denotes the tropical homology groups defined via locally finite tropical chains over $\Z$, $\Omega_X^p$ denotes the sheaf of tropical $p$-forms, and $\dual_X$ denotes the dualizing complex on $X$.
\end{thm}

Motivated by this, and in analogy to the classical theory, we denote
\[
H^{BM}_{p,q}(X)=H^{-q}R\Hom^\bullet(\Omega_X^p,\dual_X)
\]
and call it the $(p,q)$-th tropical Borel-Moore homology group. This formulation of tropical homology makes it  evident that the functorial behavior of tropical homology is completely determined by the functorial behavior of tropical $p$-forms on the one hand, and dualizing complexes on the other. With this in mind, the constructions of proper push-forwards, cross products, cup products, and cap products in tropical homology are straightforward generalizations of the classical constructions, at least after we establish some general functorial properties of the sheaves of tropical forms. Our point of view also sheds light on the identities that these operations satisfy. For example, we obtain a tropical version of the K\"unneth theorem:

\begin{thm}[= Theorem \ref{thm:Kuenneth}]
\label{thm:introkuenneth}
Let $X$ and $Y$ be rational polyhedral spaces such that all tropical homology groups of $X$ and $Y$ are finitely generated and torsion free. Then we have
\[
H^{BM}_{p,q}(X\times Y)\cong \bigoplus_{i+j=p\atop k+l=q} H^{BM}_{i,k}(X)\otimes_\Z H^{BM}_{j,l}(Y) \ .
\]
\end{thm}

Another advantage of a sheaf-theoretic view on tropical geometry is that sheaves are very well-suited to pass from local to global considerations. We exploit this in our definition of the tropical cycle class map, where we use the the definition of \cite{MZeigenwave} locally and then utilize the sheaf property to glue. This has the advantage of avoiding the necessity of dealing with global face structures or triangulations of the space, as one needs to in the definitions in \cite{MZeigenwave,Lefschetz}. Once the tropical cycle class map is defined, we prove that a natural transformation in the sense that it satisfies the compatibility conditions summarized in the following theorem:

\begin{thm}[= Corollary \ref{cor:cycle map commutes with push-forwards}, Proposition \ref{prop:cycle map commutes with cross products}, Proposition \ref{prop:cycle map respects intersections with chern classes}]
\label{thm:introfunctoriality}
The tropical cycle class map commutes with proper push-forwards, cross products, and intersections with tropical Cartier divisors.
\end{thm}

Finally, we restrict our attention to tropical manifolds. As tropical manifolds shall serve as the tropical counterpart of topological manifolds or smooth algebraic varieties, their homology should satisfy a form of Poincar\'e duality. Our sheaf-theoretic point of view allows us to use the well-established machinery of Verdier duality and makes it evident that in order to get a duality statement, one has to describe the Verdier dual $\vdual(\Omega_X^p)$ of the sheaf of tropical $p$-forms $\Omega_X^p$.

\begin{thm}[= Theorem \ref{thm:computing the Verdier dual}]
\label{thm:introduality}
Let $X$ be a purely $n$-dimensional tropical manifold. Then there is a natural isomorphism 
\begin{equation*}
\Omega^{n-p}_X[n]\cong \vdual(\Omega^p_X) \ .
\end{equation*} 
In particular, we have 
\[
H^{BM}_{p,q}(X) \cong H^{n-p,n-q}(X) \ .
\]
\end{thm}

\subsection{Related work}

If one works with real (instead of integer) coefficients, there is an isomorphism between tropical cohomology and the Dolbeault cohomology of Lagerberg's superforms, established by Jell, Shaw, and Smacka \cite{Superforms}. The main ingredient for this isomorphism is the fact that an appropriate complex of superforms defines a soft resolution of $\Omega^p_X\otimes_\Z \R$. The authors of \cite{Superforms} then proceed to use the fact that top-dimensional compactly supported superforms can be integrated to prove a Poincar\'e duality theorem on tropical manifolds for tropical cohomology with real coefficients. In the course of the proof, they implicitly prove that $\Omega^p_X\otimes_{\Z_X}\R_X$ is Verdier dual to $\Omega^{n-p}_X\otimes_{\Z_X}\R_X [n]$. Smacka elaborated on this in his thesis \cite{SmackaThesis}, where he also proves a K\"unneth theorem for tropical cohomology with real coefficients, again by using superforms. We should note that the superform approach does not work in the case of integer coefficients so that our computation of $\vdual(\Omega^{p}_X)$ has to rely on different techniques.

 To prove Poincar\'e duality for tropical (co)homology over the integers, one does not necessarily have to use Verdier duality. This has been demonstrated in Jell, Rau, and Shaw's work  \cite{Lefschetz}, where they prove Poincar\'e duality over the integers using an Mayer-Vietoris argument. To do so, they use a well-behaved open cover of the tropical manifold, which they obtain by assuming that it comes equipped with a global face structure. We strengthen their result by eliminating the need for this additional structure.
 
We are hopeful that our sheaf-theoretic approach to tropical homology can be applied to spaces that are not necessarily rational polyhedral, but possibly have singularities in their affine structures. The resulting notion should be strongly related to the invariants of integral affine manifolds appearing in Ruddat's work \cite{Ruddat} in the context of mirror symmetry. 

\subsection{Structure of the paper}

In  sections \ref{sec:tropspace} and \ref{sec:tropical cycles} we recall the definitions of the objects and operations needed in the main part of the paper. We try to follow the literature \cite{MZeigenwave,TropHomology, Lefschetz, ShawIntersection, AllerRau ,TropIntro} as closely as possible, but will provide a new perspective on some things.
Most notably, we deviate from the literature in our definition of tropical forms in section \ref{sec:tropspace}, and our treatment of tropical cycles in section \ref{sec:tropical cycles} has an emphasis on working locally and highlighting their functorial properties.

In section \ref{sec:tropical (co)homology} we introduce tropical Borel-Moore homology and study its functorial behavior. We define proper push-forwards, cross products, cup products, and cap products, and will prove Theorem \ref{thm:introkuenneth}. Finally, we compare our theory with the on obtained via locally finite tropical chains by proving Theorem \ref{thm:introcomparison}.

In  section \ref{sec:cycle class map} we will define the tropical cycle class map and show that it is compatible with proper push-forwards, cross-products, and intersections with Cartier divisors, proving Theorem \ref{thm:introfunctoriality}. Furthermore, we show that our tropical cycle class map coincides with the one introduced in \cite{MZeigenwave, Lefschetz} if we are given a (global) face structure.

 Section \ref{sec:duality} is devoted to  prove Theorem \ref{thm:introduality}. To do so, we use the local nature of the statement, allowing us to use tropical modifications and reducing it to the case where $X=\Rbar^n$.
 
The main ingredients of our proof of Theorem \ref{thm:introcomparison} are of an entirely topological nature, dealing mostly with certain sheaves of singular chains on conically stratified spaces.  As these results are of a very different flavor, and potentially of independent interest, we decided to put them in Appendix \ref{sec:appendix}.

\subsection*{Acknowledgement} 
We would like to thank Philipp Jell, Amit Patel, Johannes Rau, and Kristin Shaw for helpful discussions and conversations. AG was supported by the ERC Starting Grant MOTZETA (project 306610) of the European Research Council (PI: Johannes Nicaise) during parts of this project.
FS was supported by the Danish National Research Foundation through the Centre for Symmetry and Deformation (DNRF92) during parts of this project.

\subsection*{Conventions}
The natural numbers $\N$ include $0$. \emph{All} homology and cohomology groups in this paper, whether classical or tropical, will be considered with \emph{integer coefficients}.

\section{Rational polyhedral spaces and tropical forms} 
\label{sec:tropspace}
\renewcommand*{\thethm}{\arabic{section}.\arabic{thm}}

\subsection{Rational polyhedral spaces}

We denote $\Rbar\coloneqq \R\cup \{+\infty\}$ and consider it with the order topology. For $n\in \N$, the $n$-fold product $\Rbar^n$ has a natural stratification $\Rbar^n= \bigsqcup_{I\subseteq \{1,\ldots, n\}} \R_I$, where the stratum
\begin{equation*}
\R_I =\{ (x_i)_{1\leq i\leq n} \mid x_i=\infty \text{ if and only iff } i\in I\} 
\end{equation*}
is naturally identified with $\R^{n-|I|}$.
Recall that a \emph{(rational) polyhedron} in $\R^n$ is a finite intersection of half-spaces of the form $\{x\in \R^n\mid \langle m,x\rangle\leq a\}$ with $m\in (\Z^n)^*$ and $a\in \R$, where $\langle \cdot, \cdot\rangle$ denotes the evaluation pairing. By a polyhedron in $\Rbar^n$ we mean any set occurring as the closure of a polyhedron in some stratum $\R_I$. Note that for any polyhedron $\sigma$ in $\R^n$ and subset $I\subset \{1,\ldots , n\}$, the intersection $\overline\sigma\cap \R_I$  is a polyhedron in $\R_I$. A {\em polyhedral set} in $\Rbar^n$ is a finite union of polyhedra. 

An \emph{integral affine linear function} on a subset $X\subseteq \Rbar^n$ is a continuous function $f\colon X\to \R$ that is of the form $x\mapsto \langle m, x\rangle + a$ for some $m\in (\Z^n)^*$ and $a\in \R$ locally around every point in $X$. Here, we use the convention that $0\cdot (\infty) =0$, and that $\langle m, x\rangle $ is only defined if for all $i$ such that the coordinate $m_i$ is nonzero, we have $x_i\neq \infty$. In particular, if $f$ is integral affine linear on $X$ and $f(x)= \langle m, x\rangle + a$ for $x\in X\cap \R_I$, then $m_i=0$ for $i\in I$. For every subset $X\subseteq \Rbar^n$, the integral affine linear functions on open subsets of $X$ define a sheaf of abelian groups on $X$, denoted by $\Aff_X$.

\begin{example}
Consider the polyhedral set
\[
X= \partial \conv\{(0,0),(1,0),(0,1),(1,1)\} 
\]
in $\R^2$, which is the boundary of a square with sides of length one. Consider the function on $X$ that is given by $0$ on the top and right edge of the square, and by $-1+ x_1+x_2$ on the bottom and left edge of the square. This function is continuous and locally the restriction of an integral affine linear function on $\R^2$. In fact, it coincides with the restriction of an integral affine linear function on the union of any two adjacent edges of the square. However, it is not equal to the restriction of a single integral affine linear function on $\R^2$ everywhere because it has different slopes on parallel edges.
\end{example}

\begin{defn}
A \emph{rational polyhedral space} is a second-countable Hausdorff topological space $X$, together with a sheaf $\Aff_X$ of continuous functions such that for every $x\in X$ there exists an open neighborhood $U\subseteq X$, an open subset $V$ of a polyhedral subset of $\Rbar^n$ for some $n\in \N$, and a homeomorphism $\phi\colon U\to V$ that induces an isomorphism $\phi^{-1}\Aff_V\to \Aff_U$ via the pullback of functions. The data of $U,V$ and $\phi$ is called a \emph{chart}. 
\end{defn}

\begin{defn}
A morphism between two rational polyhedral spaces $X$ and $Y$ is a continuous map $f\colon X\to Y$ that induces a morphism $f^{-1}\Aff_Y\to \Aff_X$ via the pullback of functions. A morphism is \emph{proper}, if it is proper as a continuous map of topological spaces, that is if the preimages of compact subsets of $Y$ are compact. 
\end{defn}

\begin{defn}
Let $X$ be a rational polyhedral space. 
\begin{enumerate}[label=(\alph*)]
\item A {\em polyhedron} in $X$ is a closed subset $P\subseteq X$ such that there exists a chart $X\supseteq U \xrightarrow{\phi} V\subseteq \Rbar^n$ such that $P\subset U$ and $\phi(P)\subseteq\Rbar^n$ is a polyhedron. The \emph{faces} of $P$ are the preimages under $\phi$ of the (finite or infinite) faces of $\phi(P)$. The \emph{relative interior} $\relint(P)$ of $P$ is the complement in $P$ of the union of its proper faces.
\item A \emph{local face structure} at a point $x\in X$ is a finite set $\Sigma$ of polyhedra in $X$ that is closed under taking faces and intersections (that is if $\tau$ is a face of $\sigma\in\Sigma$, then $\tau\in \Sigma$, and $\sigma\cap\delta\in\Sigma$ for all $\sigma,\delta\in\Sigma$ with nonempty intersection), such that $x$ is contained in the (topological) interior of $|\Sigma|=\bigcup_{\sigma\in \Sigma}\sigma$, there exists a chart $X\supseteq U\to V\subseteq \Rbar^n$ with $|\Sigma|\subseteq U$, and such that $x\in\sigma$ for all inclusion-maximal $\sigma\in \Sigma$. 
\item A (global) {\em face structure} on $X$ is a set $\Sigma$ of polyhedra in $X$ that is closed under taking faces and intersections such that $X=\bigcup_{\sigma\in\Sigma}\sigma$, and for every $x\in X$ the set of all faces of polyhedra in $\Sigma$ that contain $x$ is a local face structure at $x$.
\item We say that a closed subset $S\subseteq X$ is \emph{locally polyhedral} if at every point $x\in X$ there is a local face structure $\Sigma$ and a subset $\Sigma'\subseteq \Sigma$ such that $S\cap |\Sigma|=\bigcup_{\sigma\in\Sigma'}\sigma$.  
\end{enumerate}
\end{defn}

\subsection{Tangent spaces}

As constant functions are integral affine linear, there is an inclusion $\R_X \hookrightarrow \Aff_X$, where $\R_X$ denotes the constant sheaf associated to $\R$. Following \cite{MZjacobians}, we denote the quotient sheaf $\Aff_X/\R_X$ by $\Omega_X^1$ and call it the \emph{cotangent sheaf}. The sections of $\Omega_X^1$ are called \emph{tropical $1$-forms}. The reason for this is that the cotangent space at a point should consist of linear approximations of functions, and linear functions are simply affine linear ones modulo constants. For $x\in X$ we denote  by
\begin{align*}
T^\Z_x X &\coloneqq \Hom_\Z(\Omega^1_{X,x},\Z) \  \text{ and} \\
T_x X &\coloneqq \Hom_\Z(\Omega^1_{X,x},\R) 
\end{align*}
the (integral) tangent space of $X$ at $x$. It follows immediately from the definitions that a morphism $f\colon X\to Y$ of rational polyhedral spaces induces a morphism 
\[
f^\sharp\colon f^{-1}\Omega^1_Y \to \Omega^1_X \ ,
\]
and hence morphisms of stalks $\Omega^1_{Y,f(x)}\to \Omega^1_{X,x}$ for all $x\in X$. These dualize to a morphisms
\[
d_xf\colon T^\Z_x X\to T^\Z_{f(x)} Y
\]
between the integral tangent spaces. If $Y=\R$, that is if $f$ is an affine function on $X$, the germ at $x$ of the image of $f$ in $\Gamma(X,\Omega^1_X)$ under the quotient morphism $\Aff_X\to\Omega^1_X$ defines a morphism $T^\Z_xX \to \Z $ which coincides with $d_x f$ modulo the natural identification $T^\Z_{f(x)}\R\cong\Z$. For this reason, we use the notation $df$ for the image of $f$ in $\Gamma(X,\Omega^1_X)$.

Unfortunately, there is no known interpretation of $T_x^\Z X$ or $T_xX$ as the set of equivalence classes of ``smooth'' paths through $x$ as in differential geometry. There is, however, an interpretation of a subset $T_x^\Z X$ as germs of functions $(\R_{\geq 0},0)\to (X,x)$. Recall that such a germ is a morphism $[0,\epsilon)\to X$ for some $\epsilon>0$ that sends $0$ to $x$, up the equivalence relation that allows to shrink the interval, i.e. restricting to $[0,\epsilon')$ for $\epsilon'<\epsilon$ does not change the germ. To every germ $\gamma\colon(\R_{\geq 0},0)\to (X,x)$ we can associate the tangent vector $d_x\gamma(1)\in T^\Z_x X$, where we identify $T_0^\Z(\R_{\geq 0})$ with $\Z$ in the natural way. In fact, since affine linear functions on $\R_{\geq 0}$ are completely determined by the value and slope at $0$, the germ $\gamma$ is uniquely determined by $d_x\gamma(1)$.  We define the \emph{local cone} of $X$ at $x$ as the subset of $T_x X$ given by
\[
\LC_x X\coloneqq \{\lambda\cdot d_0\gamma(1) \mid \lambda\in \R_{\geq 0},\, \gamma\colon (\R_{\geq 0},0) \to (X,x) \text{ a germ}\}
\]

\begin{prop}
\label{prop:local cone}
Let $X$ be a rational polyhedral space, and let $x\in X$. Then $\LC_x X$ is a rational polyhedral subspace of $T_x X$ with tangent space
\[
T_0(\LC_x X)= T_x X
\]
at the origin. Furthermore, there exists a unique morphism of germs	
\[
(\LC_x X, 0)\to (X,x)
\]
such that the induced map
$T_x X=T_0(\LC_x X)\to T_x X$ is the identity.
\end{prop}
\begin{proof}
Since the definition of the local cone is local, we may assume that $X$ is a polyhedral subset of $\Rbar^n$ for some $n\in \N$, and $x=(x_1,\ldots, x_n)$. After a change of coordinates, we may further assume that there exists a $0\leq k\leq n$ such such $x\in \{\infty\}^k\times \R^{n-k}$. For every connected open subset $Y$ of a polyhedral set in $\R^n$, every morphism $Y\to \Rbar^n$ whose image contains $x$ has to map entirely into $\{\infty\}^k\times \R^{n-k}$. This applies in particular to open neighborhoods of $0$ in $\R_{\geq 0}$ or $\LC_x X$, so both the local cone and the set of germs of morphisms $(\LC_x X,0)\to (X,x)$ only depend on $X\cap (\{\infty\}^k\times \R^{n-k})$. The affine functions defined on a neighborhood of $x$ are, after potentially shrinking the neighborhood, precisely those that are pullbacks of affine functions on $\Rbar^{n-k}$ under the projection $X\to \Rbar^{n-k}$ onto the last $n-k$ coordinates. Therefore, the tangent space of $X$ at $x$ only depends on $X\cap (\{\infty\}^k\times \R^{n-k})$ as well. After replacing $X$ by $X\cap (\{\infty\}^k\times \R^{n-k})$, we may thus assume that $x\in \R^{n-k}$. In this case, the local cone at $x$ is easily seen to be equal to the set
\[
\{v\in \R^{n-k} \mid x+[0,\epsilon)v \subseteq X \text{ for some } \epsilon>0 \} \ ,
\]
which is well-known to be a finite union of polyhedral cones, and in particular a polyhedral set. In this case, it is equally well known that $x$ has a neighborhood in $X$ that is isomorphic to a neighborhood of $0$ in $\LC_x X$, so for the last part of the proof we may assume that $x=0$ and $X=\LC_x X$. It follows immediately that $T_x X= T_0(\LC_x X)$ and that the identity map defines a germ of maps $(\LC_x X, 0)\to (X,x)$ inducing the identity on tangent spaces. Since such a germ is determined by the associated map on tangent spaces, this finishes the proof. 
\end{proof}

\begin{cor}
Let $X$ be a rational polyhedral space, and let $x\in X$. Then the local cone $\LC_x X$ spans the tangent space $T_x X$.
\end{cor}

\begin{proof}
Because $\LC_x X$ is invariant under scaling, a linear function on $T_x X$ vanishes on a neighborhood of $0$ if and only if it vanishes on all of $\LC_x X$, which is true if and only if it vanishes on the span of $\LC_x X$. But by definition of the tangent space, a linear function on $T_x X$ vanishes on a neighborhood of $0$ in $\LC_X X$ if and only if it vanishes on $T_0 (\LC_x X)$. By Proposition \ref{prop:local cone}, this shows that $\LC_x X$ spans $T_x X= T_0(\LC_x X)$.
\end{proof}

To define sheaves of tropical $p$-forms, one would like to take the $p$-th exterior power of $\Omega_X^1$. Unfortunately, even for very well-behaved rational polyhedral spaces, $\bigwedge^p\Omega^1_X$ might very well be nontrivial for some $p>\dim(X)$. 
This is remedied with the following definition.
\begin{defn}
For a rational polyhedral space $X$ we denote by $X^{\max}$ the set of points in $X$ that has a neighborhood isomorphic to an open set in $\R^n$. By definition, $X^{\max}$ is an open subset of $X$. Let $\iota\colon X^{\max} \hookrightarrow X$  denote the inclusion. Then one defines the sheaf of graded rings $\Omega^*_X$ as the image of $\bigwedge^* \Omega^1_X$ in $\iota_*\left(\bigwedge^*\Omega^1_X|_{X^{\max}}\right)$. Sections of $\Omega^p_X$ are called \emph{tropical $p$-forms}.
\end{defn}

Note that $\Omega^1_X\to \iota_*\Omega^1_X|_{X^{\max}}$ is a monomorphism, so that the definition of $\Omega^1_X$ is unambiguous. Also note that for $p\in X^{\max}$ the rank of $\Omega^1_{X,p}$ equals the local dimension at $p$. Therefore, $\Omega^p_X=0$ for $p>\dim X$.

\begin{rem}
The sheaf $\Omega^p_X$ of tropical $p$-forms is closely related to the sheaf $\mathcal F^p_X$ considered in \cite{MZeigenwave}: $\Omega^p_X$ can be embedded into $\mathcal F^p_X$ such that $\Omega^p_{X,x}$ has finite index in $\mathcal F^p_{X,x}$ for all $x\in X$. If $X$ is the tropical linear space associated to a loopless matroid $M$, then $\Omega^p_X$ and $\mathcal F^p_X$ are equal and their stalks in the minimal stratum are isomorphic to the projective Orlik-Solomon algebra of the matroid \cite{OrlikSolomon}.
\end{rem}

\begin{example}
Let 
\[
X= \R_{\geq 0} (1,0) \cup  \R_{\geq 0} (0,1) \cup \R_{\geq 0} (-1,-1) \subseteq \R^2
\]
be the standard tropical line, depicted in Figure \ref{fig:tropical line}. Then the stalk $\Omega^1_{X,0}$ of the sheaf of tropical $1$-forms at the origin is isomorphic to the space of integer linear functions on $\R^2$, which we can identify with $\Z^2$. For every $x\in X$ not equal to $0$, the stalk $\Omega^1_{X,x}$ is isomorphic to $\Z$. The set $X^{\max}$ is the complement of the origin, and the restriction $\Omega^1_X|_{X^{\max}}$ is locally free of rank $1$. In particular, $\bigwedge^n \left(\Omega^1_X|_{X^{\max}}\right)= 0$ for all $n>1$, and hence $\Omega^n_X=0$ for $n>1$. On the other hand, we have 
\[
\left( \bigwedge\nolimits^2 \Omega^1_X\right)_0 \cong \bigwedge\nolimits ^2 \Omega^1_{X,0} \cong \bigwedge\nolimits^2 \Z^2 \neq 0 \ .
\] 
Finally, note that an integral linear function on $\R^2$ is completely determined by its slopes in the directions $(1,0)$ and $(0,1)$, and hence the natural morphism $\Omega_{X,0}^1\to (\iota_* \Omega_X^1|_{X^{\max}})_0$, where $\iota\colon X^{\max}\to X$ is the inclusion, is an embedding.
\end{example}

\begin{figure}
\begin{center}
\begin{tikzpicture}[font=\footnotesize]
\begin{scope}
\clip (0,0) circle (2);
\draw (0,0) -- (2,0);
\draw (0,0) -- (0,2);
\draw (0,0) --( -2,-2);
\end{scope}
\node at (0,0) [fill=black,circle, inner sep= .3mm ,pin= 60:{$\Z^2$}]{};
\node at (1,0) [fill=black,circle, inner sep= .3mm ,pin= -80:{$\Z^2/\langle (0,1) \rangle \cong \Z$}]{};
\node at (0,1) [fill=black,circle, inner sep= .3mm ,pin= 180:{$\Z^2/\langle (1,0) \rangle \cong \Z$}]{};
\node at (-1,-1) [fill=black,circle, inner sep= .3mm ,pin= -90:{$\Z^2/\langle (-1,1) \rangle \cong \Z$}]{};
\end{tikzpicture}
\end{center}
\caption{The standard tropical line and the stalks of its sheaf of tropical $1$-forms}
\label{fig:tropical line}
\end{figure}
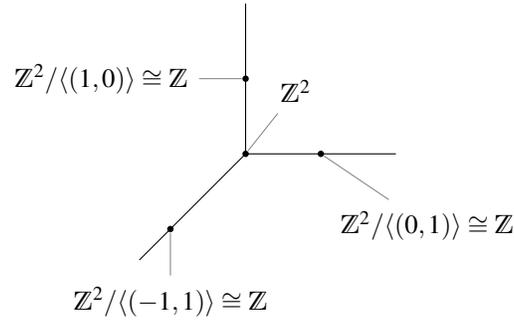

\begin{example}
Let $N$ be the lattice generated by elements $e_0,\ldots,e_3$ subject to the relation $\sum e_i=0$, and let $X$ be the union of the three half-planes in $N\otimes_\Z\R$ given by $H_i=\R e_0 + \R_{\geq 0} e_i$, where $i\in \{1,2,3\}$. Let $e_0^*,e_1^*,e_2^*\in \Hom(N,\Z)$ be the dual basis to $e_0,e_1,e_2$. Then $\bigwedge^2\Omega^1_{X,0}$ is freely generated by $e_0^*\wedge e_1^*$, $e_0^*\wedge e_2^*$, and $e_1^*\wedge e_2^*$. As both $e_1^*$ and $e_2^*$ vanish on $e_0$, the restrictions of $e_1^*\wedge e_2^*$ to the interiors of all three half planes $H_i$ vanish. On the other hand, no linear combination of $e_0^*\wedge e_1^*$ and $e_0^*\wedge e_2^*$ vanishes on the interiors of all three half planes. Since $X^{\max}= \bigcup \mathring H_i$, we conclude that $\Omega^2_{X,0}$ is the quotient of $\bigwedge^2\Omega^1_{X,0}$ by $\Z(e_1^*\wedge e_2^*)$.
\end{example}

Tropical $p$-forms can be pulled back along morphisms, as shown in the following proposition.
\begin{prop}
\label{prop:pullback of p-forms}
Let $f\colon X\to Y$ be a morphism of rational polyhedral spaces. Then the pullback
\[
f^\sharp\colon f^{-1}\Omega^1_Y\to \Omega^1_X
\]
induces a pull-back 
\[
f^{-1}\Omega^*_Y\to \Omega^*_X \ ,
\]
which we again denote by $f^\sharp$.
\end{prop}

\begin{proof}
It is immediate that $f^\sharp$ induces a morphism
\[
\bigwedge f^\sharp\colon f^{-1}\bigwedge \Omega^1_Y \to \bigwedge\Omega^1_X \ .
\]
To see that this induces a morphism on the quotients $f^{-1}\Omega^*_Y \to \Omega^1_X$, we need to show that if $U\subset Y$ is open and we are given a section $\omega\in\Gamma(U,\bigwedge\nolimits^* \Omega^1_Y)$ that restricts to zero on $U\cap Y^{\max}$, then $\bigwedge f^\sharp(\omega)$ vanishes on $f^{-1}U\cap X^{\max}$. Let $x\in f^{-1}U\cap X^{\max}$, and let $\Sigma$ and $\Delta$ be local face structures around $x$ and $f(x)$, respectively, such that $f(\sigma)\in\Delta$ for all $\sigma\in\Sigma$. Because $\bigwedge f^\sharp(\omega)$ is constant on a neighborhood of $x$, it suffices to show that it vanishes on a maximal cell $\sigma\in\Sigma$. We may thus replace $X$ by $\sigma$, in which case $f$ factors through the rational polyhedral space $f(\sigma)$. Since $f(\sigma)\in \Delta$, there exists a maximal cell $\delta\in \Delta$ containing $f(\sigma)$. As $f$ factors through $\delta$ by construction, it then suffices to show that the pullback of $\omega$ to $\bigwedge^*\Omega^1_\delta$ vanishes. But as we assumed that the restriction of $\omega$ to the interior $\mathring \delta$ of $\delta$ vanishes, this follows from the fact that the sheaf $\bigwedge\nolimits^*\Omega^1_{\delta}$ on $\delta$ is constant.
\end{proof}

\section{Tropical cycles on rational polyhedral spaces}
\label{sec:tropical cycles}

\subsection{Tropical cycles}
To define tropical cycles on rational polyhedral spaces, we first need to recall their definition on affine space.
\begin{defn}
\label{def:tropical fan cycle}
Let $N$ be a lattice. A \emph{tropical fan $k$-cycle} on $N_\R=N\otimes_\Z \R$ is an integer valued function $A\colon N_\R\to \Z$ such that
\begin{enumerate}
\item For every $\lambda\in \R_{>0}$ and $x\in N_\R$ we have $A(\lambda x)=A(x)$,
\item The \emph{support} $|A|=\overline{\{x\in N_\R\mid A(x)\neq 0\}}$ of $A$ is a purely $k$-dimensional polyhedral subset of $N_\R$,
\item $A$ is locally constant on the open subset $|A|^{\max}$ of $|A|$ and $0$ on $|A|\setminus |A|^{\max}$,
\item $A$ satisfies the so-called balancing condition: if $\Sigma$ is a face structure on $|A|$ such that every $\sigma\in \Sigma$ is a cone, then $A$ is constant on the relative interiors of the inclusion-maximal cells of $\Sigma$. Therefore, $A$ and $\Sigma$ define a weighted fan in the sense of Allermann and Rau \cite{AllerRau}. We ask that this weighted fan satisfies the so-called balancing condition, that is that it is a tropical fan in the sense of \cite{AllerRau}. By \cite[Lemma 2.11]{AllerRau} this is independent of the choice of $\Sigma$.
\end{enumerate}
\end{defn}

\begin{rem}
It is immediate from the definition that every tropical fan cycle in the sense above defines a tropical fan cycle in the sense of \cite{AllerRau}. If, conversely, $A=[(\Sigma,\omega)]$ is a tropical fan cycle in the sense of \cite{AllerRau}, where we use their notation here, then $\omega$ defines a locally constant integer-valued function on the subset of $|A|^{\max}$ consisting of the union of the relative interiors of all inclusion-maximal cones of $\Sigma$. This function can be extended uniquely to a locally constant function on all of $|A|^{\max}$ that is independent of the representative $(\Sigma,\omega)$ and, if extended by $0$ to all of $N_\R$, is a tropical fan cycle in the sense of Definition \ref{def:tropical fan cycle}. 
\end{rem}

\begin{rem}
\label{rem:lattice normal vector and balancing}
We will rarely use the balancing-condition, but let us briefly recall its definition for the sake of being self-contained. If $\sigma$ is a cone in $N_\R$ and $\tau$ is a codimension-$1$ face of $\sigma$, then a \emph{lattice normal vector} for $\sigma$ with respect to $\tau$ is an element $n\in\sigma\cap N$ such that the morphism
\[
\bigwedge\nolimits^{\dim(\tau)}T^\Z(\tau) \to \bigwedge\nolimits^{\dim(\sigma)}T^\Z(\sigma), \;\; \eta\mapsto n\wedge\eta
\]
is an isomorphism, where the tangent spaces $T^\Z(\sigma)$ and $T^\Z(\tau)$ are taken at any point of the respective cones, and we consider them as sublattices of $N$. If $\Sigma$ is a purely $k$-dimensional rational polyhedral fan in $N_\R$, and $\omega\colon \Sigma(k)\to \Z$ gives integer weights to its maximal cones, then $(\Sigma,\omega)$ satisfies the balancing condition if for every $\tau\in\Sigma(k-1)$ we have
\[
\sum_{\sigma\colon \tau\subseteq \sigma\in\Sigma(k)} \omega(\sigma) n_{\sigma/\tau} \in T^\Z(\tau)
\]
for any, and hence every, choice of lattice normal vectors $n_{\sigma/\tau}$ of $\sigma$ with respect to $\tau$.
\end{rem}

Tropical fan $k$-cycles on $N_\R$ form an Abelian group. We remark that the sum of two such tropical cycles $c$ and $d$ is {\em not} simply the sum as integer-valued functions. This does hold, however, if we consider integer-valued functions modulo those functions whose support is a polyhedral set of dimension at most $k-1$.

\begin{defn}
Let $X$ be a rational polyhedral space. We say that a function $A\colon X\to \Z$ is \emph{locally constructible} if for every $x\in X$ there exists a local face structure $\Sigma$ at $x$ such that the restrictions $A|_{\relint(\sigma)}$ are constant for all $\sigma\in \Sigma$.
\end{defn}

Every integer-valued function $A\colon X\to \Z$ on a rational polyhedral space $X$ induces, at every $x\in X$, a function germ at the origin of the local cone $\LC_x X\subseteq T_x X$ via Proposition \ref{prop:local cone}.  If $A$ is locally constructible, then for every $v\in \LC_x X$ the value $A(\epsilon v)$ is independent of $\epsilon>0$, if chosen sufficiently small. This can be used to extend the germ to an $\R_{>0}$-invariant function $\LC_x X\to \Z$, which we extend by $0$ to a function 
\[
\LC_x(A)\colon T_x X\to \Z \ .
\]

\begin{defn}
\label{def:tropical cycles}
Let $X$ be a rational polyhedral space. A \emph{tropical $k$-cycle} on $X$ is a locally constructible function $A\colon X\to \Z$ such that $\LC_x(A)$ is a tropical fan $k$-cycle on $T_x X$ for all $x\in X$. The \emph{support} of a tropical cycle $A$ on $X$ is the set $|A|=\overline {\{x\in X\mid A(x)\neq 0\}}$.
\end{defn}

The addition for tropical fan $k$-cycles induces an addition for tropical $k$-cycles on a rational polyhedral space $X$ so that there is an abelian group $Z_k(X)$ of tropical $k$-cycles. Similarly as for tropical fan cycles, the sum $c+d$ of two tropical $k$-cycles agrees with the sum of $A$ and $B$ as integer-valued functions up to an integer-valued function whose support is a locally polyhedral subset of $X$ of dimension at most $k-1$. As both the definition of tropical $k$-cycles and the definition of the addition are local, the assignment $U\to Z_k(U)$ on open sets of $X$, with the obvious restriction morphisms, defines a sheaf $\mathcal Z_k^X$ of tropical $k$-cycles on $X$. 

\subsection{Proper push-forward of tropical cycles}
\label{subsec:push-forward of cycles}

\begin{defn}
\label{def:push-forward of cycles}
Let $A$ be a tropical $k$-cycle on a $k$-dimensional rational polyhedral space $X$, and let $f\colon X\to Y$ be a proper and surjective morphism of rational polyhedral spaces. Then for ever $y\in Y$ which is not an element of the at most $(k-1)$-dimensional locally polyhedral subset 
\[
f\left(X\setminus X^{\max}\right)\cup\left( Y\setminus Y^{\max}\right) 
\]
of $Y$, we define the \emph{push-forward} of $A$ along $f$ as
\[
f_*A(y)= \sum_{x\in f^{-1}\{y\}} [ T^\Z_y X \colon d_x f(T^\Z_x X)] A(x) \ ,
\] 
where we consider the lattice index as $0$ if it is not finite. We extend this function by $0$ to a function on $Y$. Note that the sum over $f^{-1}\{y\}$ is in fact finite, since we can only get a nonzero contribution for isolated points of $f^{-1}\{y\}$, of which there can only be finitely many because $f$ is proper. 

In the general case, where  $X$ is not necessarily $k$-dimensional and $f$ is not necessarily surjective (but still proper), we consider the (co)restriction $\widetilde f\colon |A|\to f|A|$ of $f$ and define $f_*A$ as the extension to $Y$ by $0$ of $\widetilde f_*A$.
\end{defn}

If $f\colon X\to Y$ is a proper morphism of rational polyhedral spaces, and $A\in Z_k(X)$, then it follows immediately from the construction that $f_*A$ is locally constructible. It is usually not a tropical cycle in the sense of Definition \ref{def:tropical cycles}, but there is a unique tropical $k$-cycle $B$ on $Y$ such that $f_*A$ and $B$ coincide away from a locally polyhedral subset of dimension at most $(k-1)$. The uniqueness is clear, and for the existence part one only needs to show balancing, which can be proven locally and thus follows exactly as in \cite[Proposition 2.25]{GKM}. From this it is clear that the push-forward induces a morphism
\[
Z_k(X)\to Z_k(Y)
\]
of groups of tropical $k$-cycles, which, by abuse of notation, we denote by $f_*$ as well.

\subsection{Cross products of tropical cycles}
\label{subsec:cross product of cycles}

Given two tropical cycles on two rational polyhedral spaces, one gets a tropical cycle on the product space by taking their cross-product.

\begin{defn}
Let $X$ and $Y$ be rational polyhedral spaces, and let $A\in Z_k(X)$ and $B\in Z_l(X)$. Then we define the \emph{cross-product} of $A$ and $B$ as the function
\[
A\times B\colon X\times Y \to \Z, \; (x,y)\mapsto A(x)\cdot B(y) \ .
\]
It is straightforward to check that this is a tropical $(k+l)$-cycle on $X\times Y$, and it is evident from the definition that the cross-product defines a bilinear map
\[
Z_k(X)\times Z_l(Y)\to Z_{k+l}(X\times Y) \ .
\]
\end{defn}

\subsection{Tropical Cartier divisors}
\label{subsec:Cartier divisors and line bundles}

\begin{defn}
Let $X$ be a rational polyhedral space. A continuous function $\phi\colon X\to \R$ is \emph{rational} if at every $x$ in $X$ there exists a local face structure $\Sigma$ such that $\phi|_\sigma\in \Gamma(\sigma,\Aff_\sigma)$ for all $\sigma\in \Sigma$. Sums of rational functions are rational. We denote the group of rational functions on $X$ by $\Rat(X)$. 
\end{defn}

\begin{rem}
Rational functions on a rational polyhedral space $X$ are precisely the piecewise linear function on $X$ with integral slopes. The terminology ``rational'' comes from the analogy with algebraic varieties, where the rational functions play a similar role in the definition of divisors.
\end{rem}

The condition on a function on a rational polyhedral space $X$ to be rational is a  local condition. Therefore, the presheaf $U\mapsto \Rat(U)$ on $X$ is in fact a sheaf, which we denote by $\Rat_X$. Every affine linear function on $X$ is rational, so there is an inclusion $\Aff_X\hookrightarrow\Rat_X$. Its quotient is the sheaf $\Divs_X$ of Cartier divisors, that is $\Divs_X$ is defined as the unique sheaf fitting into a short exact sequence
\[
0\to \Aff_X\to \Rat_X \to \Divs_X \to 0 \ .
\]

\begin{defn}
Let $X$ be a rational polyhedral space. The group
\[
\Div(X)\coloneqq \Gamma(X,\Divs_X)
\]
is the group of \emph{Cartier divisors} on $X$.  The \emph{support} $|D|$ of $D\in\Div(X)$ is defined in the sheaf-theoretic sense as the support of $D$ considered as a global section of $\Divs_X$.
\end{defn}

If $f\colon X\to Y$ is a morphism of rational polyhedral spaces, then it is straightforward to check that for every rational function $\phi$ on $Y$, the pull-back $f^*\phi=\phi\circ f$ is a rational function on $X$. Since the pull-back of rational functions is compatible with the pull-back of affine linear functions, we obtain a pull-back morphism
\[
f^*\colon \Div(Y)\to \Div(X) 
\]
for Cartier divisors.

There is an intersection pairing 
\[
\Div(X)\times Z_k(X)\to Z_{k-1}(X)
\]
on every rational polyhedral space $X$ due to Allermann and Rau \cite{AllerRau}. Let us briefly recall its construction. To define the product $D\cdot A$ of a divisor $D$ with a tropical $k$-cycle $A$, we first pull back $D$ to $|A|$, after which we can assume that $X=|A|$. We can then work locally around a point $x\in X$ and replace $X$ by its local cone $\LC_x X$. This allows us to assume that $X=|\Sigma|$ for some rational polyhedral cone $\Sigma$ in $\R^n$, that $A$ is represented by a balanced weight function on the $k$-dimensional cones of $\Sigma$, and that $D$ is represented by a piecewise linear function $\phi$ whose restrictions to the cones of $\Sigma$ are linear. The intersection $D\cdot A$ is then represented by the weight $\Sigma(k-1)\to \Z$ that assigns to $\tau\in\Sigma(k-1)$ the integer
\[
\sum_{\sigma\colon \tau\subseteq\sigma\in\Sigma(k)} \langle \phi-l_\tau, n_{\sigma/\tau}\rangle A(\sigma) \ ,
\]
which is independent of the choice of lattice normal vectors $n_{\sigma/\tau}$ (see Remark \ref{rem:lattice normal vector and balancing}) and an integer linear function $l_\tau$ on $\R^n$ with $l_\tau|_\tau=\phi|_\tau$.

\begin{rem}
Note that rational functions on compact rational polyhedral spaces are bounded by continuity. If one allows rational functions to obtain the value $\infty$, one obtains a less restrictive notion of tropical Cartier divisors, for which one should still be able to define the intersection pairing with tropical cycles, at least under some mild assumptions on the underlying rational polyhedral space. In the prototypical example of tropical toric varieties this has been done in \cite{MeyerDiss}.
\end{rem}

\subsection{Tropical line bundles}
\label{subsec:line bundles}

Following \cite{MZjacobians}, we work with the following definition of tropical line bundles:

\begin{defn}
A \emph{tropical line bundle} on a rational polyhedral space $X$ is a morphism $Y\to X$ of rational polyhedral spaces such that locally on $X$ there are identifications $Y\cong \Rbar\times X$ of spaces over $X$. Two tropical line bundles are isomorphic if they are isomorphic as rational polyhedral spaces over $X$.
\end{defn}

Note that the only automorphisms of $\Rbar$ are the ones of the form $x\mapsto \lambda+ x$ for some $\lambda\in \R$. Therefore, the automorphism group of $\Rbar\times U $ is naturally isomorphic to $\Gamma(U,\Aff_U)$ for any rational polyhedral space $U$. Using standard arguments involving \v{C}ech cohomology, this leads to the the following description of the set of isomorphism classes of tropical line bundles on a rational polyhedral space:

\begin{prop}[cf.\ {\cite{MZjacobians}}]
Let $X$ be a rational polyhedral space. Then there is a natural bijection between the set of all isomorphism classes of tropical line bundles on $X$ and the cohomology group $H^1(X,\Aff_X)$. In particular, the set of isomorphism classes of tropical line bundles on $X$ is a group.
\end{prop}

If $X$ is a rational polyhedral space, then the first boundary map in the long exact cohomology sequence associated to the short exact sequence
\[
0\to \Aff_X\to \Rat_X \to \Divs_X \to 0 
\]
associates to every Cartier divisor $D\in H^0(X,\Divs_X)$ a tropical line bundle 
\[
\mL(D)\in H^1(X,\Aff_X) \ .
\] 

\begin{rem}
It follows from \cite[Lemma 4.5]{Lefschetz} that every tropical line bundle is of this form if $X$ admits a face structure. We expect this to remain true even in the absence of face structures, but will neither prove nor use this fact in the remainder of this paper.
\end{rem}

If $f\colon X\to Y$ is a morphism between rational polyhedral spaces, then applying $H^1$ to the pull-back morphism $f^\sharp\colon f^{-1}\Aff_Y\to \Aff_X$ induces a pull-back morphism 
\[
f^*\colon H^1(Y,\Aff_Y)\to H^1(X,\Aff_X)
\]
for tropical line bundles.

\begin{prop}
\label{prop:pullback and taking line bundles commutes}
Let $f\colon X\to Y$ be a morphism between rational polyhedral spaces, and let $D\in \Div(Y)$ be a Cartier divisor on $Y$. Then we have 
\[
f^*\mL(D)= \mL(f^*D) \ .
\]
\end{prop}

\begin{proof}
This follows immediately from the commutativity of the diagram
\begin{center}
\begin{tikzpicture}[auto]
\matrix[matrix of math nodes, row sep= 5ex, column sep= 4em, text height=1.5ex, text depth= .25ex]{
|(DivY)| H^0(Y, \Divs_Y) &
|(AffY)|  H^1(Y, \Aff_Y) \\
|(XDivY)| H^0(X,f^{-1}\Divs_Y)	&
|(XAffY)| H^1(X,f^{-1}\Aff_Y) \\
|(DivX)| H^0(X,\Divs_X)	&
|(AffX)| H^1(X,\Aff_X ) \ ,\\
};
\begin{scope}[->,font=\footnotesize]
\draw (DivY) -- (AffY);
\draw (XDivY) -- (XAffY);
\draw (DivX) -- (AffX);
\draw (AffY)--node{$f^{-1}$}(XAffY);
\draw (DivY)--node{$f^{-1}$}(XDivY);
\draw (XAffY) --node{$H^1(f^\sharp)$} (AffX);
\draw (XDivY)-- (DivX);
\end{scope}
\end{tikzpicture} 
\end{center}
where the horizontal morphisms are the first boundary maps in the long exact cohomology sequences associated to the short exact sequences
\begin{center}
\begin{tikzpicture}[auto]
\matrix[matrix of math nodes, row sep= 1ex, column sep= 2em, text height=1.5ex, text depth= .25ex]{
|(0Y1)| 0& [-1em]
|(AffY)| \Aff_Y &
|(MY)|  \Rat_Y &
|(DivY)| \Divs_Y &[-1em]
|(0Y2)| 0  &[-1.8em]
 , \phantom{\text{ and}}
\\
|(X0Y1)| 0&
|(XAffY)| f^{-1}\Aff_Y &
|(XMY)|  f^{-1}\Rat_Y &
|(XDivY)| f^{-1}\Divs_Y &
|(X0Y2)| 0 \  &
,\text{ and}\\
|(0X1)| 0 &
|(AffX)| \Aff_X &
|(MX)|  \Rat_X &
|(DivX)| \Divs_X &
|(0X2)| 0 &
 , \phantom{\text{ and}} \\
};
\begin{scope}[->,font=\footnotesize]
\draw (0Y1) -- (AffY);
\draw (AffY) -- (MY);
\draw (MY) -- (DivY);
\draw (DivY) -- (0Y2);
\draw (X0Y1) -- (XAffY);
\draw (XAffY) -- (XMY);
\draw (XMY) -- (XDivY);
\draw (XDivY) -- (X0Y2);
\draw (0X1) -- (AffX);
\draw (AffX) -- (MX);
\draw (MX) -- (DivX);
\draw (DivX) -- (0X2);
\end{scope}
\end{tikzpicture} 
\end{center}
respectively.
\end{proof}

\section{Tropical (co)homology and its functorial properties}
\label{sec:tropical (co)homology}

\subsection*{Notation and general references} We will denote the constant sheaf associated to an abelian group $A$ on a topological space $X$ by $A_X$. If $\mF$ is any sheaf on $X$ and $S\subseteq X$ is a locally closed subset, we will denote $\mF_S= \iota_! \iota^{-1}\mF$, where $\iota\colon S\to X$ denotes the inclusion. For an abelian group $A$, we will sometimes denote $(A_X)_S$ by $A_S$ if $X$ is clear from the context. We will denote the group of morphism between two $\Z_X$-modules (sheaves of abelian groups on $X$, that is) $\mF$ and $\mG$ by $\Hom_{\Z_X}(\mF,\mG)$, where we will omit the subscript $\Z_X$ if $X$ is clear from the context. The bounded derived category of $\Z_X$-modules will be denoted by $\der X$, and we will omit the subscript $X$ if $X$ is a point, that is $\der {}$ denotes the bounded derived category of abelian groups. If $\mathcal C^\bullet$ and $\mathcal D^\bullet$ are two cochain complexes of $\Z_X$-modules, then $\Hom_{\der X}(\mathcal C^\bullet, \mathcal D^\bullet)$ denotes the group of morphisms between them in $\der X$. As usual $\Hom^\bullet(\mathcal C^\bullet,\mathcal D^\bullet)$ denotes the $\Hom$-complex and $R\Hom^\bullet(\mathcal C^\bullet,\mathcal D^\bullet)$ the derived $\Hom$-complex, and similarly for the internal hom $\Homs$. The $i$-th cohomology sheaf of $\mathcal C^\bullet$ will be denoted by $H^i(\mathcal C^\bullet)$, whereas the $i$-th hypercohomology will be denoted by $\Hyper^i(\mathcal C^\bullet)$.

For background on sheaf theory and Verdier duality we refer the reader to  \cite{IntersectionCohomology,Bredon,Iversen,KashiwaraSchapira}

\subsection{Tropical cohomology}

Tropical cohomology groups are defined in analogy with the cohomology groups appearing in the Hodge decomposition in algebraic geometry.

\begin{defn}
Let $X$ be a rational polyhedral space, and let $p,q\in \N$. Then the $(p,q)$-\emph{th tropical cohomology group} is defined as
\[
H^{p,q}(X)=H^q(X,\Omega^p_X) \ .
\]
If $Z$ is a closed subset of $X$, the $(p,q)$-th tropical cohomology \emph{with supports in $Z$} is defined as
\[
H^{p,q}_{Z}(X)=H^q_Z(X,\Omega^p_X) \ ,
\]
where $H^q_{Z}$ denotes cohomology with support in $Z$, that is the $q$-th right derived functor associated to the functor $H^0_Z=\Gamma_Z$ that takes sections with support contained in $Z$.
\end{defn}

\begin{rem}
\label{rem:cohomology classes as morphisms}
We will frequently use the isomorphism
\[
H^{p,q}_Z(X)\cong \Hom_{\der X}(\Z_Z, \Omega^p_X[q])
\]
and thus write a tropical $(p,q)$-cohomology class $\alpha\in H^{p,q}_Z(X)$ as an arrow $\Z_Z\xrightarrow{\alpha} \Omega^p_X[q]$ in $\der X$.
\end{rem}

\subsection{Tropical Borel-Moore homology}
Similarly to the definition of the classical (i.e.\ non-tropical) Borel-Moore homology, our definition of tropical Borel-Moore homology will utilize the dualizing complex.

The \emph{dualizing complex} $\dual_X$ of a rational polyhedral space $X$ is an element of $\der X$ representing the functor
\begin{equation*}
\der X\to \der {} \colon A\mapsto \Hom_{\der{}} (R\Gamma_c A, \Z) \ ,
\end{equation*} 
where $R\Gamma_c$ is the (total) right derived functor of taking global sections with compact support, and $\Z$ is considered as a complex concentrated in degree $0$. 
The universal element of the representation, that is the image of $\id_{\dual_X}$ under the isomorphism 
\[
\Hom_{\der X}(\dual_X,\dual_X) \xrightarrow{\cong} \Hom_{\der{}}(R\Gamma_c\dual_X, \Z) \ ,
\]
is called the \emph{trace map} and we will denote it by $\int_X$.

\begin{defn}
Let $X$ be a rational polyhedral space, and let $p,q\in \N$. We define the  \emph{$(p,q)$-th (integral) tropical Borel-Moore homology group} as
\begin{equation*}
H^{BM}_{p,q}(X) \coloneqq \Hyper ^0R\Homs^\bullet(\Omega_X^p[q],\dual_X) \ ,
\end{equation*}
and the \emph{$(p,q)$-th (integral) tropical homology group with compact support} as
\[
H_{p,q}(X)\coloneqq \Hyper_c^0 R\Homs^\bullet(\Omega_X^p[q],\dual_X) \ .
\]
Furthermore, for a closed subset $Z\subseteq X$ we define
\[
H^{BM}_{p,q}(Z,X) \coloneqq \Hyper _Z^0 R\Homs^\bullet(\Omega_X^p[q],\dual_X) \ .
\]
\end{defn}

\begin{rem}
\label{rem:homology classes as morphisms}
Similarly as for cohomology classes, we will usually identify Borel-Moore homology classes with morphisms in the derived category. To do so, we use the identification
\begin{multline*}
H_{p,q}^{BM}(Z,X) = 
\Hyper^0_Z R\Homs^\bullet(\Omega_X^p[q],\dual_X) \cong \\
\cong H^0 R\Hom^\bullet((\Omega_X^p)_Z[q],\dual_X)
\cong \Hom_{\der X}((\Omega_X^p)_Z[q],\dual_X) 
\end{multline*}
that allows as to write $\alpha\in H_{p,q}^{BM}(Z,X)$ as an arrow $(\Omega_X^p)_Z[q]\xrightarrow{\alpha} \dual_X$.
\end{rem}

\begin{rem}
\label{rem:homology using Verdier duality}
Using Verdier duality, one obtains an identification
\begin{multline*}
H_{p,q}^{BM}(X) = 
\Hyper^0 R\Homs^\bullet(\Omega_X^p[q],\dual_X) \cong \\
\cong H^{0} R\Hom^\bullet(\Omega_X^p[q],\dual_X)
\cong H^{-q}R\Hom^\bullet(R\Gamma_c\Omega_X^p,\Z) \ .
\end{multline*}

\end{rem}

The dualizing complex $\dual_X$ on a rational polyhedral space can be described explicitly in terms of sheaves of singular chains.
\begin{defn} 
Let $X$ be a rational polyhedral space.
\begin{itemize}
\item[(a)] For $i\in \N$, let $\Delta_X^i$ denote the sheafification of the presheaf 
\[
U \mapsto C_{-i}(X, X\setminus U) \ ,
\] 
 where $C_j(A,B)$ denotes the group of relative singular $j$-chains with $\Z$-coefficients of the pair $A\supseteq B$. With the usual (co)boundary maps we obtain a cochain complex $\Delta_X^\bullet$.
\item[(b)] The $i$-th homology sheaf $\scrH^i_X \coloneqq H^{-i}(\Delta_X^\bullet)$ is the sheafification of the presheaf 
\[
U\mapsto H_{i}(X,X\setminus U) \ ,
\] whose stalk at $x\in X$ is canonically identified with $H_{i}(X,X\setminus \{x\})$.
\end{itemize}
\end{defn}

The complex $\Delta^\bullet_X$ is homotopically fine \cite[VI, Proposition 7]{SwanSheaves}, which implies that the natural morphism
\[
\Gamma_c\Delta^\bullet_X\to R\Gamma_c\Delta^\bullet
\]
is a quasi-isomorphism. The global sections with compact support of $\Delta^{-i}_X$ are naturally identified with $C_i(X)$, so the natural augmentation $C_\bullet(X)\to \Z$ defined by taking degrees of $0$-chains defines a morphism $\Delta_X^\bullet\to \dual_X$ in the derived category (using the universal property of $\dual_X)$, which is well-known to be an isomorphism. In particular, we have $\scrH^i_X\cong H^{-i}(\dual_X)$. This is one way of seeing that $H^i(\dual_X)=0$ for $i< -\dim(X)$. 

\begin{rem}
\label{rem:combinatorial dualizing complex}
If $X$ is a rational polyhedral space with a face structure $\Sigma$, the dualizing complex has a completely combinatorial description due to Shepard \cite{ShepardThesis}. More precisely, $\dual_X$ is quasi-isomorphic to the complex which is given by $\bigoplus_{\sigma\in\Sigma(k)} \Z_\sigma$ in degree $-k$, where $\Sigma(k)$ denotes the set of all $k$-dimensional polyhedra in $\Sigma$. To define the differentials, one needs to pick orientations on all $\sigma\in \Sigma$. Having done this, the differential is $0$ between the components $\Z_\sigma$ and $\Z_\tau$ of $\bigoplus_{\sigma\in\Sigma(k)}\Z_\sigma$ and $\bigoplus_{\tau\in\Sigma(k-1)}\Z_\tau$ if $\tau\nsubseteq \sigma$, and multiplication by $\epsilon_{\sigma/\tau}$ else, where $\epsilon_{\sigma/\tau}$ is $1$ if the orientations on $\sigma$ and $\tau$ agree, and $-1$ else.
\end{rem}

\begin{lemma}\phantomsection \label{lem:basicBM}
\begin{itemize} 
\item[]
\item[(a)] The classical $q$-th (non-tropical) Borel-Moore homology group of $X$ is isomorphic to $H^{BM}_{0,q}(X)$.
\item[(b)] Let $Z$ be a locally polyhedral subset of dimension $d$, and let $\iota\colon Z\to X$ denote the inclusion. Then for all $p\in \N$ we have
\[
H^{BM}_{p,d}(Z,X)= \Hom_{\Z_X}(\Omega^p_X, \iota_*\scrH_Z^d) \ .
\]
In particular, the presheaf $U\mapsto H^{BM}_{p,d}(U\cap Z, U)$ on $X$ is a sheaf.
\end{itemize}
\end{lemma}
\begin{proof}
For (a) we use the fact that $\Omega^0_X=\Z_X$ and the natural isomorphisms
\[
 H^{BM}_{0,q}(X)= H^0R\Hom^\bullet(\Z_X[q],\dual_X)\cong H^{-q}R\Gamma\dual_X = \Hyper^{-q}\dual_X \ .
\] 
By definition, $\Hyper^{-q}\dual_X$ is the $q$-th classical Borel-Moore homology group of $X$ (as introduced in \cite{BorelMoore}).

For part (b) we can use Remark \ref{rem:homology classes as morphisms} and the universal property of the dualizing complex, to obtain the isomorphism
\[
H^{BM}_{p,d}(Z,X)\cong\Hom_{\der X}((\Omega_X^p)_Z[d], \dual_X)\cong \Hom_{\der Z}(\Omega_X^p|_Z[d],\dual_Z) \ .
\]
Since $d=\dim(Z)$, the cohomology groups $H^i(\dual_Z)$ vanish for $i<-d$. Therefore, $\dual _Z$ is quasi-isomorphic to a complex of injectives that is $0$ in degrees $<-d$. It follows that 
\[
\Hom_{\der Z}(\Omega^p_X|_Z[d],\dual_Z)\cong \Hom_{\Z_Z}(\Omega^p_X|_Z, H^{-d}(\dual_Z)) \ ,
\]
which equals $\Hom_{\Z_Z}(\Omega^p_X|_Z, \scrH^d_Z)$ by definition of $\scrH^d_Z$. This in turn is isomorphic to $\Hom_{\Z_X}(\Omega^p_X,\iota_*\scrH^d_Z)$. For the ``in particular''-statement we note that for every open subset $U\subseteq X$ we have $\dual_U\cong \dual_X|_U$. Therefore, the presheaf $U\mapsto H^{BM}_{p,d}(U\cap Z,U)$ is isomorphic to the presheaf $U\mapsto \Hom_{\Z_U}(\Omega_X^p|_U, \iota_*\scrH^d_Z|_U)$, which equals the sheaf $\Homs_{\Z_X}(\Omega^p_X,\iota_*\scrH^d_Z)$.

\end{proof}

\subsection{Pull-backs} \label{subsec:pullback}
Let $f\colon X\to Y$ be a morphism of rational polyhedral spaces. Recall from Proposition \ref{prop:pullback of p-forms} that pulling back tropical forms defines a morphism of graded sheaves of rings
\[
f^\sharp\colon f^{-1}\Omega_Y^*\to \Omega_X^* \ .
\]
Let $(\Z_Y \xrightarrow{c} \Omega^p_Y[q])\in H^{p,q}(Y)$ be a tropical $(p,q)$-cohomology class. As the pull-back $f^{-1}$ of sheaves of abelian groups defines an exact functor, it induces a functor $f^{-1}\colon \der Y\to \der X$. Applying this functor to $c$ and composing the resulting arrow with $f^\sharp$ defines the \emph{pull-back} 
\[
f^*c\in H^{p,q}(X) \ .
\]
In other words, $f^*c$ is represented by the composite
\[
\Z_X\cong f^{-1}\Z_Y\xrightarrow{f^{-1}c} f^{-1}\Omega^p_Y[q]\xrightarrow{f^\sharp[q]} \Omega^p_X[q] \ .
\]
The map $f^*\colon H^{p,q}(Y)\to H^{p,q}(X)$ is a morphism of abelian groups.

\subsection{Proper push-forwards} \label{sec:BMfunc}
If $f\colon X\to Y$ is a proper morphism of rational polyhedral spaces, then precomposing the trace $\int_X \colon R\Gamma_c\dual_X\to \Z$ with the natural isomorphism $R\Gamma_c\circ Rf_*(\dual_X) \xrightarrow{\cong} R\Gamma_c\dual_X$ defines a morphism $R\Gamma_c(Rf_*\dual_X) \to \Z$. By the universal property of the dualizing complex $\dual_Y$, this corresponds to a morphism 
\begin{equation}\label{eq:Rf}
Rf_*\dual_X \to \dual_Y\ .
\end{equation}
Together with the composite
\begin{equation}\label{eq:Omp}
\Omega^p_Y\to f_*\Omega^p_X\to Rf_*\Omega^p_X 
\end{equation}
obtained by pulling back tropical $p$-forms, this defines a push-forward on tropical Borel-Moore homology:

\begin{defn}
\label{defn:top dimensional cycle map}
Let $f\colon X\to Y$ be a proper morphism of rational polyhedral spaces, and let $p,q\in \N$. The \emph{pushforward map}
\[
f_* \colon H^{BM}_{p,q}(X) \rightarrow H^{BM}_{p,q}(Y)
\]
associated to $f$ is the composite of the morphism
\[
\Hom_{\der X}(\Omega^p_X[q],\dual_X) \rightarrow 
\Hom_{\der Y}(Rf_* \Omega^p_X[q],Rf_*\dual_X)
\] 
obtained by taking the derived push-forward and the morphism
\[
\Hom_{\der Y}(Rf_* \Omega^p_X[q],Rf_*\dual_X)\rightarrow
\Hom_{\der Y}(\Omega^p_Y[q], \dual_Y)
\]
defined via composition with the natural morphisms $\Omega^p_Y\to Rf_*\Omega^p_X$ in \eqref{eq:Omp} and $Rf_*\dual_X\to \dual_Y$ in \eqref{eq:Rf}.
\end{defn}

\begin{rem}
\label{rem:push-forward with support}
Since the morphism $\Omega^p_Y\to f_*\Omega^p_X$ factors through $(\Omega^p_Y)_{f(X)}$, the push-forward map factors through $H^{BM}_{p,q}(f(X), Y)$.
\end{rem}

It follows immediately from the functoriality of the derived push-forward and the pull-back of tropical forms that  the push-forward on tropical Borel-Moore homology is functorial, that is $(f\circ g)_*=f_*\circ g_*$ whenever $f$ and $g$ are composable proper morphisms of rational polyhedral spaces.

For a better understanding of the push-forward we will need the following lemma:

\begin{lemma}
\label{lem:description push-forward of dualizing complex}
Let $f\colon X\to Y$ be a proper morphisms of rational polyhedral spaces, and let $n=\dim X$. Then the morphism
\[
f_* \scrH^n_X \to \scrH^n_Y
\]
induced by the natural morphism $Rf_*\dual_X\to \dual_Y$ in \eqref{eq:Rf} is induced by the push-forwards $H_n(X,X\setminus f^{-1}U)\to H_n(Y, U)$ of relative singular cycles for open subsets $U\subseteq Y$.
\end{lemma}

\begin{proof}
We will describe the morphism $Rf_*\Delta_X^\bullet \to \Delta_Y^\bullet$ induced by the natural morphism $Rf_*\dual_X\to\dual_Y$ explicitly. To do so, we will first need to describe $Rf_*\Delta_X^\bullet$. Note that the barycentric subdivision defines an endomorphism on $\Delta_X^\bullet$. Let $\SoftDual^\bullet_X= \injlim_\N \Delta_X^\bullet$ denote the direct limit obtained by allowing repeated barycentric subdivision, and define $\SoftDual^\bullet_Y$ similarly. The natural morphism $\Delta_X^\bullet\to \SoftDual^\bullet_X$ is a quasi-isomorphisms since taking homology commutes with direct limits and the barycentric subdivision is a quasi-isomorphism. In particular, $Rf_*\Delta^\bullet_X=Rf_*\SoftDual^\bullet_X$. By \cite[Prop.\ V-1.8 and Thm.\ V-12.14]{Bredon}, $\SoftDual_X^\bullet$ is a complex of soft sheaves, so the natural morphisms $f_*\SoftDual_X^\bullet\to Rf_*\SoftDual_X^\bullet$ and $\Gamma_c\SoftDual_X^\bullet\to R\Gamma_c\SoftDual_X^\bullet$ are quasi-isomorphisms. 

Next, we will show that the push-forward of relative singular cycles induces a morphism $f_*\SoftDual_X^\bullet\to \SoftDual_Y^\bullet$. 
If $\SoftDual^{pre,i}_X$ denotes the presheaf $U\mapsto \injlim_\N C_{-i}(X,X\setminus U)$ on $X$, and $\SoftDual^{pre,i}$ the analogous presheaf on $Y$, the push-forwards of relative chains $C_i(X,X\setminus f^{-1}U)\to C_i(Y,Y\setminus U)$ for $U\subseteq Y$ open  and $i\in\Z$ induce a morphism $f_*\SoftDual^{pre,\bullet}_X\to \SoftDual^{pre,\bullet}_Y$ by the functoriality of the barycentric subdivision. Note that this does \emph{not} automatically induce a morphism of complexes of sheaves $f_*\SoftDual^\bullet_X\to \SoftDual^\bullet_Y$ because push-forward does a priori not commute with sheafification. However using \cite[V-Lemma 1.7]{Bredon} one sees that the presheaves $\SoftDual^{pre,i}_X$ have the property that for every compact set $K\subseteq X$ there is an equality
\begin{equation*}
\injlim_{K\subseteq U} \SoftDual^{pre,i}_X (U) = \injlim_{K\subseteq U} \SoftDual^i_X (U)= \SoftDual^i_X(K) \ ,
\end{equation*}
where the direct limits are taken over all open subsets containing $K$ (note that the second equality holds for every sheaf). Applying this to the fibers of $f$ we obtain isomorphisms of stalks $(f_*\SoftDual^{pre,i}_X)_y\cong \SoftDual^i_X(f^{-1}\{y\})\cong (f_*\SoftDual^i_X)_y$
for all $y$, and hence sheafification commutes with push-forwards for the presheaves $\SoftDual^{pre,i}_X$. Consequently, the push-forward of relative singular cycles does in fact induce a morphism $f_*\SoftDual^\bullet_X \to \SoftDual^\bullet_Y$. To show that this coincides with the natural morphism $Rf_*\dual_X\to \dual_Y$ we apply $R\Gamma_c$ and obtain a morphism
\[
\Gamma_c\SoftDual^\bullet_X=\Gamma_c f_*\SoftDual^\bullet_X\to \Gamma_c\SoftDual^\bullet_Y \ ,
\]	
which is the one induced by pushing forward singular chains. In particular, in degree $0$ it is the morphism
\[
C_0(X) \to C_0(Y)
\]
that pushes forward points along $f$. This commutes with the degree morphisms to $\Z$, which are the traces defining the isomorphisms $\SoftDual^\bullet_X\cong \dual _X$ and $\SoftDual^\bullet_Y\cong \dual Y$. So by the definition of the natural morphism $Rf_*\dual_X\to \dual_Y$ it must agree with the  morphism $f_*\SoftDual^\bullet_X \to \SoftDual^\bullet_Y$ obtained by pushing forward relative chains. From this description it is clear that the morphism
\[
f_* \scrH^n_X \to \scrH^n_Y
\]
is induced by the push-forwards of relative singular cycles as well.
\end{proof}

\subsection{Cross products and the K\"unneth Theorem}
\label{subsec:cross products in homology}
In this section we study the tropical homology group on a product $X\times Y$ of two rational polyhedral spaces $X$ and $Y$.  Let $p_X\colon X\times Y\to X$ and $p_Y\colon X\times Y\to Y$ denote the projections. In what follows we will use the notation 
\[
\mathcal F \boxtimes \mathcal G = p_X^{-1} \mathcal F \otimes_{\Z_{X\times Y}} p_Y^{-1} \mathcal G
\]
for sheaves $\mathcal F$ on $X$ and $\mathcal G$ on $Y$, and
\[
\mathcal C^\bullet \boxtimes^L \mathcal D^\bullet = p_X^{-1} \mathcal C^\bullet \otimes^L_{\Z_{X\times Y}} p_Y^{-1} \mathcal D^\bullet
\]
for complexes of sheaves $\mathcal C^\bullet$ on $X$ and $\mathcal D^\bullet$ on $Y$, where $\otimes^L$ denotes the derived tensor product.

To define the cross-product in tropical homology, we first need to relate the dualizing complex of $X\times Y$ the dualizing complexes of the factors $X$ and $Y$. The trace maps $\int_X\colon R\Gamma_c\dual_X\to \Z$ and $\int_Y \colon R\Gamma_c\dual_Y\to \Z$ induce a morphism 
\[
R\Gamma_c\dual_X\otimes_\Z^L R\Gamma_c\dual_Y \to \Z\otimes_\Z^L\Z= \Z \ .
\]
By the K\"unneth formula \cite[VII-2.7]{Iversen}, $R\Gamma_c\dual_X\otimes^L R\Gamma_c\dual_Y$ is naturally isomorphic to $R\Gamma_c(\dual_X \boxtimes^L \dual_Y)$, so by the universal property of $\dual_{X\times Y}$ there is an induced morphism $\dual_X \boxtimes^L \dual_Y \to \dual_{X\times Y}$.
Indeed, this is an isomorphism by \cite[V, 10.26]{IntersectionCohomology}. 

It will be convenient for us later to have an explicit description of this isomorphism in terms of sheaves of singular chains.

\begin{lemma}
\label{lem:Eilenberg-Zilber}
Let $X$ and $Y$ be rational polyhedral spaces. Then for all $i,j\in\N$, the  morphism
\[
\scrH^i_X\boxtimes \scrH^j_Y \to \scrH^{i+j}_{X\times Y}
\]
defined via the natural isomorphism $\dual_X\boxtimes^L\dual_Y \xrightarrow{\cong} \dual_{X\times Y}$ is the one induced by the relative cross products
\[
H_i(X,X\setminus U) \otimes_\Z H_j(Y,Y\setminus V) \to H_{i+j}(X\times Y, (X\times Y)\setminus (U\times V))
\]
for open subsets $U\subseteq X$ and $V\subseteq Y$.
\end{lemma}

\begin{proof}
The Eilenberg-Zilber map defines a morphism
\[
C_\bullet(X)\otimes_Z C_\bullet(Y)\to C_\bullet(X\times Y )
\]
that induces morphisms
\[
C_\bullet(X,X\setminus U)\otimes_\Z C_\bullet(Y,Y\setminus V)\to C_\bullet(X\times Y, (X\times Y)\setminus (U\times V) )
\]
for all pairs of open subsets $U\subseteq X$ and $V\subseteq Y$. Since the products $U\times V$ for $U\subseteq X$ and $V\subseteq Y$ open form a basis for $X\times Y$, we obtain a morphism
\[
\Delta_X^\bullet \boxtimes \Delta_Y^\bullet \to \Delta_{X\times Y}^\bullet
\] 
after sheafifying. By construction, the morphisms $\scrH^i_X\boxtimes\scrH^j_Y\to \scrH^{i+j}_{X\times Y}$ induced by this is defined by relative cross products. It thus suffices to show that this morphism describes the natural morphism $\dual_X\boxtimes^L\dual_Y\to \dual_{X\times Y}$.

 Since $\Delta_X^\bullet$ and $\Delta_Y^\bullet$ are complexes of flat sheaves, the natural isomorphisms $\Delta_X^\bullet\xrightarrow{\cong} \dual_X$ and $\Delta_Y^\bullet\xrightarrow{\cong} \dual_Y$ define an isomorphism
\[
\Delta_X^\bullet\boxtimes\Delta_Y^\bullet\xrightarrow{\cong} \dual_X\boxtimes^L\dual_Y \ .
\]
To finish the proof, we must show that the diagram
\begin{center}
\begin{tikzpicture}[auto]
\matrix[matrix of math nodes, row sep= 5ex, column sep= 4em, text height=1.5ex, text depth= .25ex]{
|(DelDel)| \Delta^\bullet_X\boxtimes \Delta^\bullet_Y 	&
|(Del)|	\Delta^\bullet_{X\times Y}	\\
|(DualDual)| \dual_X\boxtimes^L\dual_Y	&	
|(Dual)| \dual_{X\times Y}	\\
};
\begin{scope}[->,font=\footnotesize]
\draw (DelDel) -- (Del);
\draw (DualDual) --(Dual);
\draw (DelDel) --node{$\cong$} (DualDual);
\draw (Del) --node{$\cong$}  (Dual);
\end{scope}
\end{tikzpicture}
\end{center}
is commutative. By the universal property of $\dual_{X\times Y}$, we can apply $R\Gamma_c$ and need to show that the two morphisms $C_{-\bullet}(X)\otimes_\Z C_{-\bullet}(Y) \to \Z$ in the diagram
\begin{center}
\begin{tikzpicture}[auto]
\matrix[matrix of math nodes, row sep= 5ex, column sep= 2.5em, text height=1.5ex, text depth= .25ex]{
|(CC)| C_{-\bullet}(X)\otimes_\Z C_{-\bullet}(Y) &
|(DelDel)| R\Gamma_c(\Delta^\bullet_X\boxtimes \Delta^\bullet_Y) 	&
|(Del)|	C_{-\bullet}(X\times Y)	&
\\
|(RDualRDual)| R\Gamma_c\dual_X\otimes^L_\Z R\Gamma_c\dual_Y &
|(DualDual)| R\Gamma_c(\dual_X\boxtimes^L\dual_Y)	&	
|(Dual)| R\Gamma_c(\dual_{X\times Y})	&
|(Z)| \Z  \ ,\\
};
\begin{scope}[->,font=\footnotesize]
\draw (CC)--node{$\cong$}(DelDel);
\draw (DelDel) -- (Del);
\draw (RDualRDual)--node{$\cong$}(DualDual);
\draw (DualDual) --(Dual);
\draw (Dual)-- node{$\int_{X\times Y}$}(Z);
\draw (CC)--node{$\cong$} (RDualRDual);
\draw (DelDel) --node{$\cong$} (DualDual);
\draw (Del) --node{$\cong$}  (Dual);
\end{scope}
\end{tikzpicture}
\end{center}
where the leftmost horizontal isomorphisms use the K\"unneth formula \cite[VII-2.7]{Iversen}. Note that the left square of the diagram is commutative by the functoriality of the K\"unneth formula, coincide. By construction, the composite of the two morphisms in the top row is the Eilenberg-Zilber map. In particular the morphism $C_{-\bullet}(X)\otimes_\Z C_{-\bullet}(Y) \to \Z$ obtained by moving clockwise through the diagram assigns $1$ to a pure tensor $[x]\otimes[y] \in C_0(X)\otimes_\Z C_0(Y)$ of $0$-simplices (i.e.\ a point in $X\times Y$). On the other hand, by the definition of the natural morphism $\dual_X\boxtimes^L\dual_Y\to\dual_{X\times Y}$, the composite of the three morphisms in the lower row is the tensor product of the traces on $X$ and $Y$. By the definition of the morphisms $\Delta_X^\bullet\to \dual_X$ and $\Delta_Y^\bullet\to\dual_Y$ it follows that the morphism $C_{-\bullet}(X)\otimes_\Z C_{-\bullet}(Y) \to \Z$ obtained by moving counterclockwise through the diagram is the tensor product of the two augmentations $C_{-\bullet}(X)\to \Z$ and $C_{-\bullet}(Y)\to\Z$ defined by the degree of $0$-cycles. This product also assigns to $1$ to any pure tensor $[x]\otimes[y]\in C_0(X)\otimes C_0(Y)$. This finishes the proof
\end{proof}

\begin{rem}
In the proof of Lemma \ref{lem:Eilenberg-Zilber}, we did not use the fact that the Eilenberg-Zilber map $C_\bullet(X)\otimes C_\bullet(Y) \to C_\bullet(X\times Y)$ is a chain homotopy equivalence. If one incorporates this into the proof of the lemma carefully, then it also shows that the natural morphism $\dual_X\boxtimes^L\dual_Y\to\dual_{X\times Y}$ is an isomorphism.
\end{rem}

To construct the tropical cross product we also need to relate the sheaves of tropical forms on $X\times Y$ with the sheaves of tropical forms on the factors. The projections $p_X$ and $p_Y$ induce morphisms 
\[
p_X^\sharp\colon p^{-1}_X\Omega^*_X \to \Omega^*_{X\times Y} \quad , \quad 
p_Y^\sharp\colon p^{-1}_Y\Omega^*_Y \to \Omega^*_{X\times Y}
\]
of sheaves of skew-commutative graded rings. These morphisms induce a morphism
\begin{equation*}
p^\sharp_X\otimes p^\sharp_Y \colon \Omega^*_X \boxtimes \Omega^*_Y \to \Omega^*_{X\times Y} 
\end{equation*}
of sheaves of skew-commutative graded rings, where we view $\Omega^*_X \boxtimes \Omega^*_Y$ as the skew tensor product (i.e. the usual tensor product of $\Z$-algebras with a slightly modified multiplication to make it skew-symmetric (cf.\ \cite[p. 571]{Eisenbud})) of $p_X^{-1}\Omega^*_X$ and $p_Y^{-1}\Omega^*_Y$.

\begin{lemma}
\label{lem:differentials on product}
The morphism
\[
p^\sharp_X\otimes p^\sharp_Y \colon\Omega^*_X \boxtimes \Omega^*_Y \to \Omega^*_{X\times Y} 
\]
is an isomorphism.
\end{lemma}

\begin{proof}
This is obvious in degree $1$: working in charts this comes down to the facts that the linear span of a product is the product of the linear spans and that the dual of a direct sum of lattices is the direct sum of the duals. Because the exterior product of a sum is the skew tensor product of the exterior products of the summands, and everything commutes with pullbacks, we obtain an isomorphism
\begin{equation*}
\bigwedge\nolimits^* \Omega^1_X \boxtimes\bigwedge\nolimits^* \Omega^1_Y \xrightarrow{\cong} \bigwedge\nolimits^* \Omega^1_{X\times Y} 
\end{equation*}
induced by $p_X$ and $p_Y$. What is left to show is that if $\alpha\boxtimes \beta$ vanishes on $(X\times Y)^{\max}$, then either $\alpha$ vanishes on $X^{\max}$ or $\beta$ vanishes on $Y^{\max}$. Assume the opposite. Then there exists a point $x\in X^{\max}$ at which $\alpha$ is nonzero and a point $y\in Y^{\max}$ at which $\beta$ is nonzero. But since the stalks of $\bigwedge\nolimits^*\Omega_X^1$ are all free, this implies that $\alpha\boxtimes\beta$ is nonzero an $(x,y)$, which is a point in $(X\times Y)^{\max}$,  a contradiction.
\end{proof}

Since $p_X^\sharp\otimes p_Y^\sharp$ is an isomorphism, we obtain a canonical splitting of the inclusion $\Omega^p_X\boxtimes \Omega^{p'}_Y \to \Omega^{p+p'}_{X\times Y}$ for every $p, p'\in \N$. Having established this we are ready to define the cross-product in tropical Borel-Moore homology. 

\begin{defn}
\label{def:cross product}
Let $X$ and $Y$  be rational polyhedral spaces, and let $\alpha\in H^{BM}_{p,q}(X)$ and $\beta\in H^{BM}_{p',q'}(Y)$. Then the composite 
\begin{equation*}
\Omega^{p+p'}_{X\times Y}[q+q'] \to \Omega^p_X[q]\boxtimes \Omega^{p'}_Y[q'] \xrightarrow{ \alpha \boxtimes^L \beta} \dual_X\boxtimes^L\dual_Y \cong \dual_{X\times Y} \ ,
\end{equation*}
where the leftmost morphism is the natural splitting of $\Omega^p_X\boxtimes \Omega^{p'}_Y\to \Omega^{p+p'}_{X\times Y}$, defines an element $\alpha\times \beta\in H^{BM}_{p+p',q+q'}(X\times Y)$, the \emph{cross product} of $\alpha$ and $\beta$. This defines a graded bilinear morphism
\begin{equation*}
\times\colon H^{BM}_{*,*}(X)\otimes H^{BM}_{*,*}(Y) \to H^{BM}_{*,*} (X\times Y) \ .
\end{equation*} 
\end{defn}

As both the identification $\dual_X\boxtimes^L\dual_Y\cong \dual_{X\times Y}$ and the pull-back of tropical forms is functorial, the same is true for cross-products. In other words, if $f\colon X\to X'$ and $g\colon Y\to Y'$ are proper morphisms of rational polyhedral spaces, then
\begin{equation*}
f_*(\alpha)\times g_*(\beta)=(f\times g)_*(\alpha\times \beta)
\end{equation*}
for all $\alpha\in H^{BM}_{*,*}(X)$ and $\beta\in H^{BM}_{*,*}(Y)$.

\begin{thm}[Tropical K\"unneth theorem]
\label{thm:Kuenneth}
Let $X$ and $Y$ be rational polyhedral spaces, and assume that $H^p_c(X, \Omega_X^q)$ is finitely generated for all $p,q\in \N$. Then for every $p,q\in \N$, there is a natural decomposition $H^{BM}_{p,q}(X\times Y)= \bigoplus_{i+j=p} A_{i,j,q}$,  where
\[
A_{i,j,q}\cong \Hom_{\der{X\times Y}}(\Omega_X^i\boxtimes \Omega_Y^j[q], \dual_{X\times Y}) ,
\]
and for each $i,j\in \N$ there is a short exact sequence
\begin{multline*}
0\to \bigoplus_{k+l=q} H^{BM}_{i,k}(X)\otimes_\Z H^{BM}_{j,l}(Y) \to A_{i,j,q} \to \\
\to \bigoplus_{k+l=q-1} \Tor^\Z_1(H^{BM}_{i,k}(X),H^{BM}_{j,l}(Y)) \to 0 \ ,
\end{multline*}
where the first morphism is given by the cross product.
\end{thm}

\begin{proof}
By Lemma \ref{lem:differentials on product}, there is a natural decomposition
\[
\Omega^{p}_{X\times Y}\cong \bigoplus_{i+j=p} \Omega^i_X \boxtimes \Omega^j_Y \ ,
\]
which induces a decomposition 
\begin{multline*}
H^{BM}_{p,q}(X\times Y)= \Hom_{\der{X\times Y}}(\Omega^p_{X\times Y}[q]  , \dual_{X\times Y}) =\\
=\bigoplus_{i+j=p} \Hom_{\der{X\times Y}}(\Omega^i_X \boxtimes \Omega^j_Y[q],\dual_{X\times Y})
\end{multline*}
which equals $\bigoplus_{i,j} A_{i,j,q}$ if we set
\[
A_{i,j,q}= \Hom_{\der{X\times Y}}(\Omega_X^i\boxtimes \Omega_Y^j[q], \dual_{X\times Y}) \ .
\]
Using the universal property of the dualizing complex, we see that this equals
\[
\Hom_{\der{}} (R\Gamma_c (\Omega_X^i\boxtimes \Omega_Y^j)[q], \Z)\cong H^{-q} R\Hom(R\Gamma_c (\Omega_X^i\boxtimes \Omega_Y^j), \Z) \ .
\]
By a version of the K\"unneth theorem \cite[VII 2.7]{Iversen}, there is a natural isomorphism
\[
R\Gamma_c(\Omega_X^i\boxtimes \Omega_Y^j)\cong R\Gamma_c\Omega_X^i\otimes_\Z^L R\Gamma_c\Omega_Y^j \ ,
\]
and hence we have
\[
R\Hom(R\Gamma_c (\Omega_X^i\boxtimes \Omega_Y^j), \Z)\cong R\Hom(R\Gamma_c\Omega_X^i\otimes_\Z^L R\Gamma_c\Omega_Y^j,\Z) \ .
\]
By assumption, the cohomology groups of $R\Gamma_c\Omega_X^i$ are finitely generated, so the natural morphism
\[
R\Hom(R\Gamma_c\Omega_X^i,\Z)\otimes_\Z^L R\Hom(R\Gamma_c\Omega_Y^j,\Z) \xrightarrow{\cong}
R\Hom(R\Gamma_c\Omega_X^i\otimes_\Z^L R\Gamma_c\Omega_Y^j,\Z) 
\]
is an isomorphism. Combining these isomorphisms, we obtain an isomorphism
\[
A_{i,j,q}\cong H^{-q}\bigg( R\Hom(R\Gamma_c\Omega_X^i,\Z)\otimes_\Z^L R\Hom(R\Gamma_c\Omega_Y^j,\Z)\bigg) \ .
\]
So by the K\"unneth theorem for complexes \cite[Theorem 3.6.3]{Weibel95} and the identifications (see Remark \ref{rem:homology using Verdier duality})
\begin{align*}
H^{-k}R\Hom(R\Gamma_c\Omega_X^i,\Z)&\cong H^{BM}_{i,k}(X) \quad\text{, and} \\
H^{-l}R\Hom(R\Gamma_c\Omega_Y^j,\Z)&\cong H^{BM}_{j,l}(X) \quad ,
\end{align*}
we obtain the desired short exact sequence. It remains to show that the first map in this sequence is given by the cross product. The map in the exact sequence is the given by the composites
\begin{multline*}
\Hom_{\der{}}(R\Gamma_c\Omega_X^i[k],\Z)\otimes_\Z \Hom_{\der{}}(R\Gamma_c\Omega_Y^j[l],\Z)\to\\
\to \Hom_{\der{}}(R\Gamma_c\Omega_X^i[k]\otimes_\Z^L R\Gamma_c\Omega_Y^j[l],\Z)\cong \\
\cong \Hom_{\der{}}(R\Gamma_c(\Omega^i_X[k]\boxtimes\Omega^j_Y[l]),\Z)
\end{multline*}
whereas by Definition \ref{def:cross product}, the cross product $H_{i,k}^{BM}(X)\otimes_\Z H_{j,l}^{BM}(Y) \to A_{i,j,q}$ is given by the composite
\begin{multline*}
\Hom_{\der X}(\Omega_X^i[k],\dual_X)\otimes_\Z \Hom_{\der Y}(\Omega_Y^j[l], \dual_Y)\to \\
\to \Hom_{\der{X\times Y}} (\Omega_X^i[k]\boxtimes \Omega_Y^j[l],\dual_X\boxtimes^L \dual_Y)\to\\
\to \Hom_{\der{X\times Y}}(\Omega_X^i[k]\boxtimes\Omega_Y[l], \dual_{X\times Y}) \ .
\end{multline*}
The domains and codomains of these composites are naturally isomorphic by the universal property of the dualizing complexes. Showing that these natural isomorphisms are compatible with the morphisms requires a quick diagram chase that eventually boils down to the facts that the K\"unneth isomorphism
\[
R\Gamma_c(\mathcal C^\bullet \boxtimes^L \mathcal D^\bullet )\cong R\Gamma_c\mathcal C^\bullet \otimes_\Z^L R\Gamma_c\mathcal D^\bullet
\]
is functorial in both $\mathcal C^\bullet $ and $\mathcal D^\bullet$, and that the morphism in 
\[
\Hom_{\der{}}(R\Gamma_c(\dual_X \boxtimes^L \dual_Y),\Z)
\]
 corresponding to the natural isomorphism $\dual_X\boxtimes^L\dual_Y\to \dual_{X\times Y}$ involved in the definition of the cross product is defined via the K\"unneth isomorphism.
\end{proof}

\subsection{Cup and Cap Products}
Let $X$ be a rational polyhedral space. Since $\Omega^*_X$ is a sheaf of rings, its cohomology group has a ring structure again, the multiplication being the so-called \emph{cup product}. If 
\begin{align*}
(\Z_V\xrightarrow{c} \Omega_X^p[q])&\in H^{p,q}_V(X) \ \;\text{, and} \\
(\Z_W\xrightarrow{d} \Omega_X^{p'}[q'])&\in H^{p',q'}_V(X) \ ,
\end{align*}
where $p,p',q,q'$ are integers and $V$ and $W$ are locally polyhedral subsets of $X$, then their cup product 
\[
c\smile d\in H^{p+p',q+q'}_{V\cap W}(X)
\]
is represented by the composite
\[
\Z_{V\cap W} \xrightarrow{\cong} \Z_V\otimes_{\Z_X} \Z_W\xrightarrow{c\otimes d} \Omega^p_X\otimes_{\Z_X} \Omega^{p'}_X [q+q']\to \Omega^{p+p'}_X[q+ q'] \ .
\]
Here, the last morphism is the product on $\Omega_X^*$, and the morphism in the middle can be obtained using the the fact that $\Omega_X^p\otimes_{\Z_X}\Omega_X^{p'}\cong \Omega_X^p\otimes^L_{\Z_X}\Omega_X^{p'}$ because $\Omega^*_X$ is flat and the functoriality of the derived tensor product.
It follows directly from the associativity property of the sheaf of rings $\Omega^*_X$  that the cup product on $H^{*,*}(X)$ is associative. It is also unital, the unity being represented by the identity map on $\Z_X\to \Z_X=\Omega^0_X$. It is clear from the construction that the restriction of the cup product to $H^{0,*}(X)$ is the classical cup product on the cohomology of $X$ (cf.\ Lemma \ref{lem:basicBM} (a) and \cite{Iversen}[II.9.9]).  

\begin{prop}
\label{prop:pullback is ring morphism}
Let $f\colon X\to Y$ be a morphism of rational polyhedral spaces. Then the pull-back (defined in \S\ref{subsec:pullback})
\[
f^*\colon H^{*,*}(Y)\to H^{*,*}(X)
\]
is a ring homomorphism.
\end{prop}

\begin{proof}
Examining the definitions of cup products and pull-backs, we see that this directly follows from the fact that the pull-back $f^\sharp\colon f^{-1}\Omega^*_Y\to \Omega^*_X$ of tropical forms is a morphism of sheaves of graded rings.
\end{proof}

Similarly as the cup product, the {\em cap product} also generalizes from the classical to the tropical setting. To define it, let 
\begin{align*}
(\Z_V\xrightarrow{c} \Omega_X^i[j])&\in H^{i,j}_V(X) \ \;\text{, and} \\
( (\Omega_X^{p})_W[q]\to \dual_X)&\in H^{BM}_{p,q}(W,X) \ ,
\end{align*}
where $p,q,i,j$ are integers and $V$ and $W$ are locally polyhedral subsets of $X$. Then the \emph{cap product}
\[
c\frown \alpha \in H^{BM}_{p-i,q-j}(V\cap W,X)
\]
is represented by the composite
\begin{multline*}
(\Omega^{p-i}_X)_{V\cap W}[q-j]\xrightarrow{\cong} \Z_V\otimes_{\Z_X}(\Omega^{p-i}_X)_W[q-j]\xrightarrow{c\otimes\id} \Omega^i_X\otimes_{\Z_X}(\Omega^{p-i}_X)_W[q]\to\\
 \to (\Omega^p_X)_W[q]\xrightarrow{\alpha} \dual_X \ ,
\end{multline*} 
where the second arrow can be defined using the functoriality of the derived tensor product, and the third morphism is the product on $\Omega^*_X$. Again by the associativity property of the sheaf of rings $\Omega^*_X$, the cap product makes $H^{BM}_{*,*}(X)$ a left $H^{*,*}(X)$-module. It is clear from the construction that the restriction
\[
H^{0,*}(X)\times H^{BM}_{0,*}(X)\to H^{BM}_{0,*}(X)
\]
of the cap product is the cap product in classical Borel-Moore homology (cf.\ Lemma \ref{lem:basicBM} (a) and \cite[IX.3]{Iversen}).

\begin{prop}[Projection formula]
\label{prop:projection formula}
Let $f\colon X\to Y$ be a proper morphism of rational polyhedral spaces, let $\alpha\in H^{BM}_{p,q}(X)$ and $c\in H^{i,j}(Y)$. Then we have the equality
\[
f_*(f^*c\frown \alpha)= c\frown f_*\alpha 
\]
in $H^{BM}_{p-i,q-j}(Y)$.
\end{prop}

\begin{proof}
Both sides of the equation correspond to a morphism 
\[
\Omega^{p-i}_Y[q-j]\to \dual_Y \ .
\]
More precisely, the left side of the equation corresponds to the morphism obtained by moving counterclockwise along the square in the diagram below, whereas the right side of the equation corresponds to the morphism obtained by moving clockwise along the square.
\begin{center}
\begin{tikzpicture}[auto]
\node (YOpi) at  (0,0) {$\Omega_Y^{p-i}[q-j]$}	;
\node (YOiOpi) at (5,0) {$\Omega_Y^{i}\otimes_{\Z_Y} \Omega_Y^{p-i}[q]$}	;
\node (YOp) at (10,0) {$\Omega^{p}_{Y}[q]$} ;
\node (XOpi) at (0,-1.5) {$Rf_*\Omega_X^{p-i}[q-j]$}	;
\node (YOiXOpi) at (0,-3) {$Rf_*(f^{-1}\Omega_Y^i\otimes_{\Z_X} \Omega_X^{p-i})[q]$};
\node (XOiOpi) at (5,-3) {$Rf_*(\Omega_X^i\otimes_{\Z_X} \Omega_X^{p-i})[q]$};
\node (XOp) at (10,-3) {$Rf_*\Omega_X^p[q]$} ;
\node (DX) at (5,-5)  {$Rf_*\dual_X$};
\node (DY) at (10,-5) {$\dual_Y$};

\begin{scope}[->,font=\footnotesize]
\draw (YOpi) --node{$c\otimes\id$} (YOiOpi);
\draw (YOiOpi) -- (YOp);
\draw (XOpi) --node{$Rf_*(f^{-1}c\otimes\id)$} (YOiXOpi);
\draw (YOiXOpi) -- (XOiOpi);
\draw (XOiOpi)--(XOp);
\draw (DX)--(DY);
\draw (YOpi)--(XOpi);
\draw (YOp)--(XOp);
\draw (XOp)--node{$Rf_*\alpha$}(DX);

\end{scope}
\end{tikzpicture}
\end{center}
It thus suffices to show that the square commutes. Using the fact that $f^{-1}$ and $Rf_*$ are adjoints, this boils down to the fact that $f^\sharp\colon f^{-1}\Omega^*_Y\to \Omega^*_X$ is a morphism of sheaves of rings.
\end{proof}

\subsection{Comparison with singular tropical homology} 
\label{subsec:Comparison of homologies}

Let $X$ be a rational polyhedral space, and let $\Sigma$ be a face structure on $X$.
We say that a singular simplex $\sigma\colon\Delta^q\to X$ (where $\Delta^q$ denotes the standard $q$-simplex) respects the face structure $\Sigma$ if for every face $\Theta\subseteq \Delta^q$ there exists a polyhedron $P\in\Sigma$ such that $\sigma(\relint(\Theta))\subseteq P$. A \emph{tropical $(p,q)$-simplex (with respect to $\Sigma$)}, is a pair $(\sigma, s)$, where $\sigma\colon \Delta^q\to X$ is a singular $q$-simplex respecting the face structure $\Sigma$ and $s\in \Hom(\Omega_X^p|_{\sigma(\Delta^q)},\Z_{\sigma(\Delta^q)})$. We denote by $C_{p,q}(X;\Sigma)$ the free abelian group generated by tropical $(p,q)$-simplices (w.r.t.\ $\Sigma$).
If $(\sigma,s)$ is a tropical $(p,q)$-simplex, then pulling back $(\sigma,s)$ along the $i$-th face morphism $\delta^{q,i}\colon\Delta^{q-1}\to \Delta^q$ yields a tropical $(p,q-1)$-simplex $\partial_{p,q,i}(\sigma, s)=(\sigma\circ \delta^{q,i}, s|_{\sigma(\delta^{q,i}(\Delta^{q-1}))})$. Extending $\partial_{p,q,i}$ by linearity and taking alternating sums, one defines the differentials 
\[
\partial_{p,q} =\sum_{i=0}^q (-1)^i\partial_{p,q,i}
\]
and obtains a chain complexes $C_{p,\bullet}(X;\Sigma)$.
We call their homology groups
\[
H^\sing_{p,q}(X;\Sigma)\coloneqq H_q(C_{p,\bullet}(X;\Sigma))
\]
 the \emph{singular tropical homology groups of $X$ (with respect to $\Sigma$)}. They agree with the tropical homology groups introduced in \cite{TropHomology}. It is well-known that they do not depend on $\Sigma$, which will also follow from Theorem \ref{thm:compatibility of homology theories}.

Allowing locally finite chains, that is infinite sums of tropical simplices such that every point has a neighborhood intersecting only finitely many of them, instead of only finite chains we obtain chain complexes $C^{lf}_{p,\bullet}(X;\Sigma)$ whose homology groups
\[
H^{lf}_{p,q}(X;\Sigma)=H_q(C^{lf}_{p,\bullet}(X;\Sigma))
\]
are the \emph{locally finite tropical homology groups of $X$ (with respect to $\Sigma$}). They agree with with  tropical homology groups studied in \cite{Lefschetz}. Again, it will follow from Theorem \ref{thm:compatibility of homology theories} that they are independent of the face structure $\Sigma$.

\begin{thm}
\label{thm:compatibility of homology theories}
Let $X$ be a rational polyhedral space equipped with a face structure $\Sigma$. Then there are natural isomorphisms
\begin{align*}
H^{lf}_{p,q}(X;\Sigma)&\cong H^{BM}_{p,q}(X) \quad \text{, and} \\
H^\sing_{p,q}(X;\Sigma)&\cong H_{p,q}(X) \ .
\end{align*}
\end{thm}

\begin{proof}
The face structure $\Sigma$ defines an admissible stratification on the space $X$ in the sense of Definition~\ref{def:admis_strat}.
By Proposition \ref{prop:dualizing complex using singular cycles respecting the stratification}, the subcomplex $\Delta_X^{\Sigma,\bullet}$ of $\Delta_X^\bullet$ consisting of chains respecting $\Sigma$ (we refer to Appendix \ref{sec:appendix} for a precise definition of $\Delta_X^{\Sigma,\bullet}$) is quasi-isomorphic to $\Delta_X^\bullet$.
Furthermore, by Proposition \ref{prop:Hom is RHom}, there is a natural quasi-isomorphism
\[
\Homs^\bullet(\Omega^p_X, \Delta_X^{\Sigma,\bullet})\xrightarrow{\cong} R\Homs^\bullet(\Omega^p_X, \dual_X) \ .
\]
Taking hypercohomology with closed/compact supports we obtain natural isomorphisms
\begin{align*}
\Hyper^{-q}\Homs^\bullet(\Omega^p_X, \Delta_\mS^\bullet)&\xrightarrow{\cong} \Hyper^{-q} R\Homs^\bullet(\Omega^p_X, \dual_X) = H^{BM}_{p,q}(X) \ ,\text{ and}\\
\Hyper_c^{-q}\Homs^\bullet(\Omega^p_X, \Delta_\mS^\bullet)&\xrightarrow{\cong} \Hyper_c^{-q}R\Homs^\bullet(\Omega^p_X, \dual_X) =H_{p,q}(X) \ .
\end{align*}
Since $\Delta_X^{\Sigma,\bullet}$ is homotopically fine, the same is true for $\Homs^\bullet(\Omega^p_X,\Delta_X^{\Sigma,\bullet})$. It follows that the natural morphisms
\begin{align*}
H^{-q}\Homs^\bullet(\Omega^p_X,\Delta_X^{\Sigma,\bullet})&\to\Hyper^{-q}\Homs^\bullet(\Omega^p_X,\Delta_X^{\Sigma,\bullet}) \ \text{ and}\\
H_c^{-q}\Homs^\bullet(\Omega^p_X,\Delta_X^{\Sigma,\bullet})&\to\Hyper_c^{-q}\Homs^\bullet(\Omega^p_X,\Delta_X^{\Sigma,\bullet})
\end{align*}
are isomorphisms. Note that for each $i\in \N$ there is a morphism 
\[
\bigoplus_{\sigma} \Z_{\sigma(\Delta^{i})}\to \Delta_X^{\Sigma,-i} \ ,
\]
where the sum is taken over all singular $i$-simplices $\sigma\colon \Delta^{i}\to X$ respecting $\Sigma$, that sends the generator of $\Z_{\sigma(\Delta^{i})}$ to the global section represented by $\sigma\in C_{i}(X;\Sigma)=C_{i}(X,X\setminus X;\Sigma)$ (see \eqref{eq:relchains}). This morphism is in fact an isomorphism: the stalk of both $\bigoplus_{\sigma}\Z_{\sigma(\Delta^i)}$ and $\Delta^{\Sigma,-i}_X$ at $x\in X$ is isomorphic to the free abelian group on singular $i$-simplices that respect $\Sigma$ and contain $x$ in their image. Therefore, if we write  $\iota_\sigma\colon\sigma(\Delta^{i})\to X$ for the inclusion , we have
\begin{multline*}
\Homs\left(\Omega^p_X, \Delta_X^{\Sigma,-i}\right)
\cong \Homs\left(\Omega^p_X, \bigoplus_{\sigma}\Z_{\sigma(\Delta^{i})}\right) \cong \\
\cong \bigoplus_{\sigma}\Homs\left(\Omega^p_X,\Z_{\sigma(\Delta^{i})}\right)
\cong \bigoplus_\sigma (\iota_\sigma)_*\Homs \left(\Omega^p_X|_{\sigma(\Delta^i)}, \Z_{\sigma(\Delta^i)}\right) ,
\end{multline*}
where the direct sum commutes with $\Homs$ because $\Omega^p_X$ is constructible. 
The group of global sections of this last sheaf is precisely $C^{lf}_{p,i}(X;\Sigma)$, and the group of its global sections with compact support is precisely $C_{p,i}(X;\Sigma)$. Leaving the straightforward check that this identification commutes with the differentials to the reader, we obtain an isomorphism
\begin{align*}
\Gamma\left(\Homs(\Omega^p_X,\Delta^{\Sigma,\bullet}_X)\right)&\cong C_{p,-\bullet}^{lf}(X;\Sigma) \ \;  \text{ and} \\
\Gamma_c\left(\Homs(\Omega^p_X,\Delta^{\Sigma,\bullet}_X)\right)&\cong C_{p,-\bullet}(X;\Sigma) 
\end{align*}
of cochain complexes of abelian groups. Taking their $(-q)$-th cohomology finishes the proof.
\end{proof}

\section{The tropical cycle class map}
\label{sec:cycle class map}

\subsection*{Conventions for orientations} 
To define the tropical cycle class map one needs to make a choice regarding orientations. There are two ways of defining an orientation for $\R^n$ that are relevant for us, one being the choice of a generator for $\bigwedge ^n T_0^\Z\R^n\cong \bigwedge^n\Z^n$, the other being the choice of a generator for $H_n(\R^n,\R^n\setminus\{0\})$. For the construction of the tropical cycle class map we need to choose, once and for all, an isomorphism 
\[
\bigwedge\nolimits^n T_0^\Z\R^n \xrightarrow{\cong} H_n(\R^n,\R^n\setminus\{0\})
\]
that allows us to compare these two notions of orientation, and our choice will be the one that sends $e_1\wedge \ldots \wedge e_n$ to class of the (linearly embedded) singular simplex $[e_n,e_{n-1},\ldots ,e_1,e_0]$, where $e_1,\ldots, e_n$ is any basis for $\Z^n$ and $e_0=-\sum_{i=1}^n e_i$.

Suppose that $\sigma$ is a $k$-dimensional polyhedron in $\R^n$ and  $\tau\subseteq \sigma$ is one of its  faces of codimension $1$. Assume we have chosen orientations $\eta_\sigma\in \bigwedge^{k}T^\Z(\sigma)$ and $\eta_\tau\in \bigwedge^{k-1}T^\Z(\tau)$. With our convention in place, cellular homology provides us with classes of locally finite cycles
\begin{align*}
[\sigma]\in H^{lf}_{k}(\sigma,\partial \sigma) \  \text{ and }\ 
[\tau]\in H^{lf}_{k-1}(\tau,\partial \tau)  \ .
\end{align*}
If $\partial_\tau$ denotes the composite 
\[
H^{lf}_{k}(\sigma,\partial \sigma) \to H^{lf}_{k-1} (\partial\sigma) \to H^{lf}_{k-1}(\partial\sigma,\partial\sigma\setminus\relint(\tau)) \cong  H^{lf}_{k-1}(\tau,\partial\tau) ,
\]
and $n_{\sigma/\tau}\in T^\Z(\sigma)$ is any lattice normal vector of $\sigma$ with respect to $\tau$, 
the equalities $\partial_\tau[\sigma]=[\tau]$  and $\eta_\sigma=\eta_\tau\wedge n_{\sigma/\tau}$ are equivalent. In this case, we say that the chosen orientations on $\sigma$ and $\tau$ are compatible.

\subsection{Tropical cycles as sheaf hom}
The following crucial observation, together with Lemma~\ref{lem:basicBM}~(b), will allow us define the tropical cycle class map.

\begin{prop}
\label{prop:canonical morphism for between n-cycles and hom}
Let $X$ be an $n$-dimensional rational polyhedral space. Then there is a natural isomorphism of sheaves
\[
\mathcal Z^X_n \cong \Homs(\Omega^n_X,\scrH_X^n) \ .
\]
\end{prop}

\begin{proof}
We will first define the isomorphism locally around a point $x\in X$. Let $\Sigma$ be a local face structure at $x$. After potentially shrinking $\Sigma$, we may assume that $|\Sigma|$ is compact, in which case $\Sigma$ gives a CW-complex structure to the neighborhood $|\Sigma|$ of $x$. After choosing orientations $\eta_\sigma\in\bigwedge^{\dim(\sigma)}T^\Z(\sigma)$ for all $\sigma\in\Sigma$, cellular homology provides a description of  $\scrH_{X,x}^n=H_n(X,X\setminus \{x\})$ as a subgroup of the free group on $\Sigma(n)$, where $\Sigma(i)$ denotes the set of $i$-dimensional cells of $\Sigma$. Namely, it is the group of all weights $w\colon \Sigma(n)\to \Z$ such that for every $(n-1)$-dimensional cell $\tau\in\Sigma$ containing $x$ we have
\begin{equation} 
\label{equ:cycle condition}
\sum_{\sigma\colon \tau\subseteq\sigma\in\Sigma(n)} \epsilon_{\sigma/\tau} w(\sigma) = 0 \ , 
\end{equation}
where $\epsilon_{\sigma/\tau}$ is either $1$ or $-1$, depending on whether the chosen orientations on $\sigma$ and $\tau$ agree or not. If $y$ is a point in the interior of $|\Sigma|$, then the set
\[
\Sigma_y=\{\sigma\in\Sigma \mid \text{ there exists } \tau\in\Sigma \text{ so that } \sigma\subseteq \tau \text{ and }y\in \tau\}
\]
is a local face structure at $y$, so with the same reasoning we conclude that $\scrH_{X,y}^n$ is the group of weights $w\colon \Sigma_y(n)\to \Z$ satisfying the condition displayed in (\ref{equ:cycle condition}) for all $(n-1)$-dimensional cells $\tau\in\Sigma$ containing $y$. This description is continuous in the sense that we obtain an exact sequence 
\begin{equation*}
0\to\scrH^n_X \to \bigoplus_{\sigma\in\Sigma(n)} \Z_\sigma \to \bigoplus_{\tau:x\in\tau\in \Sigma(n-1)} \Z_\tau
\end{equation*}
on the interior of $|\Sigma|$, where the map to the right is determined by the condition (\ref{equ:cycle condition}). Note that this is precisely the description of $\scrH^n_X$ one obtains from Shepard's combinatorial description of the dualizing complex mentioned in Remark \ref{rem:combinatorial dualizing complex}. Applying $\Homs(\Omega^n_X,-)$ to the sequence, we obtain an exact sequence
\begin{multline*}
0\to\Homs(\Omega^n_X,\scrH^n_X) \to 
\bigoplus_{\sigma\in\Sigma(n)} \iota_{\sigma *}\Homs(\Omega^n_X|_\sigma,\Z_\sigma) \to  \\
\to \bigoplus_{\tau:x\in\tau\in \Sigma(n-1)} \iota_{\tau *}\Homs(\Omega^n_X|_\tau,\Z_\tau) \ ,
\end{multline*}
where $\iota_\delta\colon \delta\to X$ denotes the inclusion for every $\delta\in\Sigma$. By Lemma \ref{lem:Sections of dual}, for every $\delta\in\Sigma$ we have an isomorphism
\[
\Homs(\Omega^n_X|_{\delta},\Z_\delta)\xrightarrow{\cong} (\kappa_\delta)_*\Homs(\Omega^n_X|_{\relint(\delta)},\Z_{\relint(\delta)}) \ ,
\]
where $\kappa_\delta \colon \relint(\delta)\to\delta$ denotes the inclusion, and the latter sheaf is in turn naturally isomorphic to the constant sheaf
\[
\left(\bigwedge\nolimits^n T^\Z_\delta X\right)_\delta
\]
on $\delta$. Here, we denote $T^\Z_\delta X\coloneqq \Gamma(\relint(\delta), \Omega^1_X|_{\relint(\delta)})^*$, which is naturally isomorphic to $T^\Z_x X$ for any $x\in \relint(\delta)$. If $\sigma\in\Sigma(n)$, then $T_\sigma^\Z X$ is isomorphic to $\Z_\sigma$ and is generated by $\eta_\sigma$. We thus obtain an exact sequence
\[
0\to\Homs(\Omega^n_X,\scrH^n_X) \to 
\bigoplus_{\sigma\in\Sigma(n)} \Z_\sigma   
\to \bigoplus_{\tau:x\in\tau\in \Sigma(n-1)} \left(\bigwedge\nolimits^n T_\tau^\Z X \right)_\tau \ ,
\]
where the component of the rightmost morphism going from $\Z_\sigma$ to $\left(\bigwedge\nolimits^n T^\Z(\tau) \right)_\tau$ is $0$ if $\tau$ is a face of $\sigma$ at infinity, and sends the generator $1$ of $\Z_\sigma$ to $\eta_\tau\wedge n_{\sigma/\tau}$ else, where $n_{\sigma/\tau}\in T^\Z(\sigma)$ is any lattice normal vector of $\sigma$ relative to $\tau$. 
This effectively yields a  presentation of $\Homs(\Omega_X^n,\scrH^n_X)$ on a neighborhood of $x$ as the sheaf of locally constant functions $A$ on $X^{\max}$ such that for every codimension-$1$ face $\tau\in\Sigma(n-1)$ we have
\[
\eta_\tau\wedge\left(\sum_{\sigma\in \Sigma_\tau(n)} A(\sigma)n_{\sigma/\tau} =0 \right) \ ,
\]
where $\Sigma_\tau(n)$ denotes the subset of $\Sigma(n)$ consisting of cells that have $\tau$ as a finite face. This equality holds if and only if
\[
\sum_{\sigma\in \Sigma_\tau(n)} A(\sigma)n_{\sigma/\tau} \in T^\Z(\tau) \ ,
\]
which is precisely the balancing condition (see Remark \ref{rem:lattice normal vector and balancing}). In other words, we obtain an isomorphism $\mathcal Z^X_n \cong \Homs(\Omega^n_X,\scrH^n_X)$ in a neighborhood of $x$.

To show that these local isomorphisms glue to a global isomorphism, we essentially need to show that the local isomorphisms are independent of all choices. The choices we made were the local face structure and the orientations on them.  We used the same orientations to pick generators for $\bigwedge^n T^\Z(\sigma)$ and $H_n(\sigma,\partial \sigma)$, so if we picked the opposite orientation on one of the cells $\sigma$, we would change signs twice and hence obtain the same isomorphism. It remains to show that a different local face structure would also provide the same isomorphism. But to compare two different choices of local face structures one can always pass to a common refinement, and it is clear that the construction of the isomorphism is compatible with refinements.
\end{proof}

\begin{rem}
\label{rem:explicit description of natural morphism}
Going through the proof of Proposition \ref{prop:canonical morphism for between n-cycles and hom} we obtain the following description of the isomorphism $\mathcal Z^X_n \cong \Homs(\Omega^n_X, \scrH^n_X)$ locally around a point $x$ using a local face structure $\Sigma$ at $x$ and orientations $\eta_\sigma\in\bigwedge^n T^\Z(\sigma)$ on each of its maximal cells $\sigma\in\Sigma$. The description of the morphisms $\Omega^n_X\to \scrH^n_X$ corresponding to $A\in Z_n(X)$ uses three ingredients:
\begin{enumerate}
\item 
Using the orientations and cellular homology, we can choose for each facet $\sigma\in \Sigma(n)$ of dimension $n$ a chain $[\sigma]$ supported on $\sigma$ such that $\scrH^n_X$ is, locally around $x$, isomorphic to the subsheaf of $\Delta^{-n}_X$ generated by the $[\sigma]$, $\sigma\in\Sigma$.

\item Whenever $\omega\in \Gamma(U,\Omega^n_X)$ is a tropical $n$-form defined on an open subset of the interior of $|\Sigma|$, we can pair it with $\eta_\sigma$ to obtain a locally constant, integer-valued function $\langle \omega, \eta_\sigma\rangle$ on $U\cap \sigma$. 

\item If $A$ is a tropical $n$-cycle defined on the interior of $|\Sigma|$, then it defines a multiplicity $A(\sigma)$ on every $\sigma\in\Sigma(n)$.
\end{enumerate}

With notation as in (1), (2), and (3), the morphism $\Omega^n_X\to \scrH^n_X$ that $A$ is mapped to by the isomorphism can now be described by the rule
\[
\omega \mapsto \sum_{\sigma\in \Sigma(n)} A(\sigma)\langle \omega, \eta_\sigma\rangle [\sigma] \ .
\]
\end{rem}

\subsection{The tropical cycle class map} 
We can now define the tropical cycle class map. We do this in two steps, first for top-dimensional tropical cycles and then in general.

\begin{defn} \label{def:cyctop}
Let $X$ be an $n$-dimensional rational polyhedral space. We define the \emph{tropical cycle class map  on $n$-dimensional tropical cycles} 
\[
\tcyc_X\colon Z_n(X)\to H^{BM}_{n,n}(X)
\] 
as the composite of the canonical map 
\[Z_n(X)\to \Hom(\Omega^n_X,\scrH^n_X)\] 
and the canonical identification
\[
\Hom(\Omega^n_X,\scrH^n_X)
=\Hom_{D(X)}(\Omega^n_X[n],\dual_X)
=H^{BM}_{n,n}(X) \, ,
\]
where the first equality holds since $H^{-j}(\dual_X)=\scrH^j_X=0$ for $j>n$.
\end{defn}

To define the tropical cycle class map in the remaining dimensions we use push-forwards:

\begin{defn} \label{defn:all dimensional cycle map}
Let $X$ be an $n$-dimensional rational polyhedral space. For a tropical cycle $A\in Z_i(X)$, $i\in \N$, we define its \emph{tropical cycle class} by 
\[
\cyc_X(A)\coloneqq 
\iota_*(\tcyc_{|A|}(A)) \in H^{BM}_{i,i}(X) \, ,
\]
where $\iota\colon |A| \hookrightarrow X$ is the inclusion map. Note that $\tcyc_{|A|}(A)$ is defined, since $i=\dim|A|$. 
\end{defn}

\begin{rem}
Using Remark \ref{rem:push-forward with support}, we can (and will) view the tropical cycle class of $A\in Z_i(X)$ as an element of $H^{BM}_{p,q}(|A|,X)$
\end{rem}

For top-dimensional tropical cycles, we now have two ways to take their tropical cycle classes which are a priori different: one by applying $\tcyc$ and one by applying $\cyc$. As a byproduct of the compatibility with push-forwards we will see that they agree (see Corollary \ref{cor: cyc and tcyc coincide}).

\subsection{Compatibility with push-forwards}

As we have seen, both tropical cycle groups and tropical Borel-Moore homology groups are functorial with respect to proper morphisms. We will now show that the tropical cycle class map respects push-forwards, that is that it defines a natural transformation between tropical cycle groups and tropical homology groups.

\begin{prop}
\label{prop:cycle map commutes with push-forwards}
Let $f\colon X\to Y$ be a proper morphism of $n$-dimensional rational polyhedral spaces. Then the tropical cycle class map $\tcyc$ commutes with push-forwards. In other words, the diagram
\begin{center}
\begin{tikzpicture}[auto]
\matrix[matrix of math nodes, row sep= 5ex, column sep= 4em, text height=1.5ex, text depth= .25ex]{
|(ZX)| Z_n(X) 	&
|(ZY)|	Z_n(Y)	\\
|(HX)| H^{BM}_{n,n}(X)	&	
|(HY)| H^{BM}_{n,n}(Y)	\\
};
\begin{scope}[->,font=\footnotesize]
\draw (ZX) --node{$f_*$} (ZY);
\draw (HX) --node{$f_*$} (HY);
\draw (ZX) -- node{$\tcyc_X$} (HX);
\draw (ZY) -- node{$\tcyc_Y$} (HY);
\end{scope}
\end{tikzpicture}
\end{center}
is commutative.
\end{prop}

\begin{proof}
Inspecting the definitions of $\tcyc_X$, $\tcyc_Y$, and the push-forward in homology, we see  that the statement boils down purely formally to proving the commutativity of the diagram
\begin{center}
\begin{tikzpicture}[auto]
\matrix[matrix of math nodes, row sep= 5ex, column sep= 4em, text height=1.5ex, text depth= .25ex]{
|(ZX)| Z_n(X) 	&
|(ZY)|	Z_n(Y)	\\
|(HomX)| \Hom(\Omega^n_X, \scrH^n_X)	&	
|(HomY)| \Hom(\Omega^n_Y,\scrH^n_Y) \ ,	\\
};
\begin{scope}[->,font=\footnotesize]
\draw (ZX) --node{$f_*$} (ZY);
\draw (HomX) -- (HomY);
\draw (ZX) -- (HomX);
\draw (ZY) -- (HomY);
\end{scope}
\end{tikzpicture}
\end{center}
where the vertical maps are induced by the natural isomorphisms of Proposition \ref{prop:canonical morphism for between n-cycles and hom}, and the lower horizontal map sends a morphism $\Omega^n_X\to \scrH^n_X$ to the composite
\begin{equation*}
\Omega^n_Y\to f_*\Omega^n_Y\to f_*\scrH^n_X \to \scrH^n_Y ,
\end{equation*}
where the rightmost morphism is the $(-n)$-th cohomology of the natural morphism $Rf_*\dual_X\to \dual_Y$. By Lemma \ref{lem:description push-forward of dualizing complex}, this morphism is induced by the push-forward morphisms $H_n(X,X\setminus f^{-1}U)\to H_n(Y,U)$ between the singular relative homology groups.
So let $A\in Z_n(X)$. We will compare the two morphisms in $\Hom(\Omega^n_Y,\scrH^n_Y)$ obtained from $A$. Since  tropical $n$-cycles on $Y$ are determined by their restriction to a dense open subset of $Y^{\max}$, it suffices to do this locally at a point $y\in Y^{\max}\setminus f(X\setminus X^{\max})$. Let $\phi\in \Omega^n_{Y,y}$. To compute the image of $\phi$ under the morphism obtained from $A$ by moving along the square counterclockwise, we first take its image $f^*\phi$ in $(f_*\Omega^n_X)_y=\Gamma(f^{-1}\{y\},\Omega_X^n)$. We note that the support of $f^*\phi$ only consists of isolated points of $f^{-1}\{y\}$. Indeed, $f^*\phi$ is nonzero at $x\in f^{-1}\{y\}$, if and only if $d_x f\colon T_x^\Z X\to T^\Z_y Y$ is injective, which is the case if and only if $f$ is injective on a neighborhood of $x$.
In particular, we see that $f^*\phi$ is supported on finitely many points, call them $x_1,\ldots, x_k$, because  $f^{-1}\{y\}$ is compact. Now we take the image of $f^*\phi$ under the morphism 
\[
(f_*\Omega_X^n)_x=\Gamma(f^{-1}\{y\},\Omega_X^n)\to \Gamma(f^{-1}\{y\},\scrH^n_X) =(f_*\scrH_X^n)_y
\]
 induced by $A$. To understand it we use the explicit description of the canonical morphism $Z_n(X)\to \Hom(\Omega^n_X,\scrH^n_X)$ given in Remark \ref{rem:explicit description of natural morphism}, which is particularly simple here because we are working on $X^{\max}$. We pick an orientation $\eta_y\in \bigwedge^n T^\Z_y Y$, which induces orientations $\eta_{x_i}\in \bigwedge^n T^\Z_{x_i} X$ for all $1\leq i\leq k$. These orientations define generators $[\sigma_{x_i}]\in \scrH^n_{X,x_i}$ and $[\delta]\in \scrH^n_{Y,y}$. 
The image of $f^*\phi$ in $\Gamma(f^{-1}\{y\},\scrH_X^n)$ is then represented by 
\[
\sum_{i=1}^k A(x_i) \langle f^*\phi, \eta_{x_i}\rangle [\sigma_{x_i}] \ ,
\]
We observe that
\[
\langle f^*\phi,\eta_{x_i}\rangle = 
[T^\Z_y Y\colon d_{x_i} f (T^\Z_{x_i}X)]\langle\phi,\eta_{y}\rangle \ .
\]
So when we finally apply the morphism 
\[
(f_*\scrH_X^n)_y=\Gamma(f^{-1}\{y\},\scrH_X^n)\to \scrH^n_{Y,y} , 
\]
 we see that
\begin{equation*}
\left(\sum_{i=1}^k [T^\Z_{y}Y\colon d_{x_i} f (T^\Z_{x_i} X)] A(x_i)\right)\langle \phi,\eta_y\rangle [\delta] 
\end{equation*}
is the image of  $\phi$. Looking back at Definition \ref{def:push-forward of cycles}, we see that
this is precisely the image of $\phi$ under the morphism $\Omega^n_{Y,y}\to \scrH^n_{Y,y}$ induced by $f_*A$.
\end{proof}

\begin{cor}
\label{cor: cyc and tcyc coincide}
Let $X$ be an $n$-dimensional rational polyhedral space. Then the two morphisms $\tcyc_X,\cyc_X\colon Z_n(X)\to H^{BM}_{n,n}(X)$ coincide.
\end{cor}
\begin{proof}
Let $A\in Z_n(X)$, and let $\iota\colon |A|\to X$ be the inclusion. The statement is trivially true for $A=0$, so we may assume that $|A|$ is an $n$-dimensional rational polyhedral subspace of $X$. By definition of $\cyc_X$, we have $\cyc_X(A)=\iota_*(\tcyc_{|A|}(A))$, which equals $\tcyc_X(\iota_*A)$ by Proposition \ref{prop:cycle map commutes with push-forwards}. As $\iota_*A=A$, this finishes the proof.
\end{proof}

\begin{cor}
\label{cor:cycle map commutes with push-forwards}
Let $f\colon X\to Y$ be a proper morphism of rational polyhedral spaces. Then the tropical cycle class map $\cyc$ commutes with push-forwards. In other words, the diagram
\begin{center}
\begin{tikzpicture}[auto]
\matrix[matrix of math nodes, row sep= 5ex, column sep= 4em, text height=1.5ex, text depth= .25ex]{
|(ZX)| Z_i(X) 	&
|(ZY)|	Z_i(Y)	\\
|(HX)| H^{BM}_{i,i}(X)	&	
|(HY)| H^{BM}_{i,i}(Y)	\\
};
\begin{scope}[->,font=\footnotesize]
\draw (ZX) --node{$f_*$} (ZY);
\draw (HX) --node{$f_*$} (HY);
\draw (ZX) -- node{$\cyc_X$} (HX);
\draw (ZY) -- node{$\cyc_Y$} (HY);
\end{scope}
\end{tikzpicture}
\end{center}
is commutative for all $i\in \N$.
\end{cor}

\begin{proof}
By definition of the tropical cycle class map, the assertion holds if $X$ is a rational polyhedral subspace of $Y$ and $f$ is the inclusion. Now assume $f$ is general and $A\in Z_i(X)$. Using the result for inclusions and the fact that push-forwards are functorial for both tropical homology classes and tropical cycles, we reduce to the case where $X=|A|$, $Y=f(|A|)$, and $\dim(X)=i$.
If $\dim(Y)<\dim(X)$, then $H^{BM}_{i,i}(Y)=0$ and the statement is trivial, so we may assume $i=\dim(Y)=\dim(X)$. In this case, the result follow from Proposition \ref{prop:canonical morphism for between n-cycles and hom} and Corollary \ref{cor: cyc and tcyc coincide}.
\end{proof}

\subsection{Compatibility with cross products}
Given tropical cycles $A$ and $B$ on locally polyhedral spaces $X$ and $Y$, we have defined their cross-product (\S\ref{subsec:cross product of cycles}) and the cross-product of their tropical cycle classes (\S\ref{subsec:cross products in homology}). We now show that the tropical cycle class of the former equals the latter.

\begin{prop}
\label{prop:cycle map commutes with cross products}
Let $X$, $Y$ be rational polyhedral spaces, and let $i,j\in \N$. Then the tropical cycle class map takes cross products to cross products. In other words, the diagram 
\begin{center}
\begin{tikzpicture}[auto]
\matrix[matrix of math nodes, row sep= 5ex, column sep= 4em, text height=1.5ex, text depth= .25ex]{
|(ZXZY)| Z_i(X)\otimes_\Z Z_j(Y) 	&
|(ZXY)|	Z_{i+j}(X\times Y)	\\
|(HXHY)| H^{BM}_{i,i}(X)\otimes_\Z H^{BM}_{j,j}(Y)	&	
|(HXY)| H^{BM}_{i+j,i+j}(X\times Y)	\\
};
\begin{scope}[->,font=\footnotesize]
\draw (ZXZY) --node{$\times$} (ZXY);
\draw (HXHY) --node{$\times$} (HXY);
\draw (ZXZY) -- node{$\cyc_X\otimes \cyc_Y$} (HXHY);
\draw (ZXY) -- node{$\cyc_{X\times Y}$} (HXY);
\end{scope}
\end{tikzpicture}
\end{center}
is commutative.
\end{prop}

\begin{proof}
Let $A\in Z_i(X)$ and $B\in Z_j(Y)$. By the functoriality of the cross-product we may assume that $X=|A|$ is purely $i$-dimensional, and $Y=|B|$ is purely $j$-dimensional. In this case, both morphisms
\begin{align*}
\cyc(A)\times\cyc(B)&\colon \Omega^{i+j}_{X\times Y}[i+j]\to \dual_{X\times Y} \ , \text{ and} \\
\cyc(A\times B)&\colon \Omega^{i+j}_{X\times Y}[i+j]\to\dual_{X\times Y}
\end{align*}
are completely determined by the morphisms $\Omega^{i+j}_{X\times Y}\to\scrH^{i+j}_X$ they induce by taking cohomology in degree $-(i+j)$. By Proposition \ref{prop:canonical morphism for between n-cycles and hom}, it suffices to compare these morphisms locally at a point  $z=(x,y)\in(X\times Y)^{\max}$.  Then $x\in X^{\max}$ and $y\in Y^{\max}$. Let $\eta_x$ be a generator of $\bigwedge\nolimits^i T^\Z_x X$ and $\eta_y$  a generator of $\bigwedge\nolimits^j T^\Z_y Y$, and let $[\sigma_x]$ and $[\sigma_y]$ be the corresponding generator of $\scrH^i_{X,x}$ and $\scrH^j_{Y,y}$. Then as explained in Remark \ref{rem:explicit description of natural morphism}, the morphism $\Omega^i_{X,x}\to \scrH^i_{X,x}$ defined by $X$ is given by $\omega\mapsto \langle\omega,\eta_x\rangle A(x)[\sigma_x]$, and the morphism $\Omega^j_{Y,y}\to \scrH^i_{Y,y}$ defined by $B$ is given by $\omega\mapsto \langle \omega,\eta_y\rangle B(y)[\sigma_y]$. By definition of the cross product and Lemma \ref{lem:Eilenberg-Zilber}, the morphism $\Omega^{i+j}_{X\times Y, z} \to \scrH^{i+j}_{X\times Y, z}$ induced by $\cyc(A)\times\cyc(B)$ takes $\omega\otimes \omega'\in \Omega^i_{X,x}\otimes \Omega^j_{Y,y}\cong \Omega^{i+j}_{X\times Y, z}$ to 
\begin{equation*}
\langle \omega, \eta_x\rangle \langle \omega',\eta_y\rangle A(x) B(y) [\sigma_x]\times [\sigma_y].
\end{equation*}
Here $[\sigma_x]\times[\sigma_y]$ denotes the classical cross-product, which equals the generator $[\sigma_z]$ of $\scrH^{i+j}_{X\times Y}$ corresponding to the generator $\eta_z=\eta_x\otimes \eta_y$ of $\bigwedge^{i+j}T^\Z_z(X\times Y)$. Therefore, the expression above for the image of $\omega\otimes \omega'$ equals
\[
\langle \omega\otimes\omega', \eta_x\otimes \eta_y\rangle (A\times B)(z) [\sigma_z]
\]
which is precisely the image of $\omega\otimes \omega'$ under the morphism $\Omega^{i+j}_{X\times Y, z}\to \scrH^{i+j}_{X\times Y,z}$ induced by $\cyc_{X\times Y}(A\times B)$.
\end{proof}

\subsection{Chern classes}

Let $X$ be a rational polyhedral space. As explained in \S\ref{subsec:line bundles}, the set of isomorphism classes of tropical line bundles on $X$ is an abelian group, naturally isomorphic to $H^1(X,\Aff_X)$. 

\begin{defn}
Let $X$ be a rational polyhedral space, and let $d\colon \Aff_X\to \Omega^1_X$ be the quotient map. Then the \emph{first Chern class} is defined as the morphism
\[
c_1\coloneqq H^1(d)\colon H^1(X,\Aff_X)\to H^1(X,\Omega^1_X)=H^{1,1}(X) \ .
\]
If $\mL$ is a tropical line bundle on $X$, corresponding to $\alpha\in H^1(X,\Aff_X)$, then the first Chern class of $\mL$ is $c_1(\mL)\coloneqq c_1(\alpha)$.
\end{defn}

\begin{prop}
\label{prop:chern class pulls back to chern class}
Let $f\colon X \to Y$ be a morphism of rational polyhedral spaces, and let $\mL$ be a tropical line bundle on $Y$. Then 
\[
c_1(f^*\mL)=f^*(c_1(\mL)) \ .
\]
\end{prop}
\begin{proof}
This follows immediately from the commutativity of the diagram
\begin{center}
\begin{tikzpicture}[auto]
\matrix[matrix of math nodes, row sep= 5ex, column sep= 4em, text height=1.5ex, text depth= .25ex]{
|(AffY)| H^1(Y, \Aff_Y) &
|(OY)|  H^1(Y, \Omega^1_Y) \\
|(XAffY)| H^1(X,f^{-1}\Aff_Y)	&
|(XOY)| H^1(X,f^{-1}\Omega^1_Y) \\
|(AffX)| H^1(X,\Aff_X)	&
|(OX)| H^1(X,\Omega^1_X ) \ ,\\
};
\begin{scope}[->,font=\footnotesize]
\draw (AffY) -- (OY);
\draw (XAffY) -- (XOY);
\draw (AffX) -- (OX);
\draw (AffY)--node{$f^{-1}$}(XAffY);
\draw (OY)--node{$f^{-1}$}(XOY);
\draw (XAffY) --node{$H^1(f^\sharp)$} (AffX);
\draw (XOY)--node{$H^1(f^\sharp)$} (OX);
\end{scope}
\end{tikzpicture} 
\end{center}
where the horizontal morphisms are induced by the quotient morphisms $\Aff\to \Omega^1$.
\end{proof}

Also recall from \S\ref{subsec:Cartier divisors and line bundles} that there is a sheaf of tropical Cartier divisors $\Divs_X$ which fits into a short exact sequence
\begin{equation}
\label{equ:sequence for divisors}
0\to \Aff_X\to \Rat_X\to \Divs_X \to 0 \ ,
\end{equation}
the first connecting homomorphism of whose associated long exact cohomology sequence is the map
\[
\Div(X)=H^0(X,\Divs_X)\to H^1(X,\Aff_X),\;\; D\mapsto \mL(D) \ .
\]
Composing this with the first Chern class defines a map
\[
\Div(X)\to H^{1,1}(X),\;\; D\mapsto c_1(\mL(D)) 
\]
that assigns a $(1,1)$-cohomology class to every Cartier divisor. We will  need to work with a factorization of this map. A Cartier divisor $D\in \Div(X)$ is, by definition, a global section of $\Divs_X$ whose support is, again by definition, equal to $|D|$. Therefore, it defines a morphism $\Z_{|D|}\to \Divs_X$. Composing this with the morphism $\Divs_X\to \Aff_X[1]$ defined by (\ref{equ:sequence for divisors}) and the projection morphism $\Aff_X[1]\to \Omega^1_X[1]$ defines a morphism $\Z_{|D|}\to \Omega^1_X[1]$, and hence an element in $H^{1,1}_{|D|}(X)$ (see Remark \ref{rem:cohomology classes as morphisms}), the image of which in $H^{1,1}(X)$ is $c_1(\mL(D))$. We can thus view $c_1(\mL(D))$ as an element in $H^{1,1}_{|D|}(X)$.

We will now show that taking the cap product with the first Chern class $c_1(\mL(D))$ corresponds to intersecting with $D$.

\begin{prop}[cf.\ {\cite[Theorem 4.15]{Lefschetz}}]
\label{prop:cycle map respects intersections with chern classes}
Let $X$ be a rational polyhedral space. Then for every tropical Cartier divisor $D\in \Div(X)$ and for every tropical cycle $A\in Z_i(X)$ we have
\begin{equation*}
\cyc_X (D\cdot A ) = c_1(\mL(D))\frown\cyc_X(A)
\end{equation*}
in $H^{BM}_{i-1,i-1}(X)$.
\end{prop}

\begin{proof}
First we reduce to the case $X=|A|$. Let $\iota\colon |A| \hookrightarrow X$ be the inclusion map. Then by definition of the intersection pairing, and by Proposition \ref{prop:cycle map commutes with push-forwards}, we have
\begin{equation*}
\cyc_X (D\cdot A)= \cyc_X(\iota_*(\iota^*D\cdot A))=\iota_*\cyc_{|A|}(\iota^*D\cdot A) \ .
\end{equation*}
On the other hand, by the projection formula in homology (Proposition  \ref{prop:projection formula})  we have
\begin{multline*}
c_1(\mL(D))\frown\cyc_X(A) =\\
 = c_1(\mL(D))\frown \iota_*\cyc_{|A|}(A)=\iota_*(\iota^*(c_1(\mL(D))\frown \cyc_{|A|}(A)) \ ,
\end{multline*}
which equals $\iota_*(c_1(\mL(\iota^*D))\frown \cyc_{|A|}(A))$ by Propositions \ref{prop:pullback and taking line bundles commutes} and \ref{prop:chern class pulls back to chern class}. It thus suffices to show that
\[
\cyc_{|A|} (\iota^*D\cdot A ) =   c_1(\mL(\iota^*D))\frown \cyc_{|A|}(A)
\]
in $H^{BM}_{i-1,i-1}(|A|)$, which allows us to assume $X=|A|$. 
In this case, we will interpret $c_1(\mL(D))$ as an element of $H^{1,1}_{|D|}(X)$ and show that the equality holds in $H^{BM}_{i-1,i-1}(|D|,X)$. Since $|D|$ is at most $(i-1)$-dimensional, the presheaf
\[
U\mapsto H^{BM}_{i-1,i-1}(U\cap |D|,U)
\]
on $X$ is a sheaf by Lemma \ref{lem:basicBM} (b). We can thus work locally around a point $x\in X$, where we can use a local face structure $\Sigma$. After potentially shrinking and refining $\Sigma$, we can assume that the divisor $D$ is principal, say $D=\divv(\phi)$, and that $\phi|_\sigma\in \Gamma(\sigma,\Aff_\sigma)$ for all $\sigma\in\Sigma$. As usual, we choose an orientation on each $\sigma\in \Sigma$ in form of a generator 
\[
\eta_\sigma\in \bigwedge\nolimits^{\dim \sigma}T^\Z (\sigma) ,
\]
and these orientations define classes $[\sigma]\in \Gamma(X,\Delta_X^{-\dim\sigma})$ supported on $\sigma$ for all $\sigma\in\Sigma$.

The rational function $\phi$ defines a morphism $\Z_{|D|}\to \Divs_X$, and, by definition, the first Chern class $c_1(\mathcal L(D))\in H^{1,1}_{|D|}(X)$ is the composite of this morphism with the composite $\Divs_X\to \Aff_X[1]\to\Omega^1_X[1]$, where the first morphism comes from the exact sequence (\ref{equ:sequence for divisors}). These are morphisms in the derived category $\der{X}$ that are not represented by morphisms between the complexes involved. To remedy this, we note that the exact sequence (\ref{equ:sequence for divisors}) yields an isomorphism between $\Divs_X$ and the complex \[
\ldots \to 0 \to \Aff_X \to \Rat_X \to 0 \to \ldots \ ,
\]
where $\Rat_X$ sits in degree $0$, and that the short exact sequence 
\[
0\to \Z_U\to \Z_X\to \Z_{|D|}\to 0 \ ,
\]
where $U=X\setminus |D|$, yields an isomorphism between $\Z_{|D|}$ and the complex
\[
\ldots \to 0 \to \Z_U\to \Z_X \to 0\to \ldots \ ,
\]
where $\Z_X$ sits in degree $0$. The first Chern class $c_1(\mathcal L(D))$ is then represented by the diagram

\begin{center}
\begin{tikzpicture}[auto]
\matrix[matrix of math nodes, row sep= 5ex, column sep= 4em, text height=1.5ex, text depth= .25ex]{
|(ZX)| \Z_X &
|(Div)|  \Rat_X &
|(01)| 0 &
|(02)| 0 \\
|(ZU)| \Z_U	&
|(Aff1)| \Aff_X &
|(Aff2)| \Aff_X &
|(OX)|  \Omega^1_X  \ ,\\
};
\begin{scope}[->,font=\footnotesize]
\draw (ZX) --node{$\cdot \phi$} (Div);
\draw (Div) --(01);
\draw (01) -- (02);
\draw (ZU) --node{$\cdot \phi$} (Aff1);
\draw (Aff1)--node{$-\id$}(Aff2);
\draw (Aff2)--(OX);
\draw (ZU) -- (ZX);
\draw (Aff1)--(Div);
\draw (Aff2) -- (01);
\draw (OX)--(02);
\end{scope}
\end{tikzpicture} 
\end{center}
where each column represents an element in $\der X$ and the minus sign in the middle morphism in the lower row is there by convention. Tensoring with $\Omega^{i-1}_X[i-1]$ from the right and composing with the multiplication morphism $\Omega^{i-1}_X\otimes_\Z \Omega^1_X\to \Omega^i_X$, and the morphism $\Omega^i_X[i]\to \Delta_X^\bullet$ representing $\cyc_X(A)$, we see that $c_1(\mathcal L(D))\frown\cyc(A)$ is represented by the morphism $((\Omega^{i-1}_X)_U\to \Omega^{i-1}_X) ) \to \Delta_X^\bullet$, where the former complex sits in degrees $-i$ and $-(i-1)$, that is given by
\[
(\Omega^{i-1}_X)_U\to \Delta^{-i}_X \colon \;\; \omega \mapsto -\sum_{\sigma\in \Sigma(i)} \langle d(\phi|_\sigma)\wedge\omega|_\sigma,\eta_\sigma\rangle A(\sigma)[\sigma] \ ,
\]
in degree $-i$ and $0$ in every other degree. The morphism in degree $-i$ extends to a morphism $\Omega^{i-1}_X\to \Delta^{-i}_X$, defining a chain homotopy between the morphism of cochain complexes just defined, and the morphism that is given by
\[
\Omega^{i-1}_X\to \Delta^{-(i-1)}_X \colon \;\; \omega\mapsto \partial \left(\sum_{\sigma\in \Sigma(i)} \langle d(\phi|_\sigma)\wedge\omega ,\eta_\sigma\rangle A(\sigma)[\sigma]\right)
\]
in degree $-(i-1)$ and is $0$ in all other degrees. 
 To simplify the expression on degree $-(i-1)$  we pick for every finite codimension-$1$ face $\tau$ of a cone $\sigma\in\Sigma(i)$ a lattice normal vector $n_{\sigma/\tau}$, and we set $\epsilon_{\sigma/\tau}$ equal to $1$ if the chosen orientations on $\sigma$ and $\tau$ are compatible, and to $-1$ otherwise. Recall that this means that
\[
\epsilon_{\sigma/\tau}\eta_\sigma= \eta_\tau\wedge n_{\sigma/\tau} \ .
\]
 We now compute that 
\begin{multline*}
\partial \left(\sum_{\sigma\in \Sigma(i)} \langle d(\phi|_\sigma)\wedge\omega ,\eta_\sigma\rangle A(\sigma)[\sigma]\right) = \\ =
\sum_{\tau\in \Sigma(i-1)}
\sum_{\sigma\in\Sigma_\tau(i)}
\langle d(\phi|_\sigma)\wedge\omega ,\epsilon_{\sigma/\tau} \eta_\tau\wedge n_{\sigma/\tau}\rangle
 A(\sigma)\epsilon_{\sigma/\tau}[\tau] = \\
 \sum_{\tau\in \Sigma(i-1)}
 \sum_{\sigma\in\Sigma_\tau(i)}
\langle \omega \wedge d(\phi|_\sigma) ,n_{\sigma/\tau}\wedge \eta_\tau\rangle
 A(\sigma)[\tau] \ ,
\end{multline*}
where $\Sigma_\tau(i)$ denotes the set of cells $\sigma\in\Sigma(i)$ that have $\tau$ as a finite face. Note that we only need to consider finite faces because if $\omega$ is defined at a point of an infinite face of $\sigma$, then $\langle d(\phi|_\sigma)\wedge \omega,\eta_\sigma\rangle=0$.
Let $l_\tau$ be an affine linear function on a neighborhood of $|\Sigma|$  such that $l_\tau|_\tau=\phi|_\tau$. Then we can rewrite the coefficient of $[\tau]$ in the expression above as
\begin{multline*}
 \sum_{\sigma\in\Sigma_\tau(i)}
\Big(\big\langle \omega\wedge(d(\phi|_\sigma)-dl_\tau) ,n_{\sigma/\tau}\wedge \eta_\tau\big\rangle
 A(\sigma)+\big\langle \omega\wedge dl_\tau ,n_{\sigma/\tau}\wedge \eta_\tau\big\rangle
 A(\sigma)\Big) = \\
=\sum_{\sigma\in\Sigma_\tau(i)} \big\langle\omega ,\eta_\tau\big\rangle\cdot\big\langle (d(\phi|_\sigma)-dl_\tau),n_{\sigma/\tau}\big\rangle
 A(\sigma) + \\ 
 + \left\langle \omega\wedge dl_\tau ,\left(\sum_{\sigma\in\Sigma_\tau(i)}A(\sigma)n_{\sigma/\tau}\right)\wedge \eta_\tau\right\rangle \ . 
\end{multline*}
By the balancing condition and the fact that $\bigwedge^iT^\Z(\tau)=0$, the second summand of the last expression vanishes and we see that the coefficient of $[\tau]$ is given by 
\[
\langle\omega ,\eta_\tau\rangle \left(\sum_{\sigma\in\Sigma_\tau(i)} \big\langle (d(\phi|_\sigma)-dl_\tau),n_{\sigma/\tau}\big\rangle A(\sigma) \right)  \ .
\]
But this is precisely the coefficient of $[\tau]$ one gets in $\cyc_X(D\cdot A)$ (compare the coefficient with the weights of $D\cdot A$ described in  \S\ref{subsec:Cartier divisors and line bundles}), finishing the proof.
\end{proof}

\begin{rem}
If one uses a version of the projection formula that respects supports in the proof of Proposition \ref{prop:cycle map respects intersections with chern classes}, one can show that the formula shown in the proposition actually holds in $H^{BM}_{i-1,i-1}(|D|\cap |A|, X)$.
\end{rem}

\subsection{Compatibility with the singular cycle class map}
Let $X$ be a locally polyhedral space, equipped with a face structure $\Sigma$ whose cells are compact, and let $A\in Z_k(X)$ be a tropical $k$-cycle on $X$. There exists a face structure $\Sigma'$ on $|A|$ such that every cell of $\Sigma'$ is contained in a cell of $\Sigma$. For each $\sigma\in\Sigma'(k)$ we choose an orientation $\eta_\sigma\in\bigwedge^k T^\Z(\sigma)$, which give rise to chains $[\sigma]\in\Gamma(X,\Delta^{-k}_X)$ supported on $\sigma$ via cellular homology. The \emph{locally finite cycle class of $A$}, as considered in \cite[Section 4]{Lefschetz} (also see \cite[Section 4]{MZeigenwave}) is defined as
\[
\cyc_X^{lf}(A)= \sum_{\sigma\in \Sigma'(k)} A(\sigma) \cdot [\sigma]\otimes\eta_\sigma  \ ,
\]
where $[\sigma]\otimes\eta_\sigma$ is the locally finite tropical $(k,k)$-chain defined by $[\sigma]$ and $\eta_\sigma$, and  $A(\sigma)$ is the constant value that $A$ has on $\sigma$. Since any two face structures have a common refinement, this is independent of the choice of $\Sigma'$ and defines a morphism
\[
\cyc_X^{lf}\colon Z_k(X)\to H^{lf}_{k,k,}(X;\Sigma) \ .
\]

\begin{thm}
Let $X$ be a rational polyhedral space, equipped with a face structure $\Sigma$ with compact cells. Then the natural isomorphism (Theorem \ref{thm:compatibility of homology theories})
\begin{equation*}
H^{lf}_{k,k}(X;\Sigma) \xrightarrow{\cong} H^{BM}_{k,k}(X) \ .
\end{equation*}
is compatible with the two tropical cycle class maps. In other words, 
the diagram
\begin{center}
\begin{tikzpicture}[auto]
\node (Z) at (60:2.5) {$Z_k(X)$};
\node (HS) at (0,0) {$H^{lf}_{k,k}(X;\Sigma)$};
\node (H) at (2.5,0) {$H^{BM}_{k,k}(X)$};
\begin{scope}[font=\footnotesize,->]
\draw (Z) -- node[swap]{$\cyc_X^{lf}$} (HS);
\draw (Z) -- node{$\cyc_X$} (H);
\draw (HS) --node{$\cong$} (H);
\end{scope}
\end{tikzpicture}
\end{center}
is commutative.
\end{thm}

\begin{proof}
Let $A\in Z_k(X)$, and let  $\Sigma'$ be a face structure on $|A|$ whose cells are contained in cells of $\Sigma$. Choose orientations $\eta_\sigma\in \bigwedge^k T^\Z(\sigma)$ for all $\sigma\in\Sigma'$, and let $[\sigma]\in\Gamma(X,\Delta^{-k}_X)$ be the chain supported on $\sigma$ obtained from $\eta_\sigma$ via cellular homology. Using the face structure $\Sigma'$, the morphism $\cyc_X(A)\colon \Omega^k_X[k]\to \dual_X$ in $\der X$ can actually be represented as a morphism of complexes $\Omega^k_X[k]\to \Delta^{\Sigma,\bullet}_X$, namely as the morphism whose component in degree $-k$ is given by
\[
\Omega^n_X[k] \ni \omega \mapsto \sum_{\sigma\in \Sigma'(k)} A(\sigma) \cdot \langle \omega , \eta_\sigma\rangle  [\sigma] \in \Delta^{\Sigma,-k}_X \ .
\]
This corresponds to the locally finite tropical $(k,k)$-chain
\[
\sum_{\sigma\in\Sigma'(k)} A(\sigma)\cdot [\sigma]\otimes\eta_\sigma \in C^{lf}_{k,k}(X;\Sigma)
\]
under the isomorphism $\Hom(\Omega^k_X,\Delta_X^{\Sigma,-k})\cong C^{lf}_{k,k}(X;\Sigma)$ used in the proof of Theorem \ref{thm:compatibility of homology theories}. Noting that this tropical chain represents $\cyc_X(A)$ by definition finishes the proof.
\end{proof}

\section{Integral Poincar\'e-Verdier duality}
\label{sec:duality}

\subsection{Tropical manifolds}
Tropical manifolds are the smooth spaces in the world of rational polyhedral spaces. Let us briefly recall their definition. To every loopless matroid $M$ on a finite base set $E(M)$ we can associate a \emph{tropical linear space} $\berg{M}\subseteq \R^{E(I)}/\R\mathbf{1}$, which is also called the \emph{Bergman fan} of $M$. Here, $\mathbf{1}$ denotes the vector with all coordinates equal to $1$. We refer to \cite{ArdilaKlivans} and \cite[Chapter 4]{TropBook} for precise definitions and further details.

\begin{defn}
A rational polyhedral space $X$ is called \emph{smooth}, or a \emph{tropical manifold}, if every point $x\in X$ has an open neighborhood that is isomorphic to an open subset of $\berg M\times \Rbar^k$ for some loopless matroid $M$ and some $k\in \N$.
\end{defn}

\subsection{Poincar\'e-Verdier duality}

The goal of this section is to prove a formula for the Verdier duals of the sheaf of tropical $p$-forms on a tropical manifold $X$. Recall that for any complex $\mathcal C^\bullet\in\der X$, its \emph{Verdier dual} is defined as
\[
\vdual(\mathcal C^\bullet)\coloneqq R\Homs^\bullet(\mathcal C^\bullet, \dual_X) \ .
\]
It is immediate from this definition that $\vdual(\Z_X)=\dual_X$. What is much less obvious is that for a constructible complex $\mathcal C^\bullet \in \der X$ there is a natural isomorphism $\vdual(\vdual(\mathcal C^\bullet)\cong \mathcal C^\bullet$, justifying the terminology ``dual''. We can now state the main theorem of this section:

\begin{thm}
\label{thm:computing the Verdier dual}
Let $X$ be a purely $n$-dimensional tropical manifold. Then there is a natural isomorphism 
\begin{equation*}
\Omega^{n-p}_X[n]\to \vdual(\Omega^p_X) \ .
\end{equation*} 
\end{thm}

Before starting to prove this theorem, let us show how Poincar\'e duality on a tropical manifold follows directly from it. 

\begin{cor}[Poincar\'e duality with integer coefficients]
\label{cor:Poincar\'e duality}
Let $X$ be a tropical manifold of pure dimension $n$, and let $p$ and $q$ be integers. Then there are natural isomorphisms 
\begin{align*}
H^{BM}_{p,q}(X) &\cong H^{n-p,n-q}(X) \ , \text{ and} \\
H_{p,q}(X) &\cong H_c^{n-p,n-q}(X) \ .
\end{align*}
\end{cor}
\begin{rem}
The first isomorphism in Corollary \ref{cor:Poincar\'e duality} has also been proven in \cite{Lefschetz} using a Mayer-Vietoris argument, and with the extra assumption of the existence of a global face structure.
\end{rem}
\begin{proof}
By definition, we have
\begin{align*}
H^{BM}_{p,q}(X) &= \Hyper^{-q}R\Homs^\bullet(\Omega_X^p,\dual_X)= \Hyper^{-q}\vdual(\Omega_X^p) \ , \text{ and} \\
H_{p,q}(X) &= \Hyper_c^{-q}R\Homs^\bullet(\Omega_X^p,\dual_X)=\Hyper_c^{-q}\vdual(\Omega_X^p) \ .
\end{align*}
Using Theorem \ref{thm:computing the Verdier dual}, this leads to isomorphisms
\begin{align*}
H^{BM}_{p,q}(X) &\cong \Hyper^{-q}(\Omega_X^{n-p}[n]) \cong \Hyper^{n-q}(\Omega_X^{n-p}) =H^{n-p,n-q}(X) \ ,  \text{ and} \\
H_{p,q}(X) &\cong \Hyper_c^{-q}(\Omega_X^{n-p}[n]) \cong \Hyper_c^{n-q}(\Omega_X^{n-p}) =H_c^{n-p,n-q}(X) \ ,
\end{align*}
which is what we wanted to show.
\end{proof}

Let us now turn to the proof of Theorem \ref{thm:computing the Verdier dual}. We begin by explaining how to obtain the natural morphism $\Omega^{n-p}_X[n]\to \vdual(\Omega^p_X)$. We recall that there is a natural morphism
\begin{equation}
\label{equ:natural morphism in deg -n}
\Homs(\Omega^p_X,\scrH^n_X)[n]\to R\Homs(\Omega^p_X,\dual_X)= \vdual(\Omega^p_X)
\end{equation}
that we relied on heavily when we defined the tropical cycle class map. It exists because $H^j(\dual_X)=0$ for $j<-n$ and is an isomorphism on cohomology in degree $-n$. 
By Proposition \ref{prop:canonical morphism for between n-cycles and hom} there also a natural morphism
\begin{equation*}
\Omega^n_X\otimes_{\Z_X} \mathcal Z_n^X \to \scrH^n_X \ .
\end{equation*}
By \cite[Lemma 2.4]{FRIntersection}, the only tropical $n$-cycles on an $n$-dimensional tropical linear space are the ones having the same weight everywhere. Therefore, $\mathcal Z_n^X\cong \Z_X$ is the constant sheaf and we obtain a natural morphism $\Omega^n_X\to \scrH^n_X$. Together with the multiplication on the sheaf of graded rings $\Omega^*_X$, this induces a morphism
\begin{equation*}
\Omega^{n-p}_X \to \Homs(\Omega^p_X,\Omega^n_X) \to \Homs(\Omega^p_X,\scrH^n_X) \ .
\end{equation*}
Shifting by $n$ an composing with the morphism displayed (\ref{equ:natural morphism in deg -n}) yields a morphism $\Omega^{n-p}_X[n]\to \vdual(\Omega^p_X)$. 

\begin{defn}
We say that a purely $n$-dimensional tropical manifold satisfies \emph{Poincar\'e-Verdier duality} if for all $p\in\N$ the natural morphism
\[
\Omega^{n-p}_X[n]\to \vdual(\Omega^p_X)
\]
defined above is an isomorphism.
\end{defn}

With this definition, Theorem \ref{thm:computing the Verdier dual} says that every tropical manifold satisfies Poincar\'e-Verdier duality. 

\begin{lemma}
\label{lem:duality for products}
Let $X$ and $Y$ be tropical manifolds of pure dimension that satisfy Poincar\'e-Verdier duality. Then their product $X\times Y$ satisfies Poincar\'e-Verdier duality as well.
\end{lemma}

\begin{proof}
Let $p\in\Z$. By Lemma \ref{lem:differentials on product}, we have
\[
\Omega^p_{X\times Y} \cong \bigoplus_{i+j=p} \Omega^i_X\boxtimes \Omega^j_Y \ .
\]
For each of the summands there is a natural morphism
\[
\vdual (\Omega^i_X)\boxtimes ^L \vdual(\Omega^j_Y) \to \vdual(\Omega^i_X\boxtimes \Omega^j_Y) \ ,
\]
which is an isomorphism by \cite[Corollary 2.8]{ExternalProduct}.
Let $m=\dim(X)$ and $n=\dim(Y)$. Since $X$ and $Y$ both satisfy Poincar\'e-Verdier duality, we obtain a chain of isomorphisms
\begin{multline*}
\Omega^{n+m-p}_{X\times Y} [m+n]
\cong \bigoplus_{i+j=p} \Omega^{m-i}_X[m] \boxtimes \Omega^{n-j}_Y[n] \xrightarrow{\cong} \\
\xrightarrow{\cong} \bigoplus_{i+j=p}\vdual(\Omega^i_X)\boxtimes ^L \vdual (\Omega^j_Y)  \xrightarrow{\cong}
\bigoplus_{i+j=p} \vdual(\Omega^i_X\boxtimes \Omega^j_Y) \cong \vdual (\Omega_{X\times Y}^p) \ .
\end{multline*}
It essentially follows from the fact that $\Omega^*_{X\times Y} \cong \Omega^*_X\boxtimes \Omega^*_Y$ is an isomorphism of sheaves of rings that the composite isomorphism 
\[
\Omega^{n+m-p}_{X\times Y}[m+n]\to \vdual (\Omega_{X\times Y}^p)
\] 
coincides with the natural morphism that we consider for Poincar\'e-Verdier duality. This finishes the proof.
\end{proof}

\begin{lemma}
\label{lem: Verdier dual for Rn}
For any $n\in\N$ the tropical manifold $\Rbar^n$ satisfies Poincar\'e-Verdier duality.
\end{lemma}

\begin{proof}
By Lemma \ref{lem:duality for products}, it suffices to show the statement for $\Rbar$. For every $x\in \Rbar$ we have 
\[
\scrH^0_{\Rbar,x} \cong H_0(\Rbar ,\Rbar\setminus \{x\}) =0 \ ,
\]
so $\dual_{\Rbar}$ is concentrated in degree $-1$ and isomorphic to $\scrH^1_\Rbar[1]$. Furthermore, the natural morphism $\Omega^1_{\Rbar}\to \scrH^1_{\Rbar}$ is an isomorphism: at $\infty$ the stalks of both sheaves are $0$, and at a point in $\R$ this follows from the identification of the two different notions of orientations discussed at the beginning of \S\ref{sec:cycle class map}. This shows that the natural morphism 
\[
\Omega_\Rbar^1[1]\to \dual_\Rbar =\vdual(\Z_\Rbar)=\vdual(\Omega^0_\Rbar)
\]
is an isomorphism. 

It is left to show that the natural morphism 
\[
\Omega^0_\Rbar[1]\to \vdual(\Omega^1_\Rbar)
\]
is an isomorphism as well. For this we need to show that 
\begin{enumerate}
\item the morphism $\Omega^0_\Rbar\to \Homs(\Omega^1_\Rbar,\Omega^1_\Rbar)$ induced by the multiplication maps is an isomorphism, and
\item $\vdual(\Omega^1_\Rbar)$ is concentrated in degree $-1$. 
\end{enumerate}
For both items we use that the identity function $\id_\R \in \Gamma(\R,\Aff_\Rbar)$ induces an isomorphism $\Z_\R\xrightarrow{\cong} \Omega^1_\Rbar$. So if $\iota\colon \R\to \Rbar$ denotes the inclusion, we have
\begin{multline*}
\Homs(\Omega^1_\Rbar,\Omega^1_\Rbar) \cong 
\Homs(\Z_\R,\Z_\R) \cong 
\iota_*\Homs(\Z_\R,\Z_\R) \cong \\
\cong \iota_*\Z_\R \cong 
\Z_\Rbar=
\Omega^0_\Rbar \ ,
\end{multline*}
showing item (1). To prove item (2), we note that
\[
\vdual(\Z_\R) \cong R\iota_*\vdual(\Z_\R)
\]
by Verdier duality. The Verdier dual of $\Z_\R$ on $\R$ is the dualizing complex of $\R$, which is isomorphic to $\Z_\R[1]$. Therefore,
\[
\vdual(\Omega^1_\Rbar)\cong \vdual(\Z_\R) \cong R\iota_*\Z[1] \cong \Z_\Rbar[1] \ ,
\]
which is concentrated in degree $-1$.
\end{proof}

\begin{thm}
\label{thm:duality for linear spaces}
Let $M$ be loopless matroid of rank $n+1$ with ground set $E(M)$, and let $i\in E(M)$ such that $i$ is not a coloop of $M$ and  $M/i$ is loopless. Assume that both $\berg {M\setminus i}$ and $\berg {M/i}$ satisfy Poincar\'e-Verdier duality. Then $\berg M$ also satisfies Poincar\'e-Verdier duality.
\end{thm}

\begin{proof}
We denote $X'=\berg M$, $X=\berg {M\setminus i}$, and $D=\berg{M/i}$. There is a natural morphism $\delta \colon X'\to X$ that exhibits $X'$ as the so-called tropical modification of $X$ at the divisor $D$ \cite[Proposition 2.25]{ShawIntersection}. By definition of tropical modifications, there exists a closed subset $E\subseteq X'$, such that the interior $\mathring E$ is isomorphic to an open subset of $D\times \R$, such that the underlying topological space of $E$ is homeomorphic to $D\times \R_{\leq 0}$, and such that the underlying space of $X'$ is homeomorphic to the quotient
\begin{equation*}
\big(X \sqcup D\times \R_{\leq 0}\big)\big/ \sim \ , 
\end{equation*}
where where $\sim$ is the relation identifying $x$ with $(x,0)$ for all  $x\in D$. With this identification the morphism $\delta$ is the identity on $X$ and the projection to the first coordinate on $D\times \R_{\leq 0 }$. In particular, the identity on $X$ defines a section $s\colon X\to X'$ of $\delta$. Note that unlike $\delta$, the closed immersion $s$ is not a morphism of rational polyhedral spaces. Let $j\colon \mathring E\to X'$ be the open immersion onto the complement of $s(X)$. Then we obtain a distinguished triangle
\begin{equation}
\label{equ:distinguished triangle}
\vdual(s_*s^{-1}\Omega^p_{X'}) \to \vdual(\Omega^p_{X'}) \to \vdual(j_!\Omega^p_{\mathring E})\to \vdual(s_*s^{-1}\Omega^p_{X'})[1]
\end{equation}
We can explicitly compute both $\vdual(s_*s^{-1}\Omega^p_{X'})$ and $\vdual(j_!\Omega^p_{\mathring E})$. 

Let us first consider the easier case of $\vdual(j_!\Omega^p_{\mathring E})$. By Verdier duality, we have 
\[
\vdual(j_!\Omega^p_{\mathring E})=Rj_*\vdual (\Omega^p_{\mathring E}) \ .
\]
Because $\mathring E$ is an open subset of $\berg{M/i}\times \R^k$ and $\berg{M/i}$ satisfies Poincar\'e-Verdier duality by assumption, $\mathring E$ satisfies Poincar\'e-Verdier duality by Lemma \ref{lem:duality for products}. It follows that there is a natural isomorphism
\begin{equation*}
Rj_*\Omega^{n-p}_{\mathring E}[n] \xrightarrow{\cong} \vdual(j_!\Omega^p_{\mathring E}) \ .
\end{equation*}
Since every $x\in s(X)$ has a neighborhood basis consisting of open subsets $U$ such that $U\cap \mathring E$ equals the interior of the support of a local face structure at some $y\in \mathring E$ close to $x$, it follow from \cite[Proposition 8.1.4 (ii)]{KashiwaraSchapira} that $R^s_*\Omega^{n-p}_{\mathring E}=0$ for all $s>0$. Therefore, $Rj_*\Omega^{n-p}_{\mathring E}\cong j_*\Omega^{n-p}_{\mathring E}$ and we obtain an isomorphism 
\[
j_*\Omega^{n-p}_{\mathring E}[n] \xrightarrow{\cong} \vdual(j_!\Omega^p_{\mathring E}) \ .
\]

Let us now compute $\vdual(s_*s^{-1}\Omega^p_{X'})$. To do so, we need to compare the integral structures of $X$ and $X'$. The morphism $\delta$ induces a pullback morphism $\Omega^p_{X}\xrightarrow{\delta^\sharp}s^{-1}\Omega^p_{X'}$, which is injective and an isomorphism away from the divisor $D$. To determine its cokernel, we may restrict our attention to $D$. The restriction to $p$-forms to $\mathring E$ defines a morphism
\[
s^{-1}\Omega^p_{X'}\to s^{-1}j_* \Omega^p_{\mathring E} \ .
\]
It follows from Lemma \ref{lem:differentials on product},from  \cite[Proposition 8.1.4(i)]{KashiwaraSchapira}, and from the fact that $\mathring E$ is isomorphic to an open subset of $D\times \R$ that  there is an isomorphism
\[
s^{-1}j_*\Omega^p_{\mathring E} \cong \iota_*\left(\Omega^p_D\oplus \Omega^{p-1}_D\right) \ ,
\]
where $\iota\colon D\to X$ denotes the inclusion. We thus obtain a composite morphism
\begin{equation}
\label{equn:definition of Omega-p on modification to Omega-p-1 on divisor}
s^{-1}\Omega^p_{X'}\to \iota_*\left(\Omega^p_D\oplus \Omega^{p-1}_D\right) \to \iota_*\Omega^{p-1}_D \ ,
\end{equation}
where the second morphism is the projection and $\Omega^{p-1}_D=0$ if $p=0$ by convention. We denote this composite morphism by
\[
\rho\colon s^{-1}\Omega^p_{X'}\to \iota_*\Omega^{p-1}_D \ .
\]
If $a$ is a $p$-form on an open subset of $X$, the image of $\delta^\sharp a$ in $\iota_* \left(\Omega^p_D\oplus \Omega^{p-1}_D\right)$ is given by $(\iota^\sharp (a),0)$ and thus $\rho\circ \delta^\sharp(a)=0$. In fact, the sequence
\begin{equation}
\label{equn:exact sequence for Omegas in modification}
0\to \Omega^p_{X}\xrightarrow{\delta^\sharp} s^{-1}\Omega^p_{X'}\xrightarrow{\rho} \iota_*\Omega^{p-1}_D\to 0
\end{equation}
is exact by \cite[Lemma 2.2.7]{ShawThesis}. 
This gives rise to the distinguished triangle
\begin{center}
\begin{tikzpicture}[auto]
\matrix[matrix of math nodes, row sep= 5ex, column sep= 2.7em, text height=1.5ex, text depth= .25ex]{
|(dOD)| \vdual(\iota_*\Omega^{p-1}_D)	&
|(dOX)| \vdual (s^{-1}\Omega^p_{X'}) &
|(dOX')| \vdual (\Omega^p_{X}) &
|(dOD1)| \vdual(\iota_*\Omega^{p-1}_D)[1] \ . \\
};
\begin{scope}[->,font=\footnotesize]
\draw (dOD) -- (dOX);
\draw (dOX) -- (dOX');
\draw (dOX')--(dOD1);
\end{scope}
\end{tikzpicture}
\end{center}

By assumption, there is a natural isomorphism $\Omega^{n-p}_{X}[n]\to \vdual(\Omega^p_{X})$. Furthermore, by Verdier duality we have $\vdual(\iota_*\Omega^{p-1}_D)=\iota_*\vdual(\Omega^{p-1}_D)$, so again by assumption there is a natural isomorphism $\iota_*\Omega^{n-p}_D[n-1]\to\vdual(\iota_*\Omega^{p-1}_D)$. If we define $\mathcal K$ as the kernel of the restriction $\iota^\sharp\colon \Omega^{n-p}_{X}\to \iota_*\Omega^{n-p}_D$, then there is an exact sequence
\begin{equation*}
\label{equn: defn of K}
0\to \mathcal K\to \Omega^{n-p}_{X}\xrightarrow{\iota^\sharp} \iota_*\Omega^{n-p}_D\to 0 \ .
\end{equation*}
We thus obtain a diagram of solid arrow 
\begin{center}
\begin{tikzpicture}[auto]
\matrix[matrix of math nodes, row sep= 5ex, column sep= 2.7em, text height=1.5ex, text depth= .25ex]{
|(OD)| \iota_*\Omega^{n-p}_D[n-1]	&
|(K)| \mathcal K[n]	&
|(OX')| \Omega^{n-p}_{X}[n] &
|(OD1)| \iota_*\Omega^{n-p}_D[n]  \\
|(dOD)| \vdual(\iota_*\Omega^{p-1}_D)	&
|(dOX)| \vdual (s^{-1}\Omega^p_{X'}) &
|(dOX')| \vdual (\Omega^p_{X}) &
|(dOD1)| \vdual(\iota_*\Omega^{p-1}_D)[1] \ ,\\
};
\begin{scope}[->,font=\footnotesize]
\draw (OD) -- (K);
\draw (K) -- (OX');
\draw (OX')--node{$\iota^\sharp$}(OD1);
\draw (dOD) -- (dOX);
\draw (dOX) -- (dOX');
\draw (dOX')--(dOD1);
\draw (OD)--node{$\cong$} (dOD);
\draw [dashed](K) -- (dOX);
\draw (OX') --node{$\cong$}(dOX');
\draw (OD1)--node{$\cong$}(dOD1);
\end{scope}
\end{tikzpicture}
\end{center}
in which both rows are distinguished triangles. If we show that the right square in this diagram commutes, then it follows formally from the axioms of triangulated categories that the dashed arrow exists, and that it is an isomorphism. Assume that this is true for the moment. Then by Verdier duality, we obtain an isomorphism
\[
s_*\mathcal K[n]\xrightarrow{\cong} s_*\vdual(s^{-1}\Omega^p_{X'}) \cong \vdual(s_*s^{-1}\Omega^p_{X'}) \ .
\]
We have then computed two of the three vertices of the distinguished triangle displayed in (\ref{equ:distinguished triangle}), giving rise to the commutative diagram of solid arrows
\begin{center}
\begin{tikzpicture}[auto]
\matrix[matrix of math nodes, row sep= 5ex, column sep= 2.7em, text height=1.5ex, text depth= .25ex]{
|(K)| s_*\mathcal K[n]	&
|(OXn)| \Omega^{n-p}_{X'}[n]	&
|(jOXn)| j_*\Omega^{n-p}_{\mathring E}[n] &
\\
|(sOX)| s_*\vdual(s^{-1}\Omega^p_{X'}) &
|(OX)| \vdual(\Omega^p_{X'}) &
|(jOX)| Rj_*\vdual(\Omega^p_{\mathring E}) &
|(sOX1)| s_*\vdual(s^{-1}\Omega^p_{X'})[1] \ .\\
};
\begin{scope}[->,font=\footnotesize]
\draw [dashed](K) -- (OXn);
\draw (OXn) --node{$j^\sharp$} (jOXn);
\draw (sOX) -- (OX);
\draw (OX) -- (jOX);
\draw (jOX)--(sOX1);
\draw (K)--node{$\cong$} (sOX);
\draw (OXn) -- (OX);
\draw (jOXn) --node{$\cong$}(jOX);
\end{scope}
\end{tikzpicture}
\end{center}
This already shows that the cohomology of $\vdual(\Omega^p_{X'})$ is concentrated in degree $-n$. What remains to finish the proof is to show that $\Omega^{n-p}_{X'}\to H^{-n}\vdual(\Omega^p_{X'})$ is an isomorphism.

The pullback morphism $\delta^\sharp\colon\delta^{-1}\Omega^{n-p}_{X}\to \Omega^{n-p}_{X'}$ maps $(n-p)$-forms on $X$ that vanish on $D$ to $(n-p)$-forms on $X'$ that vanish on $\mathring E$. Therefore, it induces a morphism $s_*\mathcal K\to \Omega^{n-p}_{X'}$ which, after shifting it by $n$, fills in the dashed arrow in the preceding diagram. It follows directly from the description of the isomorphism $\mathcal K\to \vdual(s^{-1}\Omega^p_X)$ that the resulting square in the diagram is commutative. If we combine the definition of $\mathcal K$ with the exact sequence (\ref{equn:exact sequence for Omegas in modification}), with $n-p$ substituted for $p$, and the definition of the map $s^{-1}\Omega^{n-p}_X\to \iota_*\Omega^{n-p-1}_D$ (see (\ref{equn:definition of Omega-p on modification to Omega-p-1 on divisor})), we see that $\mathcal K$ is the kernel of $j^\sharp$. So taking homology of the diagram above in degree $-n$ yields a commutative diagram 
\begin{center}
\begin{tikzpicture}[auto]
\matrix[matrix of math nodes, row sep= 5ex, column sep= 3em, text height=1.5ex, text depth= .25ex]{
|(0ol)| 0 &[-2em]
|(K)| s_*\mathcal K	&
|(OXn)| \Omega^{n-p}_{X'}	&
|(jOXn)| j_*\Omega^{n-p}_{\mathring E} &[-2em]
|(0or)| 0  \\
|(0ul)| 0 &
|(sOX)| H^{-n}s_*\vdual(s^{-1}\Omega^p_{X'}) &
|(OX)| H^{-n}\vdual(\Omega^p_{X'}) &
|(jOX)| R^{-n}j_*\vdual(\Omega^p_{\mathring E}) &
|(0ur)| 0 \\
};
\begin{scope}[->,font=\footnotesize]
\draw (0ol) --(K);
\draw (K) -- (OXn);
\draw (OXn) -- (jOXn);
\draw (jOXn)--(0or);
\draw (0ul) -- (sOX);
\draw (sOX) -- (OX);
\draw (OX) -- (jOX);
\draw (jOX)--(0ur);
\draw (K)--node{$\cong$} (sOX);
\draw (OXn) -- (OX);
\draw (jOXn) --node{$\cong$}(jOX);
\end{scope}
\end{tikzpicture}
\end{center}
with exact rows. Applying the five lemma shows that $X'=\berg M$ satisfies Poincar\'e-Verdier duality.

To finish the proof we need to show that the square
\begin{center}
\begin{tikzpicture}[auto]
\matrix[matrix of math nodes, row sep= 5ex, column sep= 2.7em, text height=1.5ex, text depth= .25ex]{
|(OX')| \Omega^{n-p}_{X}[n] &
|(OD1)| \iota_*\Omega^{n-p}_D[n]  \\
|(dOX')| \vdual (\Omega^p_{X}) &
|(dOD1)| \vdual(\iota_*\Omega^{p-1}_D)[1] \ ,\\
};
\begin{scope}[->,font=\footnotesize]
\draw (OX')--node{$\iota^\sharp$}(OD1);
\draw (dOX')--(dOD1);
\draw (OX') --node{$\cong$}(dOX');
\draw (OD1)--node{$\cong$}(dOD1);
\end{scope}
\end{tikzpicture}
\end{center}
is commutative. Since all four complexes in this square have cohomology concentrated in degree $-n$, we may check this after taking cohomology in degree $-n$. Because this is a local question, we may work locally at a point $x\in D$ and pick a local face structure $\Sigma$ at $x$ such that $D\cap \vert\Sigma\vert$ is a union of cones. If we denote
\[
\Delta= \{\sigma\in\Sigma\mid \sigma\subseteq D\} \ ,
\]
then $\Delta$ is a local face structure at $x$ in $D$. After appropriately choosing a polyhedron $\tau'$ in $E$ for every $\tau\in\Delta$, the set of polyhedra
\[
\Sigma'=\{s(\sigma) \mid \sigma\in\Sigma\} \cup \{\tau'\mid \tau\in\Delta\}
\]
is a local face structure at $s(x)$ in $X'$. We also choose an orientation $\eta_{\sigma}\in \bigwedge^{\dim(\sigma)} T^\Z (\sigma)$ for every $\sigma\in\Sigma'$, which we use to define a chain $[\sigma]$ on $X'$ via cellular homology. For every $\sigma\in\Sigma$, there is an induced orientation $\eta_\sigma=d\delta (\eta_{s(\sigma)})$, we use to define a chain $[\sigma]$ on $X$ via cellular homology.

Now let $\omega$ be a section of $\Omega^{n-p}_X$ defined in a neighborhood of $x$. We first compute the image of $\omega$ when going through the square counterclockwise. The image of $\omega$ under the left vertical map is represented by the section of $\Homs(\Omega^p_X,\Delta_X^{-n})$ given by
\[
\phi\mapsto \sum_{\sigma\in\Sigma(n)} \langle  \phi \wedge \omega, \eta_\sigma\rangle [\sigma] \ .
\]
To compute the image of this under the connecting homomorphism that constitutes the lower horizontal arrow of the square, we lift it to the section of $\Homs(s^{-1}\Omega^p_{X'},\Delta_X^{-n})$ given by the formula
\[
\phi\mapsto \sum_{\sigma\in\Sigma(n)} \langle  \phi\wedge \delta^\sharp \omega , \eta_{s(\sigma)} \rangle [\sigma] \ .
\]
Now we need to apply the differential $\partial$ of $\Delta^\bullet_X$ to this. We observe that because the fundamental cycle of $X'$ satisfies the balancing condition, we have
\[
\partial\left(\sum_{\sigma\in\Sigma'(n)} \langle \phi' \wedge \delta^\sharp \omega, \eta_{\sigma}\rangle [\sigma] \right)=0
\]
for every section $\phi'$ of $\Omega^p_{X'}$. It follows that if $s_*$ denotes the push-forward of chains along $s$, we have
\begin{equation}
\label{equ:s*}
s_*\left(\partial \left(\sum_{\sigma\in\Sigma(n)} \langle \phi\wedge \delta^\sharp \omega , \eta_{s(\sigma)}\rangle [\sigma] \right)\right)
=-\partial\left(\sum_{\tau\in\Delta(n-1)} \langle \phi\wedge \delta^\sharp \omega , \eta_{\tau'}\rangle [\tau'] \right)
\end{equation}
for all sections $\phi$ of $s^{-1}\Omega^p_{X'}$. 

If $\phi=\delta^\sharp \psi$, then this is zero since the pull-back of every $n$-form on $X$  vanishes on $\mathring E$. Because $s_*$ is injective on chains, it follows that the morphism 
\begin{equation}
\label{equ:morph}
\phi\mapsto \partial\left(\sum_{\sigma\in\Sigma(n)} \langle \phi \wedge  \delta^\sharp \omega , \eta_{s(\sigma)} \rangle [\sigma]\right)
\end{equation}
induces a section of $\Homs(\iota_*\Omega^{p-1}_D,\Delta^{-n}_X)$. 

Let $\phi$ be a section of $\iota_*\Omega^{p-1}_D$. To compute the effect of the morphism $\iota_*\Omega^{p-1}\to \Delta^{-n}_X$ defined by $\omega$ on $\phi$, we first have to lift $\phi$ to $s^{-1}\Omega^p_{X'}$. One possible choice is $\omega_\R\wedge\delta^\sharp\phi' $, where where $\phi'$ is any lift of $\phi$ to $\Omega^p_X$, and $\omega_\R\in \Gamma(X',\Omega^1_{X'})$ is chosen such that it coincides with the pull-back of the identity on $\R$ in an identification of $\mathring E$ with an open subset of $D\times \R$. If we then apply the morphism (\ref{equ:morph}) to $\omega_\R\wedge \delta^\sharp\phi'$, apply $s_*$, and use Equation (\ref{equ:s*}), we obtain
\begin{multline*}
-\partial\left(\sum_{\tau\in\Delta(n-1)} \langle \omega_\R \wedge\delta^\sharp (\phi'\wedge \omega) , \eta_{\tau'}\rangle [\tau'] \right)  =\\
=-s_*\left(\sum_{\tau\in\Delta(n-1)} \langle \omega_\R \wedge\delta^\sharp (\phi'\wedge \omega) , \eta_\tau\wedge n_{\tau'/s(\tau)} \rangle [\tau]\right) 
=\\=s_*\left(\sum_{\tau\in\Delta(n-1)} \langle  \phi\wedge \iota^\sharp \omega , \eta_{s(\tau)}\rangle [\tau]\right) \ ,
\end{multline*}
where $n_{\tau'/s(\tau)}$ denotes a lattice normal vector (see Remark \ref{rem:lattice normal vector and balancing}), and the last equality holds because $\langle \omega_\R, \eta_{\tau'/s(\tau)}\rangle=-1$. Thus, the image of $\omega$ in $H^{-n+1}\vdual(\iota_*\Omega^{p-1}_D)$ when moving through the square counterclockwise is represented by the section of $\Homs(\iota_*\Omega^{p-1}_D, \Delta^{-n+1}_{X})$ given by
\begin{equation*}
\phi\mapsto \sum_{\tau\in\Delta(n-1)} \langle  \phi\wedge \iota^\sharp\omega , \eta_\tau\rangle [\tau] \ .
\end{equation*}
The image of $\omega$ when moving through the square clockwise is represented by the same section of $\Homs(\iota_*\Omega^{p-1}_D, \Delta^{-n+1}_{X})$, finishing the proof.
\end{proof}

\begin{proof}[Proof of Theorem \ref{thm:computing the Verdier dual}]
Since the assertion is local we may assume that $X=\berg M \times \overline \R^k$ for a loopless matroid $M$ and some $k\in\N$. By Lemmas \ref{lem:duality for products} and \ref{lem: Verdier dual for Rn} we can further reduce to the case where $k=0$ and $X=\berg M$ for a loopless matroid $M$. We prove that any such tropical linear space satisfies Poincar\'e-Verdier duality by induction on the cardinality $\#E(M)$ of the ground set $E(M)$ of the matroid. In the base case $\#E(M)=1$, the tropical linear space $\berg M$ is a point and the assertion is trivial. We may thus assume $\#E(M)>1$. First assume that $E(M)$ is independent, in which case $M$ is uniform of maximal rank. Then the tropical linear space $\berg M$ is isomorphic to $\R^{\#E(M)-1}$ and thus satisfies Poincar\'e-Verdier duality by Lemma \ref{lem: Verdier dual for Rn}. We may thus assume that there exists $i\in E(M)$ which is not a coloop. In this case, the deletion $M\setminus i$ is a loopless matroid on the ground set $E(M\setminus i)= E(M)\setminus \{i\}$. If $i$ is a loop $M/i$, then $\berg M$ is isomorphic $\berg {M\setminus i}$ by \cite[Proposition 2.25]{ShawIntersection} and we are done by induction. If $M/i$ is loopless we can use Theorem \ref{thm:duality for linear spaces} and are also done by induction.
\end{proof}

\begin{cor}
Let $X$ be a purely $n$-dimensional tropical manifold. Then there is a natural isomorphism $\dual_X\cong \Omega^n_X[n]$.
\end{cor}

\begin{proof}
The dualizing complex $\dual_X$ is canonically isomorphic to $\vdual(\Z_X)$. Since $ \Z_X= \Omega^0_X$, we obtain
\begin{equation*}
\dual_X\cong \vdual(\Z_X)= \vdual(\Omega^0_X)\cong \Omega^n_X[n]
\end{equation*}
by applying Theorem \ref{thm:computing the Verdier dual}.
\end{proof}

\appendix
\section{Complexes of singular chains on CS sets}
\label{sec:appendix}
\renewcommand*{\thethm}{\Alph{section}.\arabic{thm}}

\begin{defn}
Let $X$ be a topological space. A \emph{stratification} of $X$ is a collection $\mS$ of disjoint locally closed subsets of $X$ such that $X=\bigcup_{S\in\mS}S$, each $S\in \mS$ is a pure-dimensional topological manifold, and such that for every $S\in \mS$ the closure $\overline S$ is a union of strata of dimension less than $\dim(S)$.
\end{defn}

Next we recall the definition of conically stratified spaces. If $L$ is a stratified space, we will use the notation $\mathring c(L)$ for the open cone over $L$, that for the space $(\Rbar\times L)/(\{\infty\}\times L)$. The open cone $\mathring c(L)$ has an induced stratification, with the cone point being the unique $0$-dimensional stratum, and all other strata being of the form $\R\times S$, where $S$ is a stratum of $L$.

\begin{defn}
Let $X$ be a topological space, equipped with a stratification $\mS$.  We say that $X$ is \emph{conically stratified}, or a \emph{CS set} for short, if for all $S\in \mS$ and  $x\in \mS$ there exist a neighborhood $U$ of $x$ in $S$, a neighborhood $V$ of $x$ in $X$, and a compact stratified space $L$ such that $V$ is homeomorphic to $U\times \mathring c(L)$ in a way respecting the stratification.
\end{defn}

\begin{defn} \label{def:admis_strat}
We say that a stratification $\mS$ of a topological space $X$ is \emph{admissible}, if the stratified space $(X,\mS)$ is conically stratified and for every stratum $S\in\mS$ the pair $(\overline S,S)$ is homeomorphic to a pair $(U,\mathring D^n)$, where $\mathring D^n$ is the open unit disc in $\R^n$ and $U$ is an open subset of the closed unit disc $D^n$ that contains $\mathring D^n$.
\end{defn}

\begin{example}
If $X$ is a rational polyhedral space with a face structure $\Sigma$, then the relative interiors of the polyhedra in $\Sigma$ stratify $X$, and this stratification is admissible.
\end{example}

Let $X$ be a topological space equipped with a stratification $\mS$. Exactly as for face structures (see \S\ref{subsec:Comparison of homologies}), we say that a singular simplex $\sigma\colon \Delta^q\to X$ (where $\Delta^q$ denotes the standard $q$-simplex) respects the stratification $\mS$ if the relative interior of any face of $\Delta^q$ is mapped into a stratum of $\mS$. For every open set $U\subseteq X$ and $i\in \Z$ we denote by $C_i(U;\mS)$ the free abelian group on all singular $i$-simplices in $U$ respecting the stratification $\mS$. Since faces of simplices respecting the stratification respect the stratification again, we obtain a chain complex $C_\bullet(U,\mS)$, and a quotient 
\begin{equation} \label{eq:relchains}
C_\bullet(X,U;\mS)= C_\bullet(X;\mS)/C_\bullet(U;\mS)
\end{equation} 
of relative chains that respect the stratification. We denote the $i$-th homology of these complexes by
\begin{align*}
H_i(U;\mS)&= H_i(C_\bullet(U;\mS)) \ , \text{ and} \\
H_i(X,U;\mS)&= H_i(C_\bullet (X,U;\mS)) \ .
\end{align*}
 For every $k$ we denote by $\Delta_X^{\mS,-k}$ the sheafification of the presheaf $U\mapsto C_k(X,X\setminus \overline U;\mS)$. The differentials on the complexes of relative chains that respect the stratification induce a differential that makes $\Delta_X^{\mS,\bullet}$ a cochain complex. By definition, $\Delta_X^{\mS,\bullet}$ is a subcomplex of $\Delta_X^\bullet$.

\begin{prop}
\label{prop:dualizing complex using singular cycles respecting the stratification}
Let $X$ be a conically stratified space with stratification $\mathcal S$. Then the inclusion map
\[
\Delta_X^{\mS,\bullet} \to \Delta_X^\bullet
\]
is a quasi-isomorphism.
\end{prop}

\begin{proof}
For the purpose of this proof we will denote 
\[
H^{\mathrm{strat}}_*(X)= H_*(X;\mS)
\]
for a conically stratified space $X$ with stratification $\mS$.

We need to show that $H^{-i}(\Delta_X^{\mS,\bullet})\to H^{-i} (\Delta_X^\bullet)$ is an isomorphism of sheaves for all $i\in \Z$. At a point $x\in X$, the stalks of these sheaves are $H_i(X, X\setminus \{x\};\mS)$ and $H_i(X,X\setminus \{x\})$, respectively, so using the long exact sequence for relative homology and the five lemma, it suffices to show that the natural morphisms
\[
H^{\mathrm{strat}}_*(U) \to H_*(U)
\]
are isomorphisms for all open subsets $U\subseteq X$. This follows from \cite[Theorem 5.1.4]{FriedmanBook} once we show that the four hypothesis of the theorem are satisfied. 

\begin{enumerate}
\item
Since the barycentric subdivision restricts to an equivalence of complexes $C_\bullet(X;\mS)\to C_\bullet(X;\mS)$, there are compatible Mayer-Vietoris sequences for  $H^\mathrm{strat}_*$ and $H_*$.

\item 
If $\{U_\alpha\}$ is an increasing collection of open subsets of a CS set $X$ such that 
\[
H^\mathrm{strat}_*(U_\alpha)\to H_*(U_\alpha)
\]
is an isomorphism for each $\alpha$, then 
\[
H^\mathrm{strat}_*\left(\bigcup_\alpha U_\alpha\right) \to H_*\left(\bigcup_\alpha U_\alpha\right)
\]
is also an isomorphism because
\begin{align*}
H^\mathrm{strat}_*\left(\bigcup_\alpha U_\alpha\right) & \cong \injlim_\alpha H^\mathrm{strat}_*(U_\alpha)  \text{ and} \\
 H_*\left(\bigcup_\alpha U_\alpha\right)  &\cong \injlim_\alpha H_*(U_\alpha) \ .
\end{align*}

\item The statement is true if $X$ is a point. It is also true if $X$ is homeomorphic to $\R^n\times \mathring c(L)$  in a way respecting the stratification for some $n\in \N$ and some CS set $L$, because in this case $X$ can be contracted to a point in a way that respects the stratification, reducing to the case where $X$ is a point. 

\item If $X$ only has a single stratum, then $C^\mathrm{strat}_\bullet(X)=C_\bullet(X)$ and therefore the statement is true for $X$.
\end{enumerate}
 
\end{proof}

\begin{prop}
\label{prop:restriction yields equivalence}
Let $X$ be a conically stratified space with stratification $\mS$, and let $U\subseteq X$ be an open subset. Then the inclusion
\[
\Delta_U^{U\cap \mS,\bullet}\to \Delta_X^{\mS,\bullet}|_U \ ,
\]
where $U\cap \mS$ is the induced stratification on $U$, is an equivalence of complexes
\end{prop}

\begin{proof}
For every open subset $V\subseteq U$ with $\overline V\subseteq U$ the inclusion
\[
C_\bullet(U, U\setminus \overline V; \mS)\to C_\bullet(X, X\setminus \overline V;\mS)
\]
is an equivalence by the excision theorem. To show that this stays an equivalence when sheafifying, we need to make sure that the homotopy inverses are compatible with restrictions. Let $S\colon C_\bullet(X;\mS)\to C_\bullet(X;\mS)$ be the barycentric subdivision and let $T\colon \id\Rightarrow S$ be a (functorial) chain homotopy between the identity and $S$. We use $T$ to define a new chain homotopy $T_U$ whose action on singular $n$-simplices is given by
\[
T_U\sigma =\begin{cases}  
0, &\text{if }\sigma(\Delta^n) \subseteq U, \\
T\sigma, &\text{else.}
\end{cases}
\]
By the functoriality of $T$, this induces morphisms
\[
C_k(X, X\setminus \overline V ; \mS)\to C_{k+1}(X, X\setminus \overline V;\mS)
\]
for all $k\in \N$ and open subsets $V\subseteq X$, which are compatible with restrictions. Therefore, they induce morphism $\Delta^{\mS,k}_X\to \Delta^{\mS,k-1}_X$ of sheaves for all $k\in \Z$, whose restrictions to $U$ we denote by $\psi^k \colon \Delta^{\mS,k}_X|_U\to \Delta^{\mS,k-1}_X|_U$. Let $f=\id - (d\psi+\psi d)$. We claim that the sequence $(f^i)_i$ of powers of $f$ converges to a morphism $f^\infty$ in the sense that for every section $s\in \Gamma(V, \Delta^{\mS,k}_X)$, where $V\subseteq U$ is open, there exists a covering $V= \bigcup _\lambda V_\lambda$ such that the sequence $f^i(s)|_{V_\lambda}$ converges to $f^\infty(s)|_{V_\lambda}$ in the discrete topology on $\Gamma(V_\lambda,\Delta^{\mS,k}_X)$ (that is the sequence is eventually constant). As the open subsets $V\subset U$ with $\overline V\subseteq U$ form a basis for the topology of $U$, it suffices to show that the sequence $(f^i(s))_i$ is eventually constant for every $s\in C_{-k}(X, X\setminus \overline V;\mS)$ for such an open subset $V$ of $U$. By linearity, we may even assume that $s$ is represented by a single singular $(-k)$-simplex $\sigma$. We have
\begin{multline*}
f(\sigma)= \sigma- (\partial\psi+\psi\partial)\sigma = \sigma -(\partial T+ T\partial)\sigma + (\partial (T-T_U)+ (T-T_U)\partial )\sigma =\\
=  S(\sigma) +(\partial (T-T_U)+ (T-T_U)\partial )\sigma\ .
\end{multline*}
Note that for any singular $j$-simplex $\delta$ we have
\[
(T-T_U)(\delta)= \begin{cases}
T\delta, & \text{if } \delta(\Delta^j)\subseteq U \\
0, & \text{else.}
\end{cases}
\]
In particular, $(T-T_U)c$ is a linear combination of simplices contained in $U$ for every chain $c$. Therefore, $f(\sigma)-S\sigma$ is represented by a linear combination of singular simplices that are contained in $U$. As $f$ is the identity on simplices that are contained in $U$, we see inductively that $f^i(\sigma)-S^i\sigma$ is represented by a linear combination of simplices that are contained in $U$ for all $i\in\N$. Since we assumed that $\overline V\subset U$, the space $X$ is the union of $U$ and $X\setminus \overline V$. Thus, for $i$ large enough the chain $S^i\sigma$ will be represented by a linear combination of simplices that are either contained in $U$ or in $X\setminus \overline V$. The latter are $0$ in $C_{-k}(X,X\setminus \overline V;\mS)$, so $f^i(\sigma)$ is represented by a linear combination of simplices that are contained in $U$. It follows immediately that the sequence $(f^i(\sigma))_i$ is eventually constant, and we conclude that $f^\infty$ is well-defined. It also follows that $f^\infty (\sigma)$ is contained in $C_{-k}(U, U\setminus \overline V;\mS)$. This shows that $f^\infty$ maps into the subcomplex $\Delta_U^{U\cap\mS,\bullet}$ of $\Delta^{\mS,\bullet}_X|_U$.

It is clear from the definitions that $f^\infty$ respects the differentials, so we have constructed a morphism $\Delta^{\mS,\bullet}_X|_U\to \Delta^{U\cap \mS,\bullet}_U$. If $\iota\colon \Delta^{U\cap\mS,\bullet}_U\to \Delta^{\mS,\bullet}_X|_U$ denotes the inclusion, then $f^\infty\circ \iota=\id$ by construction. The construction of $f^\infty$ also provides a chain homotopy $\id \Rightarrow \iota\circ f^\infty$, namely the limit
\[
\sum_{i=0}^\infty \psi f^i \ .
\]
This converges in the same sense as before, because once $f^i(s)$ in $\Delta^{U\cap\mS,\bullet}_U$ we have $\psi(f^i(s))=0$ by definition. We conclude that $\iota$ is an equivalence of complexes.
\end{proof}

\begin{prop}
\label{prop: Hom is equal to derived hom for locally closed sets}
Let $X$ be a conically stratified space with stratification $\mS$, and let $A\subset X$ be a locally closed subset of $X$ that is a union of strata. Then the morphism
\[
\Hom^\bullet(\Z_A, \Delta^{\mS,\bullet}_X)\to R\Hom^\bullet(\Z_A, \dual_X)
\]
that is induced by the inclusion $\Delta_X^{\mS,\bullet}\to \Delta_X^\bullet$ and the natural identification $\Delta_X^\bullet\cong \dual_X$, is an isomorphism in $\der{}$.

In particular, the natural morphism
\[
\Homs^\bullet(\Z_A, \Delta^{\mS,\bullet}_X)\to R\Homs^\bullet(\Z_A, \dual_X)
\]
is an isomorphism in $\der X$.
\end{prop}

\begin{proof}
By Proposition \ref{prop:dualizing complex using singular cycles respecting the stratification}, we need to show that if
$\Delta^{\mS,\bullet}_X\to\mathcal I^\bullet$ is an injective resolution, then the induced morphism
\[
\Hom^\bullet(\Z_A,\Delta^{\mS,\bullet}_X)\to \Hom^\bullet(\Z_A,\mathcal I^\bullet)
\]
is a quasi-isomorphism. 
Let $U\subseteq X$ be an open subset such that $A$ is closed in $U$.  Then the morphism above is a quasi-isomorphism if and only if the morphism
\[
\Hom^\bullet(\Z_A, \Delta^{\mS,\bullet}_X|_U)\to \Hom^\bullet(\Z_A,\mathcal I^\bullet|_U)
\]
is a quasi-isomorphism. Since the natural morphism $\Delta^{U\cap\mS,\bullet}_U\to \Delta^{\mS,\bullet}_X|_U$ is a chain equivalence by \ref{prop:restriction yields equivalence}, the induced morphism
\[
\Hom^\bullet(\Z_A,\Delta^{U\cap \mS,\bullet}_U)\to \Hom^\bullet(\Z_A,\Delta^{\mS,\bullet}_X|_U)
\]
is a quasi-isomorphism as well, so it suffices to show that the morphism
\[
\Hom^\bullet(\Z_A,\Delta^{U\cap \mS,\bullet}_U)\to \Hom^\bullet(\Z_A,\mathcal I^\bullet|_U)
\]
that is induced by the composite $\Delta^{U\cap \mS,\bullet}_U\to \Delta^{\mS,\bullet}_X|_U\to \mathcal I^\bullet|_U$ is a quasi-isomorphism. As this composite is an injective resolution of $\Delta^{U\cap \mS,\bullet}_U$ we may replace $X$ by $U$ and assume that $A$ is closed in $X$. 

Let $i\colon A\to X$ be the inclusion. Since $i^!$ is right-adjoint to $i_*$, it suffices to show that the morphism
\[
\Gamma_A(X, \Delta^{\mS,\bullet}_X)=\Hom^\bullet\left(\Z_A, i^!\left(\Delta^{\mS,\bullet}_X\right)\right)\to \Hom^\bullet\left(\Z_A, i^!\mathcal I^\bullet\right)=\Gamma(A, i^! \mathcal I^\bullet)
\]
is a  quasi-isomorphism. The natural morphism $\Delta^{A\cap \mS,\bullet}_A\to i^!(\Delta^{\mS,\bullet}_X)$ defined by push-forwards of chains along $i$ defines an isomorphism
\[
\Gamma(A,\Delta^{A\cap \mS,\bullet}_A)\cong\Gamma_A(X,\Delta^{\mS,\bullet}_X)
\]
on global sections because both sides are the chain complexes of locally finite chains in $A$ that respect the stratification. It thus suffices to show that the morphism
\begin{equation}
\label{equ:morphism of restrictions}
\Gamma(A,\Delta^{A\cap \mS,\bullet}_A)\to \Gamma(A, i^!\mathcal I^\bullet)
\end{equation}
induced by the composite $\Delta^{A\cap \mS,\bullet}_A\to i^!\Delta^{\mS,\bullet}_X\to i^!\mathcal I^\bullet$ is a quasi-isomorphism. 
By construction, there is a commutative diagram
\begin{center}
\begin{tikzpicture}[auto]
\matrix[matrix of math nodes, row sep= 5ex, column sep= 4em, text height=1.5ex, text depth= .25ex]{
|(DelA)| \Delta^{A\cap \mS,\bullet}_A 	&
|(iI)| i^!\mathcal I^\bullet 
\\
|(DA)| \dual_A &
|(iD)| i^!\dual_X 
\\
};
\begin{scope}[->,font=\footnotesize]
\draw (DelA) -- (iI);
\draw (DA) -- (iD);
\draw (DelA)--node{$\cong$}(DA);
\draw (iI) --node{$\cong$}(iD);
\end{scope}
\end{tikzpicture}
\end{center}
in $\der A$ whose vertical arrows are isomorphism. Since we defined the upper horizontal morphism via the push-forward of singular cycles, which is, of course, compatible with the trace morphisms, the lower horizontal morphism is the natural isomorphism $\dual_A\cong i^!\dual_X$. We conclude that the upper horizontal morphism is an isomorphism in $\der A$ as well. Because $i^!$ is the right-adjoint of the exact functor $i_*$, the upper horizontal morphism in the diagram is in fact an injective resolution. 
That the morphism displayed in (\ref{equ:morphism of restrictions}) is a quasi-isomorphism now follows from the fact that $\Delta^{A\cap \mS,\bullet}_A$ is homotopically fine and from \cite[IV Theorem 2.2]{Bredon}, finishing the proof of the main statement.

For the ``in particular'' statement we note that
\[
\Hom^\bullet(\Z_{U\cap A}, \Delta^{U\cap\mS,\bullet}_U)\to R\Hom^\bullet(\Z_{U\cap A}, \dual_U)
\]
is a quasi-isomorphism for every open subset $U\subseteq X$ by the main statement. Together with Proposition \ref{prop:restriction yields equivalence} we see that 
\[
\Hom^\bullet(\Z_{A}|_U, \Delta^{\mS,\bullet}_X|_U)\to R\Hom^\bullet(\Z_A|_U, \dual_X|_U)
\]
is a quasi-isomorphism for all open subsets $U$ of $X$, which directly implies the claim.
\end{proof}

\begin{lemma}
\label{lem:Sections of dual}
Let $Y$ be a subset of the closed unit disc $D^n\subseteq \R^n$ such that its intersection $Z=Y\cap \mathring D^n$ with the open unit disc $\mathring D^n$ is nonempty and connected. Furthermore,  let $\mF$ be sheaf of abelian groups on $D^n$ such that the restriction $\mF|_{\mathring D^n}$ is constant. Then the restriction maps
\begin{align*}
\Hom(\mF|_{\mathring D^n}, \Z_{\mathring D^n}) &\to \Hom(\mF|_Z,\Z_Z) \ , \text{ and} \\
\Hom(\mF|_{Y}, \Z_{Y}) &\to \Hom(\mF|_Z,\Z_Z) 
\end{align*}
are isomorphisms.
\end{lemma}

\begin{proof}
The fact that 
\[
\Hom(\mF|_{\mathring D^n}, \Z_{\mathring D^n}) \to \Hom(\mF|_Z,\Z_Z) \\
\]
is an isomorphism follows immediately from the fact that $\mF|_{\mathring D^n}$ is constant and $Z$ is connected. For the second map we consider the short exact sequence
\[
0\to \mF_Z \to \mF_{Y} \to \mF_{Y\setminus Z} \to 0 
\]
of sheaves on $Y$. Applying $\Hom(-,\Z_{Y})$ we obtain an exact sequence
\[
0\to \Hom(\mF_{Y\setminus Z},\Z_{Y})\to \Hom(\mF_{Y}, \Z_{Y}) \to \Hom(\mF_Z,\Z_{Y})\to \Ext^1(\mF_{Y\setminus Z},\Z_{Y}) \ .	
\]
As $\Z_{Y}$ does not have any sections supported on a proper closed subset of $Y$, we have
\[
\Hom(\mF_{Y\setminus Z},\Z_{ Y}) = 0 \ .
\]
Furthermore, there is a natural isomorphism 
\[
\Hom(\mF_Z,\Z_{ Y}) \cong \Hom(\mF|_Z,\Z_Z) \ .
\]
We conclude that the restriction 
\[
\Hom(\mF|_{Y}, \Z_{ Y}) \to \Hom(\mF|_Z,\Z_Z)
\]
is injective. To show that it is surjective as well we consider the commutative square
\begin{center}
\begin{tikzpicture}[auto]
\matrix[matrix of math nodes, row sep= 5ex, column sep= 4em, text height=1.5ex, text depth= .25ex]{
|(Dbar)| \Hom(\mF,\Z_{D^n}) 	&
|(D)| \Hom(\mF|_{\mathring D^n}, \Z_{\mathring D^n})	
\\
|(Ybar)| \Hom(\mF|_{Y},\Z_{Y}) &
|(Y)| \Hom(\mF|_Z,\Z_Z)  \ .
\\
};
\begin{scope}[->,font=\footnotesize]
\draw (Dbar) -- (D);
\draw (Ybar) -- (Y);
\draw (Dbar)--(Ybar);
\draw (D) --node{$\cong$}(Y);
\end{scope}
\end{tikzpicture}
\end{center}
From the discussion above we know that the vertical arrow on the right is an isomorphism, and that the horizontal arrows are injective (set $Y=D^n$ in the discussion above for the top arrow). So to finish the proof, it suffices to show that the top horizontal arrow is surjective. In other words, it suffices to prove the result for $Y= D^n$, that is to show the surjectivity of
\[
\Hom(\mF,\Z_{D^n}) 	\to
\Hom(\mF|_{\mathring D^n}, \Z_{\mathring D^n})	\ .
\]
In the exact sequence from above we can see that this map is surjective if and only if $\Ext^1(\mF_{S^{n-1}},\Z_{D^n})=0$, where
$S^{n-1}=D^n\setminus \mathring D^n$. Let $i\colon S^{n-1}\to D^n$ be the inclusion. By Verdier duality for $i$ we see that
\[
\Ext^1(\mF_{S^{n-1}},\Z_{D^n})=\Ext^1(\mF|_{S^{n-1}}, i^! \Z_{D^n}) \ .
\]
For $k\in \N$, the  $k$-th cohomology sheaf $H^k(i^! \Z_{D^n})$ is the restriction to $S^{n-1}$ of the sheaf associated to the presheaf
\[
U \to H^k(U, U\cap \mathring D^n) 
\] 
on $D^n$ (cf.\ \cite[pp.\ 14-15]{LocalCohomology} for a closely related example). As the pair $(D^n,\mathring D^n)$ is locally homeomorphic to the pair given by an open half-space in $\R^n$ and its closure, these sheaves are all zero and hence $i^!\Z_{D^n}=0$, finishing the proof.
\end{proof}

\begin{prop}
\label{prop:Hom is RHom}
Let $X$ be a conically stratified space with an admissible stratification $\mS$, and let $\mF$ be a sheaf of abelian groups on $X$  such that $\mF|_{S}$ is locally free of finite rank for every stratum $S\in \mS$. Then the natural morphism
\[
\Homs^\bullet(\mF, \Delta_X^{\mS}) \to R\Homs^\bullet (\mF,\dual_X)
\]
is an isomorphism in $\der X$. In particular, the natural morphism
\[
\Hom^\bullet(\mF, \Delta_X^{\mS}) \to R\Hom^\bullet (\mF,\dual_X)
\]
is an isomorphism in $\der{}$. 
\end{prop}

\begin{proof}
The statement is local on $X$, so we may assume that the stratification $\mS$ is finite. We do induction on the number of strata on which $\mF$ is nontrivial. If this number is $0$, then $\mF=0$ and the statement is trivial. So let us assume there is a stratum on which $\mF$ is nontrivial, and let $S\in\mS$ be maximal with that property. The stratum $S$ is an open subset of the support $\supp(\mF)$, so if $A=\supp(\mF)\setminus S$ we obtain an exact sequence
\[
0\to \mF_S \to \mF \to \mF_A \to 0 \ .
\]	
We can use this to obtain a commutative diagram
\begin{center}
\begin{tikzpicture}[auto]
\matrix[matrix of math nodes, row sep= 5ex, column sep= 2em, text height=1.5ex, text depth= .25ex]{
|(FA)| \Homs^\bullet(\mF_A, \Delta_X^{\mS})	&
|(F)| \Homs^\bullet(\mF, \Delta_X^{\mS})	&
|(FU)| \Homs^\bullet(\mF_S, \Delta_X^{\mS}) &
\\
|(RFA)| R\Homs^\bullet(\mF_A, \Delta_X^{\mS})	&
|(RF)| R\Homs^\bullet(\mF, \Delta_X^{\mS})	&
|(RFU)| R\Homs^\bullet(\mF_S, \Delta_X^{\mS}) &
|(RFA1)| 
\\
};
\begin{scope}[->,font=\footnotesize]
\draw (FA) -- (F);
\draw (F) -- (FU);
\draw (RFA) -- (RF);
\draw (RF) -- (RFU);
\draw (RFU) -- (RFA1);
\draw (FA) -- (RFA);
\draw (F) -- (RF);
\draw (FU) --(RFU);
\end{scope}
\end{tikzpicture}
\end{center}
where the lower row is an exact triangle. Note that the vertical arrow on the left is an isomorphism by the induction hypothesis.  Since $\mF_S$ is locally free and $S$ is simply connected, the sheaf $\mF_S$ is isomorphic to a finite sum of several copies of $\Z_S$. So by Proposition \ref{prop: Hom is equal to derived hom for locally closed sets}, the right arrow is an isomorphism as well. If we can show that the morphisms of complexes in the upper row of the diagram defines a short exact sequence in every degree, the statement follows from the five lemma.
We recall from the proof of Theorem \ref{thm:compatibility of homology theories} that for every $i\in\Z$ there is an isomorphism  
\[
\Delta^{\mS,-i}_X \cong \bigoplus_{\sigma}\Z_{\sigma(\Delta^i)} \ ,
\]
where the direct sum is over all singular $i$-simplices $\sigma\colon \Delta^i\to X$ respecting the stratification. Consequentially, for every sheaf of Abelian groups $\mG$ that is locally free of finite rank when restricted to any stratum in $\mS$, we have
\[
\Homs(\mG, \Delta^{\mS, -i}_X) = \Homs\left(\mG,\bigoplus_{\sigma}\Z_{\sigma(\Delta^i)} \right) = \bigoplus_\sigma  \Homs(\mG, \Z_{\sigma(\Delta^i)}) \ ,
\]
where the last equality holds because $\mG$ is constructible. If for a simplex $\sigma$ appearing in the direct sum we denote by $T_\sigma\in\mS$ the unique stratum into which the relative interior of $\Delta^i$ maps, there is an isomorphism
\[
\Homs(\mG, \Z_{\sigma(\Delta^i)}) \cong (\Hom(\mG|_{T_\sigma},\Z_{T_\sigma}))_{\sigma(\Delta^i)}
\]
induced by restricting sections by Lemma \ref{lem:Sections of dual}.
So to finish the proof it suffices to show that for every stratum $T\in \mS$ the sequence
\[
0\to \Hom((\mF_A)|_T,\Z_T) \to \Hom(\mF|_T,\Z_T) \to \Hom((\mF_S)|_T,\Z_T) \to 0
\]
is exact. If $T=S$, this is the case because the second morphism is an isomorphism and the first group is trivial, whereas if $T\neq S$ this is the case because the first morphism is an isomorphism and the last group is trivial.

For the ``in particular'' statement we apply $R\Gamma$ to the isomorphism 
\[
\Homs^\bullet(\mF, \Delta_X^{\mS,\bullet}) \to R\Homs^\bullet (\mF,\dual_X)
\]
and note that the natural morphism
\[
\Hom^\bullet(\mF,\Delta_X^{\mS,\bullet})=\Gamma(\Homs^\bullet(\mF,\Delta_X^{\mS,\bullet})) \to R\Gamma\Homs^\bullet(\mF,\Delta_X^{\mS,\bullet}) 
\]
is an isomorphism because $\Delta_X^{\mS,\bullet}$, and hence $\Homs^\bullet(\mF,\Delta_X^{\mS,\bullet})$, is homotopically fine \cite[IV Theorem 2.2]{Bredon}.
\end{proof}

\input{sheafTrop.bbl}


\end{document}